\documentclass[11pt,reqno]{amsart}

\usepackage{amsmath,amssymb,amsthm,mathtools,dsfont,bm,bbm} 

\newtheorem{lem}{Lemma}[section]
\newtheorem{prop}{Proposition}[section]
\newtheorem{cor}{Corollary}[section]

\newtheorem{thm}{Theorem}[section]

\theoremstyle{definition}

\theoremstyle{remark}

\theoremstyle{remark}
\newtheorem{remark}{Remark}[section]
\newtheorem*{remarks*}{Remarks}
\newtheorem*{remark*}{Remark}
\numberwithin{equation}{section}

\newcommand{\Hfr}{\mathfrak{H}}
\newcommand{\C}{\mathbb{C}}
\newcommand{\N}{\mathbb{N}}
\newcommand{\Z}{\mathbb{Z}}
\newcommand{\R}{\mathbb{R}}
\newcommand{\T}{\mathbb{T}}

\newcommand{\UU}{\mathcal{U}}
\newcommand{\Ss}{\mathbb{S}}
\newcommand{\VV}{\mathcal{V}}
\newcommand{\EE}{\mathcal{E}}
\newcommand{\ov}{\overline}
\newcommand{\ran}{\mathrm{ran}}
\renewcommand{\ker}{\mathrm{ker}}
\newcommand{\dom}{\mathrm{dom}}
\newcommand{\rank}{\mathrm{rank}}
\newcommand{\sv}{\mathbf{s}}
\newcommand{\nv}{\mathbf{n}}
\newcommand{\Qv}{\mathbf{Q}}
\newcommand{\qv}{\mathbf{q}}
\newcommand{\pt}{\partial}

\newcommand{\be}{\begin{equation}}
\newcommand{\ee}{\end{equation}}

\newcommand{\eps}{\varepsilon}

\newcommand{\Ms}{\mathsf{M}}

\newcommand{\weakto}{\rightharpoonup}
\newcommand{\ii}{\mathrm{i}}
\newcommand{\eu}{\mathrm{e}}

\renewcommand{\rho}{\varrho}

\newcommand{\mv}{\mathbf{m}}

\renewcommand{\Im}{\mathrm{Im}}

\newcommand{\Hil}{\mathsf{H}}
\newcommand{\uu}{\mathbf{u}}
\newcommand{\vv}{\mathbf{v}}

\newcommand{\Ab}{\mathbf{A}}
\newcommand{\Bb}{\mathbf{B}}
\newcommand{\Vb}{\mathbf{V}}
\newcommand{\Wb}{\mathbf{W}}
\newcommand{\Ub}{\mathbf{U}}
\newcommand{\Fb}{\mathbf{F}}
\newcommand{\Pb}{\mathbf{P}}
\newcommand{\Gb}{\mathbf{G}}
\newcommand{\id}{\mathrm{Id}}
\newcommand{\Tr}{\mathrm{Tr}}
\newcommand{\Ds}{|D|}
\newcommand{\Dsr}{\langle D \rangle}
\newcommand{\Gr}{\mathsf{Gr}}

\newcommand{\Ts}{\mathsf{T}}
\newcommand{\Tst}{\widetilde{\mathsf{T}}}
\newcommand{\Gs}{\mathsf{G}}

\newcommand{\Gst}{\widetilde{\Gs}}
\newcommand{\Mst}{\widetilde{\Ms}}
\renewcommand{\phi}{\varphi}

\newcommand{\comment}[1]{}

\catcode`@=11
\def\section{\@startsection{section}{1}%
  \z@{1.5\linespacing\@plus\linespacing}{.5\linespacing}%
  {\normalfont\bfseries\large\centering}}
\catcode`@=12

\makeatletter
\def\@cite#1#2{[\textbf{#1\if@tempswa , #2\fi}]}
\def\@biblabel#1{[\textbf{#1}]}
\makeatother

\begin{document}
\title[GWP and Soliton Resolution for HWM with Rational Data]{Global Well-Posedness and Soliton Resolution for the Half-Wave Maps Equation with Rational Data}

\author{Patrick G\'erard}
\address{P. G\'erard,  Laboratoire de Math\'ematiques d'Orsay, CNRS, Universit\'e Paris-Saclay, 91405 Orsay, France.}%
\email{patrick.gerard@universite-paris-saclay.fr}

\author{Enno Lenzmann}
\address{E. Lenzmann, University of Basel, Department of Mathematics and Computer Science, Spiegelgasse 1, CH-4051 Basel, Switzerland.}%
\email{enno.lenzmann@unibas.ch}

\maketitle

\begin{abstract}
We study the energy-critical half-wave maps equation:
\[
\partial_t \uu = \uu \times |D| \uu 
\]
for $\uu : [0, T) \times \mathbb{R} \to \Ss^2$. Our main result establishes the global existence and uniqueness of solutions for all rational initial data $\uu_0 : \mathbb{R} \to \Ss^2$. This demonstrates global well-posedness for a dense subset within the scaling-critical energy space \( \dot{H}^{1/2}(\mathbb{R}; \mathbb{S}^2) \). Furthermore, we prove soliton resolution for a dense subset of initial data in the energy space with uniform bounds for all higher Sobolev norms $\dot{H}^s$ for $s > 0$. 

Our analysis utilizes the Lax pair structure of the half-wave maps equation on Hardy spaces in combination with an explicit flow formula. Extending these results, we establish global well-posedness for rational initial data (along with a soliton resolution result) for a generalized class of matrix-valued half-wave maps equations with target spaces in the complex Grassmannians \( \Gr_k(\mathbb{C}^d) \). Notably, this includes the complex projective spaces \( \mathbb{CP}^{d-1} \cong \Gr_1(\mathbb{C}^d) \) thereby extending the classical case of the target \( \Ss^2 \cong \mathbb{CP}^1 \).
\end{abstract}

%\begin{abstract}
%We consider the energy-critical half-wave maps equation
%$$
%\pt_t \uu = \uu \times |D| \uu
%$$
%for $\uu : [0,T) \times \R \to \Ss^2$. We show that all rational initial data $\uu_0 : \R \to \Ss^2$ lead to unique global-in-time solutions. By this result, we obtain global well-posedness for a dense subset in the scaling-critical energy space $\dot{H}^{\frac 1 2}(\R; \Ss^2)$. Furthermore, we prove that soliton resolution occurs for a dense subset of initial data in the energy space together with uniform bounds on all  higher Sobolev norms $\dot{H}^s$ with any $s > 0$. Our proofs make essential use of the Lax pair structure on Hardy spaces combined with an explicit flow formula. 
%
%More generally, we prove global well-posedness for rational initial data for a generalized class of matrix-valued half-wave maps equations with values in the complex Grassmannians $\Gr_k(\C^d)$. In particular, our results include the target spaces given by the complex projective spaces $\mathbb{CP}^{d-1} \cong \Gr_1(\C^d)$ for all $d \geq 2$, extending the classical case of the target $\Ss^2 \cong \mathbb{CP}^1$.  
%\end{abstract}

\tableofcontents

%%%%%%%%%%%%%%%%%%%%%%%%%%%%%%
\section{Introduction and Main Results} \label{sec:intro}

This paper is devoted to the half-wave maps equation posed on the real line, which reads
\be \label{eq:HWM} \tag{HWM}
\pt_t \uu = \uu \times |D| \uu
\ee
with $\uu : [0,T) \times \R \to \Ss^2$. Here $\Ss^2$ denotes the standard unit two-sphere embedded in $\R^3$ and $\times$ stands for the vector/cross product in $\R^3$. Formally, the operator $|D|$ is given by its Fourier multiplier $|\xi|$ corresponding to the half-Laplacian $|D|=\sqrt{-\Delta}$ on $\R$. Equivalently in our setting, we can write $|D| = \Hil \pt_x$ where 
\be
(\Hil f)(x) = \frac{1}{\pi} \mathrm{p.v.} \int_\R \frac{f(y)}{x-y} \, dy
\ee
denotes the Hilbert transform on the real line. The main physical motivation for studying (HWM) stems from the fact that it can be seen as a continuum version of discrete completely integrable so-called {\em spin Calogero--Moser models}; see \cite{ZhSt-15, LeSo-24}. See also \cite{LeSc-18} for a complete classification of traveling solitary waves of (HWM) in relation to (non-free) minimal surfaces of disk type, as well as the studies \cite{BeKlLa-20, Ma-22} of the dynamics of rational solutions of (HWM) in the applied math literature.

\medskip
As shown in \cite{GeLe-18}, the half-wave maps equation is a completely integrable Hamiltonian PDE in the sense of having a Lax pair structure that yields an infinite set of conserved quantities and that also shows that rationality is preserved by the flow of (HWM). We remark that its Hamiltonian energy functional is easily found to be
\be
E(\uu) = \frac{1}{2} \int_{\R} \uu \cdot  |D| \uu  \, dx = \frac{1}{4 \pi} \int_{\R} \! \int_{\R} \frac{|\uu(x)-\uu(y)|^2}{|x-y|^2} \, dx \, dy \, .
\ee
Note that the scaling $\uu(t,x) \mapsto \uu(\lambda t, \lambda x)$ with some constant $\lambda > 0$ preserves solutions of \eqref{eq:HWM} as well as the energy $E(\uu)$. Thus we see that \eqref{eq:HWM} is {\em energy-critical}.

However, the question of existence (or non-existence) of global-in-time solutions for (HWM) -- even for smooth and sufficiently small data -- has been left completely open so far. Here one of the major obstacles lies in the non-dispersive nature of the half-wave operator $|D|$ in one space dimension occurring in the quasi-linear evolution problem \eqref{eq:HWM}. In fact, this situation prevents us from adapting known tools developed to prove global well-posedness results for other dispersive geometric PDEs such as the Schr\"odinger maps and wave maps equations; see e.\,g.~\cite{Ta-06,KotaVi-14} and references therein. We refer also to \cite{KrSi-18, KiKr-21, Li-23} for small data global existence for \eqref{eq:HWM} in the non-integrable case with space dimensions at least $N \geq 3$, where dispersive estimates can be used which are not available for our setting here. 

\medskip 
In the present paper, we shall develop an entirely different approach that will lead to global well-posedness for all rational initial data, which are shown to form a dense subset in the scaling-critical energy space $\dot{H}^{\frac 1 2}(\R; \Ss^2)$ for (HWM). Our proof will be based on the Lax pair structure on suitable Hardy spaces together with an explicit flow formula for (HWM) akin to the explicit formulae recently found by the first author for the Benjamin--Ono equation. Furthermore, as a byproduct of our analysis, we will also study the long-time behavior of rational solutions, leading to a general result on {\em soliton resolution} in this setting. In particular, this result  yields a rigorous proof of the numerical findings for (HWM) that have been recently presented in \cite{BeKlLa-20}.

\subsection*{Global Well-Posedness for Rational Data}
We consider (HWM) with initial data that are given by rational functions. To this end, we define the set
$$
\mathcal{R}at(\R; \Ss^2) := \left \{ \uu : \R \to \Ss^2 \mid \mbox{$\uu(x)$ is rational in $x \in \R$}  \right \} \, .
$$
Explicit examples of rational maps $\uu : \R \to \Ss^2$ are easily constructed by means of the (inverse) stereographic projection from the extended complex plane $\C \cup \{ \infty \}$ to $\Ss^2$; see Section \ref{sec:target_S2} below for details.

In fact, rational maps play a distinguished role in the analysis of (HWM), as they occur in the complete classification of traveling solitary waves. Furthermore, due to its Lax pair structure (detailed below) and a Kronecker-type theorem for Hankel operators (see Section \ref{sec:spectral} below), another essential feature of (HWM) is that rationality is preserved by the flow, see \cite{GeLe-18}. For any $\uu \in \mathcal{R}at(\R; \Ss^2)$, we readily verify the following properties:
\begin{itemize}
\item Limit: $\displaystyle \lim_{x \to \pm \infty} \uu(x) = \mathbf{p}$ for some unit vector $\mathbf{p} \in \Ss^2$.
\item Smoothness: $\uu \in \dot{H}^\infty = \bigcap_{s >0} \dot{H}^s$. 
\end{itemize}
In addition, we can derive with the following nontrivial fact.

\begin{thm}\label{thm:dense_rational}
$\mathcal{R}at(\R; \Ss^2)$ is a dense subset in $\dot{H}^{\frac 1 2}(\R; \Ss^2)$.
\end{thm}

\begin{remark*}
Due to the nonlinear constraint of taking values in the unit sphere $\Ss^2$, this density result is far from obvious. For the proof of Theorem \ref{thm:dense_rational}, we refer to Appendix \ref{sec:app_density} below. 
\end{remark*}

Our first main result shows that rational data always lead to unique global-in-time solutions of (HWM).

\begin{thm}[GWP for Rational Data] \label{thm:gwp_S2}
For every $\uu_0 \in \mathcal{R}at(\R; \Ss^2)$, there exists a unique global-in-time solution $\uu \in C(\R; H^\infty_\bullet(\R; \Ss^2))$ of \eqref{eq:HWM} with initial datum $\uu(0) = \uu_0$.
\end{thm}

\begin{remarks*} 
1) The global solutions $\uu : \R \times \R \to \Ss^2$ constructed above are of the form
$$
\uu(t) = \uu_\infty + \mathbf{v}(t) \in \Ss^2 + C(\R; H^\infty(\R; \R^3))
$$
with the point $\uu_\infty = \lim_{|x| \to \infty} \uu_0(x) \in \Ss^2$ given by the initial datum $\uu_0 \in \mathcal{R}at(\R; \Ss^2)$. See also below, for the definition of the space $H^\infty_\bullet(\R; \Ss^2)$. 

2) Our result establishes global well-posedness of \eqref{eq:HWM} for initial data belonging to a {\em dense} subset in the scaling-critical energy space $\dot{H}^{\frac 1 2}(\R; \Ss^2)$. Hence any finite-time blowup solution for \eqref{eq:HWM} in the energy space -- provided such solutions exist at all -- must be highly unstable.

3) The  solutions of Theorem \ref{thm:gwp_S2} exhibit an infinite set of conserved quantities 
$$
I_{p}(\uu(t)) = I_{p}(\uu_0) \quad \mbox{for $p \geq 1$}
$$ 
due to the Lax structure for \eqref{eq:HWM}. In particular, we obtain conservation of energy $E(\uu(t)) = E(\uu_0) \sim I_2(u_0)$. As a consequence by Peller's theorem, we obtain the infinite family of a-priori bounds
$$
\| \uu(t) \|_{\dot{B}^{1/p}_p} \lesssim_p \| \uu_0 \|_{\dot{B}^{1/p}_p} \quad \mbox{for $p \geq 1$}
$$
with the homogeneous Besov semi-norms $\| \cdot \|_{\dot{B}^{1/p}_p}$; see Section \ref{sec:lax} for details. However, these bounds do not seem to provide strong enough control to deduce global existence. In this paper, we thus use an entirely different approach based on an explicit flow formula for \eqref{eq:HWM}. 

4) In \cite{BeKlLa-20}, the authors study the dynamics for rational initial data $\uu_0 : \R \to \Ss^2$ with simple poles and derive an self-consistent system of ordinary differential equations of spin Calogero--Moser type. However, by following this approach, it still remains unclear whether such rational solutions can be extended globally in time, since a possible loss of simplicity of poles could arise at finite time, rendering the simple pole ansatz invalid in finite time.   
\end{remarks*}

\subsection*{Soliton Resolution and Non-Turbulence}

As our next main result, we discuss the long-time behavior of rational solutions provided by Theorem \ref{thm:gwp_S2} above. Here a suitable spectral condition will enter the scene as follows. For $\uu \in \mathcal{R}at(\R; \Ss^2)$, we define the {\em Toeplitz operator}  by
$$
T_{\Ub} f = \Pi_+ ( \Ub f) \quad \mbox{for $f \in L^2_+(\R; \C^2)$} \, .
$$
Here $\Pi_+ : L^2(\R; \C^2) \to L^2_+(\R; \C^2)$ is the Cauchy--Szeg\H{o} projection onto the vector-valued Hardy space defined as
$$
L^2_+(\R; \C^2) := \left \{ f \in L^2(\R; \C^2) \mid \mbox{$\mathrm{supp} \, \widehat{f}_k \subset [0, \infty)$ for $k=1,2$} \right \}\, . 
$$
The symbol in $T_\Ub$ is given by the matrix-valued function $\Ub : \R \to \C^{2 \times 2}$ with
\be
\Ub = \uu  \cdot \bm{\sigma} = \sum_{k=1}^3 u_k \sigma_k = \left ( \begin{array}{cc} u_3 & u_1 - \ii u_2 \\ u_1 + \ii u_2 & -u_3 \end{array} \right ) ,
\ee
where $\bm{\sigma}=(\sigma_1, \sigma_2, \sigma_3)$ contains the standard Pauli matrices. For later use, we also remark that, by introducing the matrix-valued function $\Ub = \uu \cdot \bm{\sigma}$, we can equivalently rewrite \eqref{eq:HWM} as 
\be
\pt_t \Ub = -\frac{\ii}{2} [\Ub, |D| \Ub] \, ,
\ee
where $[X,Y] \equiv XY-YX$ is the commutator of matrices; see also \cite{GeLe-18} for a more details on this.

In fact, by recasting \eqref{eq:HWM} in terms of the matrix-valued function $\Ub$, we will be able to fully exploit the Lax structure initially found in \cite{GeLe-18}. Also, we note that $\Ub(x) =\Ub(x)^* \in \C^{2 \times 2}$ takes values in the Hermitian matrices subject to the algebraic constraint that $\Ub(x)^2 = \mathds{1}_2$. As a consequence, the Toeplitz operator $T_\Ub=T_\Ub^*$ is self-adjoint with operator norm $\| T_\Ub \| \leq 1$. Moreover, it turns out that  $T_\Ub$ will be a Lax operator along the flow. Hence its spectrum $\sigma(T_{\Ub(t)})$ will be preserved in time for solutions $\uu(t)$ of \eqref{eq:HWM}. As another key fact, we mention  that the discrete spectrum
$$
\sigma_{\mathrm{d}}(T_\Ub) = \{ \lambda \in \sigma(T_{\Ub}) \mid \mbox{$\lambda$ is isolated and has finite multiplicity} \}
$$
is {\em finite} if and only if the function $\uu : \R \to \Ss^2$ is rational; see Section \ref{sec:spectral} for a detailed discussion of the spectral properties of $T_{\Ub}$ for general $\uu \in \dot{H}^{\frac 1 2}(\R; \Ss^2)$.

Our next main result will prove that simplicity of the discrete spectrum $\sigma_{\mathrm{d}}(T_{\Ub})$ implies scattering of the corresponding global rational solution $\uu \in C(\R; H^\infty_\bullet(\R; \Ss^2))$ into a sum of traveling ground state solitary waves receding from each other, i.\,e., we obtain {\em soliton resolution} in this case. From \cite{LeSc-18} we recall that {\em traveling solitary waves} for \eqref{eq:HWM} are, by definition, finite-energy solutions of the form
\be
\uu_{v}(t,x) = \qv_v(x- vt)
\ee 
with some profile $\qv \in \dot{H}^{\frac{1}{2}}(\R; \Ss^2)$ and where $v \in \R$ corresponds to the traveling velocity. From the complete classification result in \cite{LeSc-18} we recall that the any such profile $\qv_v$ can be expressed in terms of a finite Blaschke product, whence it follows that $\qv_v \in \mathcal{R}at(\R; \Ss^2)$ holds. Moreover, the energy is quantizied according to  
\be \label{eq:E_qv}
E(\qv_v) = (1-v^2) \cdot m \pi \quad \mbox{with some $m=0,1,2, \ldots$}
\ee
where the integer $m$ corresponds to the degree of the Blaschke product. The case $m=0$ corresponds to the trivial case of constant $\qv_v$, whereas for non-constant profiles $\qv_v$ we must have that 
\be
|v| < 1 \, .
\ee
Note also that the special case $v=0$ yields static solutions of \eqref{eq:HWM} and the profiles $\qv_{v=0}$ are then so-called {\em half-harmonic maps,} see also \cite{DaRi-11,MiSi-15}.

 In view of \eqref{eq:E_qv}, we refer to the case $m=1$ as {\em ground state} solitary waves, since these are non-trivial with the least possible energy for a given velocity. From the explicit classification in \cite{LeSc-18} we can deduce that profiles $\qv_v \in \mathcal{R}at(\R; \Ss^2)$ for ground state solitary waves are exactly rational functions of the form
\be
\qv_v(x) = \qv_\infty + \frac{\sv}{x-z} + \frac{\ov{\sv}}{x-\ov{z}}
\ee
with some $\qv_\infty \in \Ss^2$, $z \in \C_-$, and  $\sv \in \C^3 \setminus \{ 0 \}$ satisfying the nonlinear constraints
\be \label{eq:qv_constraints}
\sv \cdot \sv = 0 \quad \mbox{and} \quad \sv \cdot \left ( \qv_\infty + \frac{\ov{\sv}}{z-\ov{z}} \right ) = 0 \, .
\ee
Here $\mathbf{a} \cdot \mathbf{b} = \sum_{k=1}^3 a_k b_k$ denotes the non-Hermitian dot product of $\mathbf{a}, \mathbf{b} \in \C^3$. We remark that  \eqref{eq:qv_constraints} is easily seen to be equivalent (by partial fraction expansion) to the geometric constraint that $\qv_v(x) \in \Ss^2$ for all $x \in \R$. Moreover, the real part $\mathrm{Re} \, z$ corresponds to the spatial center of the solitary wave profile $\qv_v$, whereas $E(\qv_v)= (\sv \cdot \ov{\sv}) \cdot \pi = (1-v^2) \cdot \pi$ yields its energy. For more details on $\qv_v$, we refer to the discussion in Appendix \ref{app:perturbation} below.

We are now ready to state our second main result.

\begin{thm}[Soliton Resolution and Non-Turbulence] \label{thm:soliton_S2}
Let $\uu_0 \in \mathcal{R}at(\R; \Ss^2)$ and suppose the corresponding Toeplitz operator $T_{\Ub_0}  : L^2_+(\R; \C^2) \to L_+^2(\R; \C^2)$ has \textbf{simple} discrete spectrum $\sigma_{\mathrm{d}}(T_{\Ub_0})=\{ v_1, \ldots, v_N \}$. 

Then the corresponding solution $\uu \in C(\R; H^\infty_\bullet(\R; \Ss^2))$ of \eqref{eq:HWM} with initial datum $\uu(0) = \uu_0$ satisfies
$$
\lim_{t \to \pm \infty} \| \uu(t) - \uu^\pm(t) \|_{\dot{H}^s} = 0 \quad \mbox{for all $s > 0$} \, ,
$$
where
$$
\uu^{\pm}(t,x) = \sum_{j=1}^N \qv_{v_j}(x-v_jt) - (N-1) \uu_\infty \, .
$$
Here each $\qv_{v_j} \in \mathcal{R}at(\R; \Ss^2)$ is a profile of a ground state solitary wave for \eqref{eq:HWM} with traveling velocity $v_j$ and it is given by
$$
\qv_{v_j}(x) = \uu_\infty + \frac{\sv_j}{x-y_j+\ii \delta_j} + \frac{\ov{\sv}_j}{x-y_j-\ii \delta_j}
$$
with some complex vectors $\sv_{1}, \ldots, \sv_N \in \C^3 \setminus \{ 0 \}$, some real numbers $y_1, \ldots, y_N \in \R$, some positive real numbers $\delta_1, \ldots, \delta_N  > 0$, and the point $\uu_\infty = \lim_{|x| \to \infty} \uu_0(x) \in \Ss^2$. 

Moreover, we have the a-priori bounds
$$
\sup_{t \in \R} \| \uu(t) \|_{\dot{H}^s} \leq C(\uu_0,s) \quad \mbox{for all $s >0$} \, .
$$ 
\end{thm}

\begin{remarks*}
1) Obtaining a-priori bounds on all higher Sobolev norms $\| \uu(t) \|_{\dot{H}^s}$ is a remarkable fact, since the infinite hierarchy of conservation laws given by the Lax structure for \eqref{eq:HWM} only provides a-priori control over the weaker homogeneous Besov norms $\| \uu(t) \|_{\dot{B}^{1/p}_{p}}$ for $1 \leq p < \infty$. The latter fact follows from Peller's theorem applied to the Hankel operator $H_{\Ub}$ and the conserved quantities given by the operator traces $\Tr(|H_{\Ub}|^p)$; see  \cite{GeLe-18} for more details.

2) Note that the scattering profile $\uu^{\pm}(t)$ is the same for both $t \to -\infty$ and $t \to +\infty$, which can be seen as {\em triviality of the scattering map} in this setting.

3) It would be interesting to prove or disprove the existence of rational initial data $\uu_0$ leading to turbulent behaviour in the sense of growth of higher Sobolev norms such that $\| \uu(t) \|_{\dot{H}^s} \to +\infty$ as $t \to \infty$ for some $s > \frac 1 2$. Of course, the discrete spectrum $\sigma_{\mathrm{d}}(T_{\Ub_0})$ for such data must have degenerate eigenvalues. 

4) It is interesting to compare our result to other completely integrable equations with a Lax pair structure on Hardy spaces: In \cite{GeLe-24}, turbulent rational global-in-time solutions have been constructed for the {\em Calogero--Moser derivative NLS} on the real line. For the {\em cubic Szeg\H{o} equation} on the real line, turbulent rational solutions have been proven to exist in \cite{GePu-24b} along with their genericity. 
\end{remarks*}

We conclude this subsection by establishing that the spectral assumption in Theorem \ref{thm:soliton} for the Toeplitz operator $T_{\Ub_0}$ holds on a dense subset in  $\dot{H}^{\frac 1 2}(\R; \Ss^2)$, thereby showing that the soliton resolution above holds on a dense subset in the energy space.

\begin{thm} \label{thm:generic}
The subset
$$
\mathcal{R}at_{\mathrm{s}}(\R;\Ss^2) := \left \{ \uu \in \mathcal{R}at(\R; \Ss^2) \mid \mbox{$\sigma_{\mathrm{d}}(T_{\Ub})$ {\em is simple}} \right \}
$$
is dense in $\dot{H}^{\frac 1 2}(\R; \Ss^2)$.
\end{thm} 

We remark that the nonlinear constraint of taking values in $\Ss^2$ poses serious challenges when proving this density result. Also, the reader should avoid the fallacy of claiming that rational functions $\uu \in \mathcal{R}at(\R; \Ss^2)$ with simple poles will always lead to simple discrete spectrum $\sigma_{\mathrm{d}}(T_\Ub)$. We refer to Section \ref{sec:spectral} for more details.

\bigskip
\begin{center}
***
\end{center}
\medskip

\subsection*{Generalized Half-Wave Maps Equation}
We now discuss a natural geometric generalization of (HWM) beyond the target $\Ss^2$. The reader who is mainly interested in the $\Ss^2$-valued case may skip this subsection at first reading.
 
 For a given integer $d \geq 2$, we let $M_d(\C) \equiv \C^{d \times d}$ denote the vector space of complex $d \times d$-matrices. For matrix-valued maps $\Ub : [0,T) \times \R \to M_d(\C)$, we introduce the \textbf{generalized half-wave maps equation} given by
\be \label{eq:HWMd} \tag{HWM$_d$}
\pt_t \Ub = -\frac{\ii}{2} [\Ub, |D| \Ub] \, ,
\ee
subject to the initial condition $\Ub(0) = \Ub_0 : \R \to M_d(\C)$ satisfying the algebraic constraints such that
\be \label{eq:consU}
\Ub_0(x) = \Ub_0(x)^* \quad \mbox{and} \quad \Ub_0(x)^2 = \mathds{1}_d \quad \mbox{for a.\,e.~$x \in \R$} \, .
\ee 
We readily check that these properties of $\Ub_0$ are formally preserved along the flow of \eqref{eq:HWMd}. At this point, we also mention that \eqref{eq:HWMd} can be formally seen the zero-dispersion limit of the so-called {\em spin Benjamin--Ono equation} recently introduced in \cite{BeLaLe-22}; see also below for further remarks on this. 

The matrix-valued generalization of (HWM) above also has a straightforward geometric meaning as follows. Let $\Gr_k(\C^d)$ denote the {\em complex Grassmannian} consisting of the $k$-dimensional subspaces of the complex vector space $\C^d$. We recall that $\Gr_k(\C^d)$ can be canonically identified with the space of self-adjoint projections $P =P^* \in M_d(\C)$ with $\mathrm{rank}(P)=k$. Since $\mathrm{Tr}(P)=\mathrm{rank}(P)$ for such projections $P$, we find
\be
\Gr_k(\C^d) = \left \{ P \in M_d(\C) \mid \mbox{$P^*=P = P^2$ and $\Tr(P)=k$} \right \}.
\ee
We remark that $\Gr_k(\C^d)$ is a compact submanifold of real dimension $2k(n-k)$ embedded in $M_d(\C)$. In fact, we have that $\Gr_k(\C^d)$ is a compact complex K\"ahler manifold, see also below.

Thanks to the elementary affine relation 
\be \label{eq:AP}
U= \mathds{1}_d - 2 P \, ,
\ee
 we obtain the natural identification of the complex Grassmannians such that
\be \label{eq:Gr_k}
\Gr_k(\C^d) \cong \left \{ U \in M_d(\C) \mid \mbox{$U=U^*, \, U^2 = \mathds{1}_d$ and $\Tr(U) = d- 2k$} \right \}
\ee
for all $k=0, \ldots, d$.  With the simple relation \eqref{eq:AP} in mind, we will use the slight abuse of notation and identify elements in the right-hand side in \eqref{eq:Gr_k} as elements in $\Gr_k(\C^d)$ in what follows.  Moreover, we will throughout our discussion also include the trivial cases when $k=0$ or $k= d$ corresponding to $\{ \mathds{1}_d \}$ or $\{-\mathds{1}_d \}$, respectively.

In addition to the constraints \eqref{eq:consU}, it is easy to see that the matrix trace $\mathrm{Tr}(\Ub(t,x))$ is formally preserved in time along the flow of \eqref{eq:HWMd}. Hence we can view solutions of \eqref{eq:HWMd} as maps 
$$
\Ub : [0,T) \times \R \to \Gr_k(\C^d) \, ,
$$ 
provided that the initial condition $\Ub_0 : \R \to M_d(\C)$ satisfies the pointwise condition
\be \label{eq:trace}
\Tr(\Ub_0(x))= d-2k \quad \mbox{for a.\,e.~$x \in \R$}
\ee
in addition to the constraints \eqref{eq:consU} above. 

\begin{remarks*}
1) For $d=2$ and $k=1$, we see that \eqref{eq:HWMd} reduces to \eqref{eq:HWM} in accordance with the classical fact that $\Gr_1(\C^2) \cong \mathbb{CP}^1 \cong \Ss^2$. 

2) For general $d \geq 2$ and $k=1$, we recall that $\Gr_1(\C^d) \cong \mathbb{CP}^{d-1}$. In particular, our global well-posedness result below will apply to the generalized half-wave maps equation with target being the complex projective spaces $\mathbb{CP}^{d-1}$ for any $d \geq 2$.  
\end{remarks*}

We will prove that \eqref{eq:HWMd} also possess a Lax structure on suitable $L^2$-based Hardy spaces, which will be discussed in Section \ref{sec:lax} below. For $d\geq 2$ and $0 \leq k \leq d$ given,  we observe that the natural energy space for \eqref{eq:HWMd} reads
$$
\dot{H}^{\frac 1 2}(\R; \mathbf{Gr}_k(\C^d))  := \left \{ \Ub \in \dot{H}^{\frac 1 2}(\R; M_d(\C)) \mid \mbox{$\Ub(x) \in \mathbf{Gr}_k(\C^d)$ for a.\,e.~$x \in \R$} \right \} \, 
$$
equipped with the natural Gagliardo semi-norm $\| \cdot \|_{\dot{H}^{\frac 1 2}}$ whose square is  (up to a multiplicative constant)  the energy functional for \eqref{eq:HWMd} given by
\be
E(\Ub) = \frac{1}{2} \| \Ub \|_{\dot{H}^{\frac 1 2}}^2 = \frac{1}{4 \pi} \int_{\R} \int_{\R} \frac{ | \Ub(x)- \Ub(y)|_F^2}{|x-y|^2} \, dx \, dy \, .
\ee
Here $| A |_F = \sqrt{\Tr(A^* A)}$ denotes the natural Frobenius norm for matrices $A \in M_d(\C)$.

In analogy to our analysis of \eqref{eq:HWM}, we define the set 
$$
\mathcal{R}at(\R; \mathbf{Gr}_k(\C^d)) := \left \{ \Ub : \R \to \mathbf{Gr}_k(\C^d) \mid \mbox{$\Ub(x)$ is rational} \right \}  .
$$
We have the following global well-posedeness result about \eqref{eq:HWMd} for rational initial data, which includes Theorem \ref{thm:gwp_S2} as a special case.

\begin{thm}[GWP of \eqref{eq:HWMd} for Rational Data] \label{thm:gwp}
Let $d \geq 2$ and $0 \leq k\leq d$ be integers. Then, for every initial datum $\Ub_0 \in \mathcal{R}at(\R; \Gr_k(\C^d))$, there exists a unique global-in-time solution $\Ub \in C(\R; H^\infty_\bullet(\R; \Gr_k(\C^d)))$ of \eqref{eq:HWMd} with  $\Ub(0) = \Ub_0$.
\end{thm}

Generalizing the density result in Theorem \ref{thm:dense_rational_Gr}, we have the following result proven in Appendix \ref{sec:app_density}.

\begin{thm} \label{thm:dense_rational_Gr}
For every $d\geq 2$ and $0 \leq k \leq d$, the subset $\mathcal{R}at(\R; \Gr_k(\C^d))$ is dense in $\dot{H}^{\frac 1 2}(\R; \Gr_k(\C^d))$.
\end{thm}

\begin{remark*}
The reader may wonder about finding explicit elements in $\mathcal{R}at(\R; \Gr_k(\C^d))$. Indeed, in the case $\Gr_1(\C^d) \cong \mathbb{CP}^{d-1}$, we can easily construct rational maps as follows. Let $P_1, \ldots, P_d \in \C[X]$ be polynomials such that $f(x) := (P_1(x), \ldots, P_d(x)) \in \C^d \setminus \{ 0 \}$ for all $x \in \R$. Evidently, the map $P : \R \to M_d(\C)$ with
$$
P(x) := \frac{f(x) \ov{f}(x)^t}{\langle f(x),  f(x) \rangle_{\C^d}} 
$$ 
satisfies $P(x)=P(x)^*=P(x)^2$ with $\Tr(P(x)) \equiv 1$. Thus $\Ub(x) = \mathds{1}_d - 2 P(x)$ belongs to $\mathcal{R}at(\R; \Gr_1(\C^d))$. 
\end{remark*}

Next, we will extend Theorem \ref{thm:soliton_S2} to the setting of half-wave maps with target $\Gr_k(\C^d)$. Here, for a given initial datum $\Ub_0 \in \mathcal{R}at(\R; \Gr_k(\C^d))$, the corresponding Toeplitz operator $T_{\Ub_0} : L^2_+(\R; \C^d) \to L^2_+(\R;\C^d)$ is analogously defined via $T_{\Ub_0} f = \Pi_+(\Ub_0 f)$. Furthermore, the notion of traveling solitary waves for \eqref{eq:HWMd} is defined in the obvious manner: We say that a finite-energy solution to \eqref{eq:HWMd} of the form
$$
\Ub_v(t,x) = \Qv_v(x-vt)
$$
is a {\em traveling solitary wave} with profile $\Qv_v \in \dot{H}^{\frac 1 2}(\R; \Gr_k(\C^d))$ and velocity $v \in \R$. We have the following result.

\begin{thm}[Soliton Resolution and Non-Turbulence for \eqref{eq:HWMd}] \label{thm:soliton}
Let $d \geq 2$ and $0 \leq k \leq d$ be given. Suppose that $\Ub_0 \in \mathcal{R}at(\R; \Gr_k(\C^d))$ and that its Toeplitz operator $T_{\Ub_0} : L^2_+(\R; \C^d) \to L^2_+(\R; \C^d)$ has \textbf{simple} discrete spectrum $\sigma_{\mathrm{d}}(T_{\Ub_0})= \{v_1, \ldots, v_N \}$. 

Then the corresponding solution $\Ub \in C(\R; H^\infty_\bullet(\R; \Gr_k(\C^d)))$ of  \eqref{eq:HWMd} with initial datum $\Ub(0) = \Ub_0$ satisfies
$$
\lim_{t \to \pm \infty} \| \Ub(t) - \Ub^\pm(t) \|_{\dot{H}^s} = 0 \quad \mbox{for all $s > 0$} \, ,
$$
where
$$
\Ub^{\pm}(t,x) =  \sum_{j=1}^N \Qv_{v_j}(x-v_j t) - (N-1) \Ub_\infty \, .
$$
Here each $\Qv_{v_j} \in \mathcal{R}at(\R; \Gr_k(\C^d))$ is a profile of a traveling solitary wave for \eqref{eq:HWMd} with velocity $v_j$ and it is given by
$$
\Qv_{v_j}(x) =  \Ub_\infty + \frac{A_j}{x- y_j  + \ii \delta_j} + \frac{A_j^*}{x-y_j- \ii \delta_j} \, ,
$$
with some matrices $A_j \in M_d(\C)$ with $\rank(A_j)=1$ and $A_j^2 = 0$ for $j =1, \ldots, N$, some real numbers $y_1, \ldots, y_N \in \R$, some positive real numbers $\delta_1, \ldots, \delta_N > 0$, and the constant matrix $\Ub_\infty = \lim_{|x| \to \infty} \Ub_0(x) \in \Gr_k(\C^d)$.  

Moreover, we have the a-priori bounds
$$
\sup_{t \in \R} \| \Ub(t) \|_{\dot{H}^s} \leq C(\uu_0,s) \quad \mbox{for all $s >0$} \, .
$$ 
\end{thm}

\begin{remarks*}
1) In the particular case $\Gr_1(\C^2) \cong  \mathbb{CP}^1 \cong \Ss^2$, we obtain Theorem \ref{thm:soliton_S2} above, except that  we also find in Theorem \ref{thm:soliton_S2} that the traveling solitary profiles  are  known to be of ground state type in this case. For general targets $\Gr_k(\C^d)$ with $(k,d) \neq (1,2)$, the complete classification of traveling solitary waves is open and hence we can only conclude that the profiles $\Qv_{v_j}$ above give rise to traveling solitary wave for \eqref{eq:HWMd} with velocity $v_j$.

2) It would be desirable to extend the density result for the simplicity condition on the discrete spectrum $\sigma_{\mathrm{d}}(T_{\Ub_0})$ stated in Theorem \ref{thm:generic} to general targets $\Gr_k(\C^d)$.
\end{remarks*}

\bigskip
\begin{center}
***
\end{center}
\medskip

\subsection*{Strategy of Proofs}
   
Let us briefly outline the main ideas used in this paper.
 
The starting point of our analysis is a detailed study of the Lax pair structure of \eqref{eq:HWMd}. In particular, this will largely extend the previous results found in \cite{GeLe-18} for \eqref{eq:HWM} with target $\Ss^2$. More precisely, we will show that, given a sufficiently smooth solution $\Ub : [0,T] \times \R \to \Gr_k(\C^d)$ of the matrix-valued \eqref{eq:HWMd}, we obtain the following Lax equation 
\be \label{eq:lax_intro}
\frac{d}{dt} T_{\Ub(t)} =  \big [B^+_{\Ub(t)}, T_{\Ub(t)} \big ] \, .
\ee     
Here $T_\Ub : L^2_+(\R; \VV) \to L^2_+(\R;\VV)$ denotes the {\em Toeplitz operator} given by
$$
T_{\Ub} f = \Pi_+( \Ub f) \quad \mbox{for $f \in L^2_+(\R; \VV)$}\, ,
$$ 
where $\VV$ either stands for 
$$
\VV=\C^d \quad \mbox{or} \quad \VV=M_d(\C) \, ,
$$ 
equipped with their canonical scalar products, see below. In fact, we shall use both choices of $\VV$ in the course of our analysis below. Moreover, we remark that $T_{\Ub} = T_{\Ub}^*$ is self-adjoint and bounded with operator norm $\| T_\Ub \| = \| \Ub \|_{L^\infty}=1$ thanks to the algebraic constraints imposed on the matrix-valued function $\Ub$. The second operator appearing in \eqref{eq:lax_intro} reads
\be
B^+_\Ub = \frac{\ii}{2} \left ( D \circ T_\Ub + T_{\Ub} \circ D \right ) - \frac{\ii}{2}  T_{|D| \Ub} \, ,
\ee
which is an unbounded skew-adjoint operator with  $\dom(B_\Ub) = H^1_+(\R; \VV)$ as its operator domain. 

Now, another decisive feature of the Lax structure for \eqref{eq:HWMd} enters, which again is due to the algebraic constraints satisfied by the matrix-valued function $\Ub$. Notably, we can derive the following key identity
\be \label{eq:key}
\boxed{ T_{\Ub}^2 = \id - H_{\Ub}^* H_{\Ub}}
\ee 
where $H_\Ub : L^2_+(\R; \VV) \to L^2_-(\R; \VV)$ denotes the {\em Hankel operator} given by
$$
H_{\Ub} f = \Pi_- (\Ub f) \quad \mbox{for $f \in L^2_+(\R; \VV)$}
$$
where $\Pi_- := \id-\Pi_+$ denotes the projection in $L^2(\R;\VV)$ onto orthogonal complement of the Hardy space $L^2_+(\R; \VV)$. By the Lax evolution \eqref{eq:lax_intro} combined with \eqref{eq:key}, we obtain the infinite set of conserved quantities for \eqref{eq:HWMd} of the form
\be
I_p(\Ub) = \Tr(|K_{\Ub}|^{p/2}) \quad \mbox{for any $p >0$} \, ,
\ee  
with the nonnegative operator $K_\Ub = H_{\Ub}^* H_{\Ub}$. In particular, for $p=2$, we obtain the trace-class norm of $K_\Ub$ which is easily seen to be equivalent to the scaling-critical energy (semi-)norm $\| \Ub \|_{\dot{H}^{\frac{1}{2}}}$; see Section \ref{sec:lax} for more details. Furthermore, we see from \eqref{eq:key} that $T_{\Ub}$ is Fredholm with index 0. We will make use of this fact further below in our analysis.

\medskip
However, as we have already mentioned above, the infinite family of conserved quantities $\{ I_p(\Ub) \}_{p \geq 1}$ does not seem to yield sufficient control to obtain global solutions for \eqref{eq:HWMd}, even for smooth and sufficiently small initial data (i.\,e., small perturbations of a constant). To overcome this obstruction, we shall  derive an {\em explicit flow formula} for sufficiently smooth solutions of \eqref{eq:HWMd}, which is akin to the result  discovered in \cite{Ge-23} for the Benjamin--Ono equation. More precisely, for solutions of \eqref{eq:HWMd} of the form
$$ 
\Ub(t) = \Ub_\infty + \Vb(t) \in M_d(\C) \oplus C([0,T]; H^s(\R; M_d(\C)) \quad \mbox{with $s > \frac{3}{2}$} \, ,
$$  
we derive that 
\be \label{eq:explicit_intro}
\boxed{\Pi_+ \Vb(t,z) = \frac{1}{2\pi \ii} I_+ \left ( (X^* + t T_{\Ub_0} -z \id)^{-1} \Pi_+ \Vb_0  \right ) \quad \mbox{for $t \in [0,T]$ and $z \in \C_+$}} 
\ee
where $\Ub_0 = \Ub_\infty + \Vb_0 \in M_d(\C) \oplus H^s(\R; M_d(\C))$ denotes the initial datum for \eqref{eq:HWMd}. In this formula, we emphasize the fact that $T_{\Ub_0}$ is now regarded as a Toeplitz operator acting on the Hardy space $L^2_+(\R; M_d(\C))$ with functions taking values in the space of complex $d \times d$-matrices $M_d(\C)$. Furthermore, in analogous fashion to \cite{Ge-23}, the operators $I_+$ and $X^*$ are given by
$$
I_+(f) = \lim_{\xi \to 0^+} \widehat{f}(\xi) \quad \mbox{and} \quad \widehat{(X^* f)}(\xi) = \ii \frac{d \widehat{f}}{d \xi}(\xi)
$$
defined on their suitable domains $\dom(I_+)$ and $\dom(X^*)$ in $L^2_+(\R;\VV)$; see Section \ref{sec:prelim} below for details. Now, the main challenge is to decide whether we can exploit this explicit representation above to deduce that these strong solutions can be extended to all (forward) times, i.\,e., whether it is true that $\Ub \in C([0,\infty); H^s_\bullet(\R; \Gr_k(\C^d)))$ holds? Surprisingly, this turns out to be a rather delicate question whose affirmative answer must  necessarily exploit the algebraic constraints satisfied by the matrix-valued function $\Ub$ solving \eqref{eq:HWMd}. By contrast, we remark that the explicit formula (up to an inessential rescaling of $t$) for the {\em dispersionless limit of the scalar-valued Benjamin--Ono} on the line reads the same as \eqref{eq:explicit_intro} with the simple replacement of $T_{\Ub_0}$ with the Toeplitz operator $T_{u_0} : L^2_+(\R;\C) \to L^2_+(\R; \C)$ with the bounded scalar-valued function $u_0 \in L^2(\R) \cap L^\infty(\R)$. However, for the dispersionless limit of (BO), it is known \cite{Ge-24c} that strong continuity of the flow breaks down in finite-time (corresponding to development of shocks). Thus we cannot expect to derive global-in-time existence for \eqref{eq:HWMd} by a naive use of \eqref{eq:explicit_intro} neglecting the algebraic constraints for $\Ub$.

\medskip
In order to further exploit the fact that the initial data $\Ub_0$ for \eqref{eq:HWMd} are valued in $\Gr_k(\C^d)$, we appeal again to the key identity \eqref{eq:key}. As a direct consequence, we obtain the natural orthogonal decomposition of the underlying Hardy space of the form
$$
L^2_+(\R; \VV) = \Hfr_0 \oplus \Hfr_1 
$$
with the closed subspace
$$
\Hfr_0 := \ker(K_{\Ub_0}) \quad \mbox{and} \quad \Hfr_1 := \Hfr_0^\perp = \overline{\ran(K_{\Ub_0})} \, ,
$$
where we recall the definition of the trace-class operator $K_{\Ub_0}=H_{\Ub_0}^* H_{\Ub_0}$. Now, it turns out that $\Pi_+ \Vb_0 \in \Hfr_1$ and, in addition to this, we see that $\Hfr_1$ is an {\em invariant subspace} of both $T_{\Ub_0}$ as well as the semigroup generated by $X^*$. As a consequence, we see that the resolvent appearing on right-hand side in \eqref{eq:explicit_intro} satisfies the mapping property $(X^* + t T_{\Ub_0} - z \id)^{-1} : \Hfr_1 \to \Hfr_1$ for any $t \in \R$. Hence the explicit flow formula found for \eqref{eq:HWMd} effectively takes place only the invariant subspace $\Hfr_1$. This is a great deal of information which can be used to deduce global existence of strong solutions! In particular, an adaptation of the classical Kronecker theorem for Hankel operators shows that 
$$
\dim(\Hfr_1) < +\infty \quad \mbox{if and only if} \quad \mbox{$\Ub_0$ is a rational map} \, . 
$$
Thanks to this fact, the proof of global existence of strong solutions via \eqref{eq:explicit_intro} for rational initial data $\Ub_0$ amounts to showing that $M(t)=X^* + t T_{\Ub_0}$ has no real eigenvalues for any $t \in \R$, proving its injectivity on $\Hfr_1$ and hence the surjectivity of $M(t)$ because $\Hfr_1$ is finite-dimensional in this setting.

\begin{remark*}
The case of non-rational initial data $\Ub_0$, which implies that $\dim \Hfr_1 = +\infty$, and the question to global well-posedness for \eqref{eq:HWMd} will be studied in our companion work \cite{GeLe-24b} posed on the torus.
\end{remark*}

Finally, let us briefly comment on the strategy behind the proofs of our further main results stated as Theorems \ref{thm:soliton_S2} and \ref{thm:soliton} concerning the long-time behaviour of rational solutions. Inspired by our recent study of $N$-solitons for the Calogero--Moser derivative NLS in \cite{GeLe-24},  the main idea rests on using the explicit flow formula combined with a perturbation analysis of the family of (bounded) operators
$$
\eps X^* + T_{\Ub_0} : \Hfr_1 \to \Hfr_1
$$
with $\eps = \frac{1}{t}$ in the regime where $\eps \to 0$ under the assumption that $T_{\Ub_0} : \Hfr_1 \to \Hfr_1$ has simple spectrum. However, as a striking difference to the analysis in \cite{GeLe-24}, we will encounter that turbulence (i.\,e.~growth of higher Sobolev norms) can be ruled out for rational solutions of \eqref{eq:HWMd} provided that the Lax operator $T_{\Ub_0} : L^2_+(\R; \C^d) \to L^2_+(\R; \C^d)$ has simple discrete spectrum.

\subsection*{Links to Schr\"odinger Maps and Spin Benjamin--Ono Equation} 
In order to put \eqref{eq:HWMd} in a broader geometric perspective, we recall the well-known fact that $\mathbf{Gr}_k(\C^d)$ is a K\"ahler manifold of complex dimension $k(d-k)$. Its complex structure $J_A$ on the tangent space $T_A \Gr_k(\C^d)$ at a point $A \in \mathbf{Gr}_k(\C^d)$ can be expressed as the matrix commutator
$$
J_A(X) = -\frac{\ii}{2}[A, X] \, .
$$
Thus we see that \eqref{eq:HWMd} can be written as a {\em Schr\"odinger maps-type equation} of the form\footnote{A-priori this geometric rewriting of \eqref{eq:HWMd} would involve using the projection $P_\Ub$ onto the tangent space $T_\Ub \Gr_k(\C^d)$, i.\,e., we have $\pt_t \Ub = J_{\Ub} P_\Ub |D| \Ub$. However, it can readily checked that $[\Ub, (\id-P_\Ub) B] = 0$ for Hermitian matrices $B  \in M_d(\C)$.}
\be \tag{SM}
\pt_t \Ub = J_{\Ub} |D| \Ub
\ee
with the first-order pseudo-differential operator $|D|$. However, we will not further exploit this geometric point of view in our analysis here. 

On the other hand, we also mention the remarkable fact that \eqref{eq:HWMd} can be formally seen as the {\em zero-dispersion limit} of the {\em spin Benjamin--Ono equation (sBO)}, which was recently introduced by Berntson--Langmann--Lenells in \cite{BeLaLe-22}. In our choice of units, this equation can be written as
\be \label{eq:sBO} \tag{sBO}
\pt_t \Vb = \frac{1}{2} \pt_x ( |D| \Vb - \Vb^2) - \frac{\ii}{2} [\Vb, |D| \Vb] \, ,
\ee
for the matrix-valued map $\Vb : [0,T) \times \R \to M_d(\C)$; see also \cite{Ge-23b} where a Lax pair structure for (sBO) was found. We notice that, in the special case of real-valued maps $\Vb(t,x) \in \R$, we obtain the classical Benjamin--Ono equation (apart from trivial rescaling of $t$ compared to the standard form of this equation). 

At least on a formal level, we see that replacing $|D|$ by $\eps |D|$ with $\eps > 0$ in \eqref{eq:sBO} and forcing the condition that $\Vb^2 = \mathds{1}_d$, we are led to \eqref{eq:HWMd} when formally taking the zero-dispersion limit as $\eps \to 0$. For a rigorous analysis of the zero-dispersion of the scalar Benjamin--Ono equation, we refer to the recent work in \cite{Ge-24c}. However, as already mentioned above, we will encounter a striking difference in our analysis here due to the algebraic constraint $\Ub^2 = \mathds{1}_d$ that is absent in the scalar case. From an operator theoretic point of view, this remarkable difference stems from the fact that Toeplitz operators $T_\Ub$ with matrix-valued symbols $\Ub : \R \to \C^{d \times d}$ for $d \geq 2$ can have entirely different spectral properties compared to Toeplitz operators $T_f$ with scalar-valued symbols $f : \R \to \C$. The interested reader will find more details on this difference further below.

\subsection*{Acknowledgments} P.~G.~was partially supported by the French Agence Nationale de la Recherche under the ANR project ISAAC–ANR-23–CE40-0015-01. E.\,L.~acknowledges financial support from the Swiss National Science Foundation (SNSF) under Grant No.~204121. In addition, E.\,L.~thanks Herbert Koch for valuable discussions and he thanks Yi Zhang for the kind hospitality and the opportunity to present this  work in a series of talks at the Chinese Academy of Sciences, Beijing, in September 2024.  Finally, we are grateful to the anonymous referee for valuable comments and suggestions that have helped improve this paper.

%%%%%%%%%%%%% Section 2
\section{Preliminaries and Notation} \label{sec:prelim}

In this section, we setup some definitions and notation used throughout this paper. 

\subsection*{Sobolev-type  Spaces} For the study of the generalized half-wave maps equations (HWM$_d$), we introduce the following Sobolev-type spaces for matrix-valued functions. For an integer $d \geq 2$, we use $M_d(\C) \equiv \C^{d \times d}$ to denote the Hilbert space of complex $d \times d$-matrices equipped with the inner product 
$$
\langle A, B \rangle_F := \Tr( A B^*) \quad \mbox{for $A,B \in M_d(\C)$} 
 $$
and the corresponding Frobenius norm of a matrix $A \in M_d(\C)$ will be denoted by $| A |_F = \sqrt{\langle A, A \rangle_F}$. 

The Lebesgue spaces $L^p(\R; M_d(\C))$ and $L^p_{\mathrm{loc}}(\R; M_d(\C))$ are defined in an obvious manner. For $s > 0$, we use the Sobolev spaces
$$
\dot{H}^s(\R; M_d(\C)) := \left \{ \Ub \in L^1_{\mathrm{loc}}(\R; M_d(\C)) \mid \| \Ub \|_{\dot{H}^s} := \| |D|^s \Ub \|_{L^2} < +\infty \right \} \, , 
$$
$$
H^s(\R; M_d(\C)) := \left \{ \Vb \in L^1_{\mathrm{loc}}(\R; M_d(\C)) \mid \| \Vb \|_{H^s} := \| \langle D \rangle^s \Vb \|_{L^2} < +\infty \right \} \, .
$$
We set $\dot{H}^\infty(\R; M_d(\C)) := \cap_{s > 0} \dot{H}^s(\R; M_d(\C))$ and $H^\infty(\R; M_d(\C)) := \cap_{s >0} H^s(\R; M_d(\C))$. Note that $\| \cdot \|_{\dot{H}^s}$ is a semi-norm, since non-trivial constant maps also belong to $\dot{H}^s(\R; M_d(\C))$ for $s>0$. Furthermore, for $0 \leq k \leq d$ given, we define the spaces
$$
\dot{H}^s(\R; \Gr_k(\C^d)) := \left \{ \Ub \in \dot{H}^s(\R; M_d(\C)) \mid \mbox{$\Ub(x) \in \Gr_k(\C^d)$ for a.\,e.~$x \in \R$} \right \} \, .
$$
Note that the scaling-critical energy space associated to (HWM$_d$) with target $\Gr_k(\C^d)$ is $\dot{H}^{\frac 1 2}(\R; \Gr_k(\C^d))$ equipped with the Gagliardo semi-norm $\| \cdot \|_{\dot{H}^{\frac 1 2}}$ such that
\be \label{eq:H_half}
\| \Ub \|_{\dot{H}^{\frac 1 2}}^2 = \| |D|^{\frac 1 2} \Ub \|_{L^2}^2 = \frac{1}{2 \pi} \int_{\R} \! \int_{\R} \frac{|\Ub(x)- \Ub(y)|_F^2}{|x-y|^2} \, dx \, dy \, .
\ee
Note that $E(\Ub) = \frac{1}{2} \| \Ub \|_{\dot{H}^{\frac 1 2}}^2$ is the Hamiltonian energy functional for (HWM$_d$) with the natural symplectic form for maps defined on $\R$ with values in the complex Grassmannian $\Gr_k(\C^d)$.

\medskip
In addition to the space $\dot{H}^s$-spaces, it turns out to be convenient to introduce the following family of affine inhomogeneous Sobolev-type spaces given by
$$
H^s_\bullet(\R; M_d(\C)) := M_d(\C) \oplus H^s(\R; M_d(\C)) 
$$
and we define $H^\infty_\bullet(\R; M_d(\C)) := \cap_{s > 0} H^s_\bullet(\R;M_d(\C))$. Furthermore, we set
$$
H^s_\bullet(\R; \Gr_k(\C^d)) := \left \{ \Ub\in H^s_\bullet(\R; M_d(\C)) \mid \mbox{$\Ub(x) \in \Gr_k(\C^d)$ for a.\,e.~$x\in \R$} \right \}  \, .
$$
It is easy to see that the following strict inclusions hold true:
$$
\mathcal{R}at(\R; \Gr_k(\C^d)) \subsetneq H^s_\bullet(\R; \Gr_k(\C^d)) \subsetneq  \dot{H}^s(\R; \Gr_k(\C^d)) \subsetneq L^\infty(\R; M_d(\C)) \, ,
$$
where we recall that $\mathcal{R}at(\R; \Gr_k(\C^d))$ denotes the space of rational maps from $\R$ to $\Gr_k(\C^d)$.

\medskip
For (HWM) with maps valued in the unit sphere $\Ss^2 \subset \R^3$, we make use of the corresponding Sobolev spaces $H^s(\R; \R^3)$ and $\dot{H}^s(\R; \R^3)$, where the energy space is 
$$
\dot{H}^{\frac 1 2}(\R; \Ss^2) := \left \{ \uu \in \dot{H}^{\frac 1 2}(\R; \R^3) \mid \mbox{$\uu(x) \in \Ss^2$ for a.\,e.~$x \in \R$} \right \} 
$$
endowed with the Gagliardo semi-norm $\| \cdot \|_{\dot{H}^{\frac 1 2}}$ such that
$$
\| \uu \|_{\dot{H}^{\frac 1 2}}^2 = \| |D|^{\frac 1 2} \uu \|_{L^2}^2 = \frac{1}{2 \pi} \int_{\R} \! \int_{\R} \frac{|\uu(x)- \uu(y)|^2}{|x-y|^2} \, dx \, dy \, .
$$
From the introduction above, we recall that unit vectors $\uu \in \Ss^2$ can be equivalently encoded by using the standard Pauli matrices $(\sigma_1, \sigma_2, \sigma_3)$ via the relation
$$
\Ub = \uu \cdot \bm{\sigma} = u_1 \sigma_1 +u_2 \sigma_2 + u_3 \sigma_3 = \left ( \begin{array}{cc} u_3 & u_1 - \ii u_2 \\ u_1 + \ii u_2 & -u_3 \end{array} \right ) \, ,
$$
where we easily check that $\Ub = \Ub^*$ with $\Ub^2 = \mathds{1}_2$ and $\Tr(\Ub) = 0$. Also, we find that $u_k= \frac{1}{2} \Tr(\Ub \sigma_k) = \frac{1}{2} \langle \Ub, \sigma_k \rangle_F$ for $k=1,2,3$. Thus, by means of the relation $\Ub = \uu \cdot \bm{\sigma}$, we find the equivalence of (semi)-norms $\| \uu \|_{\dot{H}^{s}} \sim \| \Ub \|_{\dot{H}^{s}}$ for all $s>0$.

\subsection*{Hardy Spaces, Toeplitz and Hankel Operators}
We consider the Hilbert space $L^2(\R; \VV)$ for maps on $\R$ with values in the finite-dimensional Hilbert spaces 
$$
\VV = \C^d \quad \mbox{or} \quad \VV = M_d(\C) \, ,
$$
which we endow with their natural inner products and norms, i.\,e.,
$$
\langle u,v \rangle_{\VV} = \sum_{k=1}^d u_k \ov{v}_k \quad \mbox{if $\VV= \C^d$}, \qquad \langle A,B \rangle_\VV = \Tr(A B^*) \quad \mbox{if $\VV = M_d(\C)$} \, .
$$
The {\em Cauchy--Szeg\H{o} projection} $\Pi_+ : L^2(\R; \VV) \to L^2_+(\R;\VV)$ onto the Hardy space
$$
L^2_+(\R; \VV) := \{ f \in L^2(\R; \VV) \mid \mathrm{supp} \, \widehat{f} \subset [0, \infty) \}
$$
is given by
$$
(\Pi_+ f)(x) := \frac{1}{2 \pi} \int_0^{+\infty} \eu^{\ii x \xi} \widehat{f}(\xi) \, d\xi \quad \mbox{with} \quad \widehat{f}(\xi) = \int_{-\infty}^{+\infty} f(x) \eu^{-\ii \xi x} \, dx \, ,
$$ 
or, equivalently, we have $\widehat{(\Pi_+ f)}(\xi) = \mathds{1}_{\xi \geq 0} \widehat{f}(\xi)$  on the Fourier side. We use 
$$
\Pi_- := \id - \Pi_+
$$
 to denote projection onto the orthogonal complement 
$$
L^2_-(\R; \VV) := (L^2_+(\R;\VV))^\perp = \{ f \in L^2(\R;\VV) \mid \mathrm{supp} \, \widehat{f}(\xi) \subset (-\infty,0] \}\, .
$$

From standard Paley--Wiener theory we recall that elements $f \in L^2_+(\R; \VV)$ can be naturally identified with holomorphic functions defined on the complex upper half-plane $\C_+$ such that
$$
L^2_+(\R; \VV) \cong\left \{ f \in \mathrm{Hol}(\C_+; \VV) \mid \sup_{y >0} \int_{\R} |f(x+ \ii y)|_\VV ^2 \, dx < +\infty \, \right \} \,  ,
$$ 
where $|\cdot|_\VV$ denotes the norm on $\VV$. Throughout this paper, we will freely make use of this fact and we thus regard elements $f \in L^2_+(\R; \VV)$ as holomorphic functions $f=f(z)$ with $z \in \C_+$. We use $\Hil = -\ii \Pi_+ + \ii \Pi_-$ to denote the {\em Hilbert transform} on $L^2(\R;\VV)$, which can also be written as the singular integral operator
$$
(\Hil f)(x) = \frac{1}{\pi} \mathrm{p.v.} \int_\R \frac{f(y)}{x-y} \, dy.
$$

\medskip
For a bounded matrix-valued function $\Ub \in L^\infty(\R; M_d(\C))$, we define the corresponding {\em Toeplitz operator} as
$$
T_\Ub : L^2_+(\R; \VV) \to L^2_+(\R;\VV), \quad f \mapsto T_\Ub f := \Pi_+(\Ub f) \, .
$$
Likewise, the corresponding {\em Hankel operator} is given by
$$
H_\Ub : L^2_+(\R; \VV) \to L^2_-(\R; \VV), \quad f \mapsto H_{\Ub} f := \Pi_-(\Ub f ) \, .
$$
We remark that we adapt the definition of $H_\Ub$ from Peller's book \cite{Pe-03}; another (equivalent) definition of Hankel operators can be achieved by anti-linear operators (see e.\,g.~\cite{GePu-24}). However, for studying the Lax pair structure for (HWM$_d$), we have found it more convenient to use the present convention for $H_\Ub$.

A central fact about Hankel operators used in this paper is  {\em Kronecker's theorem}, which relates the rationality of the symbol $\Ub$ with the property that $H_{\Ub}$ has finite rank. We refer the reader to Section \ref{sec:spectral} for details.  Furthermore, we remark that $H_\Ub$ is {\em Hilbert-Schmidt} if and only if $\Ub \in \dot{H}^{\frac 1 2}$; see again Section \ref{sec:spectral} for a detailed discussion.

\subsection*{The Operators $X^*$, $X$, and $I_+$}

On the Hardy space $L^2_+(\R; \VV)$, we recall that the {\em adjoint Lax--Beurling semigroup} $\{ S(\eta)^* \}_{\eta \geq 0}$ is given by
$$
(S(\eta)^* f)(x) = \Pi_+( e^{-\ii \eta x} f(x)) \quad \mbox{for $f \in L^2_+(\R; \VV)$ and $\eta \geq 0$} \, ,
$$
which corresponds to the contraction semigroup of left shifts on $L^2_+(\R; \VV)$. We remark that $S(\eta)^* = e^{-\ii \eta X^*}$, where its generator $X^*$ given by the unbounded operator 
$$
\widehat{(X^* f)}(\xi) = \ii \frac{d}{d \xi} \widehat{f}(\xi) \mathds{1}_{\xi \geq 0} 
$$
with the operator domain
$$
\dom(X^*) = \big \{ f \in L^2_+(\R; \VV) \mid \frac{d \widehat{f}}{d\xi} \in L^2(\R_+; \VV) \big \} \, . 
$$
It is straightforward to check that all {\em rational} functions $f\in L^2_+(\R; \VV)$ belong to $\dom(X^*)$. For $z_0 \in \C_+$, the action of the resolvent $(X^*-z_0)^{-1}$ is easily found to be
$$
((X^*-z_0)^{-1} f)(z) = \frac{f(z)-f(z_0)}{z-z_0} \, ,
$$
for all $f \in L^2_+(\R; \VV)$. We remark that $X^*$ is the adjoint of the unbounded operator 
$$
(X f)(x)= x f
$$
corresponding to multiplication with $x \in \R$ and its operator domain is given by
\begin{align*}
\dom(X) & =  \big \{  f \in L^2_+(\R; \VV) \mid x f \in L^2(\R; \VV) \big \} \\
& = \big \{  f \in L^2_+(\R; \VV) \mid \mbox{$\frac{d \widehat{f}}{d\xi} \in L^2(\R_+; \VV)$ and $\widehat{f}(0) = 0$} \big \} \, . 
\end{align*}
Note that $X$ is the generator of the {\em Lax--Beurling semigroup} $\{ S(\eta) \}_{\eta \geq 0}$ corresponding to right shifts on $L^2_+(\R; \VV)$, i.\,e., we have
$$
(S(\eta) f)(x) = e^{\ii \eta x} f(x) \quad \mbox{for $f \in L^2_+(\R; \VV)$ and $\eta \geq 0$} \, .
$$
We will sometimes use the notation $S(\eta) = e^{\ii \eta X}$. Note that the strict inclusion $\dom(X) \subsetneq \dom(X^*)$ holds, e.\,g., the rational function $\frac{1}{x+\ii} \in \dom(X^*)$ does not belong to $\dom(X)$. Further details on the generators $X^*$ and $X$ can be found in \cite{GePu-24} in the scalar-valued case when $\VV$ is replaced by $\C$, but  the necessary adaptations to our setting are elementary.

\medskip
In addition to the generator $X^*$, another important operator in our analysis is given by the (unbounded) linear operator
$$
I_+ : \dom(X^*) \subset L^2_+(\R; \VV) \to \VV, \quad f \mapsto I_+(f):= \widehat{f}(0^+) = \lim_{\xi \to 0^+} \widehat{f}(\xi) \, .
$$
Note that the definition of $I_+$ as the one-sided limit of $\widehat{f}(\xi)$ as $\xi \to 0^+$ makes sense for any $f \in \dom(X^*)$ by the standard trace theorem for Sobolev functions in $H^1(\R_+)$. An alternative and useful expression for the action of $I_+$ is found by using the {\em approximate identity} $\chi_\eps$ with
$$
\chi_\eps(x) := \frac{1}{1- \ii \eps x} \in L^2_+(\R; \C) \quad \mbox{for $\eps > 0$} \, .
$$
Let $\mathbf{v} \in \VV$ be a fixed vector. By Plancherel's identity, we have
$$
\lim_{\eps \to 0} \langle f, \mathbf{v} \chi_\eps \rangle = \lim_{\eps \to 0} \frac{1}{\eps} \int_0^\infty \langle \widehat{f}(\xi), \mathbf{v}    \rangle_\VV \,  \eu^{-\xi/\eps} \, d \xi = \langle \widehat{f}(0^+), \mathbf{v} \rangle_\VV = \langle I_+(f), \vv \rangle_\VV \, .
$$
Thus, for any orthonormal basis $(\vv_1, \ldots, \vv_N)$ in $\VV$ with $N = \dim \VV$, we  obtain
\be \label{eq:I_ident}
I_+(f) = \lim_{\eps \to 0} \sum_{k=1}^{N}   \langle f, \vv_k \chi_\eps \rangle \vv_k \, .
\ee

For later use, we also record the following formula
\be \label{eq:G_I}
\Im  \langle X^* f, f \rangle = -\frac{1}{4 \pi} | I_+(f)|^2_\VV \quad \mbox{for $f \in \dom(X^*)$} \, ,
\ee
which is a simple consequence from Plancherel's identity and integration by parts. 

Finally, we record another elementary fact involving the operators $I_+$ and $X^*$ as follows. Let $f \in L^2_+(\R, \VV)$ be given. As before, we suppose that $(\vv_1, \ldots, \vv_N)$ is an orthonormal basis of $\VV$ with $N = \dim \VV$. Thus we can write 
$$
f(x) = \sum_{k=1}^N f_k(x) \vv_k
$$ 
with $f_k(x) = \langle f(x), \vv_k \rangle_{\VV} \in L^2_+(\R; \C)$ for $k=1, \ldots, N$. Since $\widehat{f}(\xi)= \sum_{k=1}^N \widehat{f}_k(\xi) \vv_k$, we find
\begin{align*}
\widehat{f}(\xi)  = \sum_{k=1}^N \lim_{\eps \to 0} \left ( \int_\R \eu^{-\ii  x \xi} \frac{ f_k(x) }{1+\ii \eps x} \, dx \right ) \vv_k =  \sum_{k=1}^N \lim_{\eps \to 0} \langle S(\xi)^* f,  \vv_k\chi_\eps \rangle \, \vv_k \, .
\end{align*}
By taking the inverse Fourier transform, we obtain, for any $z \in \C_+$, that
\begin{align*}
f(z) & = \frac{1}{2 \pi} \int_0^{\infty} \eu^{\ii z \xi}   \left ( \sum_{k=1}^N \lim_{\eps \to 0} \langle S(\xi)^* f,  \vv_k\chi_\eps \rangle \, \vv_k \right ) d \xi \\
& = \frac{1}{2 \pi}  \lim_{\eps \to 0} \sum_{k=1}^N \left ( \int_0^{\infty}    \langle e^{\ii z \xi -\ii \xi X^*}  f,  \vv_k\chi_\eps \rangle \, d \xi \right ) \vv_k  \\
& = \frac{1}{2 \pi \ii} \lim_{\eps \to 0}  \sum_{k=1}^N \left \langle (X^*- z \id)^{-1} f, \vv_k \chi_\eps \right \rangle \vv_k \, .
\end{align*}
In view of \eqref{eq:I_ident}, we therefore deduce the identity
\be \label{eq:reproduce}
\boxed{f(z) = \frac{1}{2 \pi \ii} I_+[(X^*- z \id)^{-1} f] }
\ee
which is valid for any $f \in L^2_+(\R; \VV)$ and $z \in \C_+$.

%%%%%%%%%%%%% Section Lax Pair Structure

\section{Lax Pair Structure}

\label{sec:lax}

In this section, we will largely extend the results from \cite{GeLe-18}, where a Lax pair structure for (HWM) was discovered. In fact, we will consider the generalized matrix-valued equation (HWM$_d$) in this section.

 Let $d \geq 2$ and $0 \leq k \leq d$ be fixed integers. We consider solutions $\Ub : [0,T] \times \R \to M_d(\C)$ to the initial-value problem for the generalized matrix-valued half-wave maps equation which is given by 
\be \tag{HWM$_d$} \label{eq:hwm_d}
\pt_t \Ub = -\frac{\ii}{2} [ \Ub, |D| \Ub ], \quad \Ub(0) = \Ub_0 \in H^s(\R; \Gr_k(\C^d)) \, .
\ee
For local well-posedness of \eqref{eq:hwm_d} with initial data in the inhomogeneous Sobolev-type spaces $H^s_\bullet(\R; \Gr_k(\C^d))$ with $s > \frac{3}{2}$, we refer the reader to Section \ref{sec:gwp} below. Note that in \eqref{eq:hwm_d} we use $[X,Y] \equiv XY-YX$ to denote the commutator of $d \times d$-matrices and the operator $|D|$ is supposed to act on each component of the matrix-valued function $\Ub$.

We introduce some notation as follows. We recall that $\VV$ either denotes $\C^d$ or $M_d(\C)$, equipped with their natural inner products and norms.  For a bounded matrix-valued function $\mathbf{F} \in L^\infty(\R; M_d(\C))$, we let $\mu_{\mathbf{F}}$ denote the corresponding multiplication operator acting on $L^2(\R; \VV)$, i.\,e., we set
$$
(\mu_\mathbf{F} f)(x) = \mathbf{F}(x) f(x) \, .
$$ 
This distinction between $\mathbf{F}$ and its multiplication operator $\mu_{\mathbf{F}}$ will be needed for better clarity in this section.\footnote{Further below, we shall omit this distinction between $\mathbf{F} \in L^\infty(\R; M_d(\C))$ and its corresponding multiplication operator $\mu_{\mathbf{F}}$ acting on $L^2(\R; \VV)$.}  

Given a map $\Ub : [0,T] \times \R \to \Gr_k(\C^d)$ and some time $t \in [0,T]$, we denote the corresponding (bounded) multiplication operator by
$$
\mu_{\Ub(t)} : L^2(\R, \VV) \to L^2(\R, \VV), \quad f \mapsto \mu_{\Ub(t)} f \, .
$$ 
Since $\Ub(t,x)^* = \Ub(t,x)$ for a.\,e.~$x \in \R$, we readily see that $\mu_{\Ub(t)}= \mu_{\Ub(t)}^*$ is self-adjoint.

We have the following general result about \eqref{eq:hwm_d} that establishes a general Lax pair structure.

\begin{lem}[Lax equation] \label{lem:lax}
Let $s > \frac{3}{2}$ and suppose $\Ub \in C([0,T], H^s_\bullet(\R; \Gr_k(\C^d))$ is a solution of \eqref{eq:hwm_d}. Then for any operator $L_{\Ub(t)} \in \{ \mu_{\Ub(t)}, \Pi_+, \Pi_- \}$, it holds
$$
\frac{d}{dt} L_{\Ub(t)} = [B_{\Ub(t)}, L_{\Ub(t)}] \, ,
$$
with the operator 
$$
B_\Ub = -\frac{\ii}{2} ( \mu_{\Ub} \circ \Ds + \Ds \circ \mu_{\Ub}) + \frac{\ii}{2} \mu_{\Ds \Ub} \, .
$$
\end{lem}

\begin{remarks*}
1) From the assumed regularity of $\Ub=\Ub(t,x)$ above, we readily infer that the pseudo-differential operator $B_\Ub$ is of order one with operator domain $\dom(B_\Ub) = H^1(\R; \VV)$, which is found to be essentially skew-adjoint, i.\,e., there exists a unique skew-adjoint extension with $B_\Ub^* = -B_\Ub$. See Appendix \ref{app:LWP} for details.

2) The fact $\mu_{\Ub(t)}$ together with orthogonal projections $\Pi_{\pm}$ are Lax operators for the same $B_\Ub$ will allow us to restrict the Lax structure to the Hardy space $L^2_+(\R; \VV)$ involving Toeplitz and Hankel operators; see below for more details.  
\end{remarks*}

\begin{proof}
We divide the proof into the following cases.

\medskip
\textbf{Case: $L=\mu_\Ub$.} Using \eqref{eq:hwm_d}, we directly find
\be
\pt_t \mu_\Ub = -\frac{\ii}{2} [\mu_\Ub, \mu_ {\Ds \Ub}] = \frac{\ii}{2} [\mu_{\Ds \Ub}, \mu_\Ub] \, .
\ee 
In view of the expression for $B_\Ub$, it remains to show that
\be
[\mu_\Ub \circ \Ds + \Ds \circ \mu_\Ub, \mu_\Ub]  = 0 \, .
\ee
Indeed, by using that $(\mu_\Ub)^2 = \mu_{\Ub^2} = \id$ since $\Ub^2= \mathds{1}_d$, we readily check that
\begin{align*}
[\mu_\Ub \circ  \Ds + \Ds \circ \mu_\Ub, \mu_\Ub] & = (\mu_\Ub \circ \Ds + \Ds \circ \mu_\Ub) \circ \mu_\Ub - \mu_\Ub \circ (\mu_\Ub \circ \Ds + \Ds \circ \mu_\Ub) \\
& =\mu_\Ub \circ \Ds \circ \mu_\Ub + \Ds - \Ds - \mu_\Ub \circ  \Ds \circ \mu_\Ub = 0 \, .
\end{align*}

\medskip
\textbf{Case: $L=\Pi_{\pm}$}. Here it is convenient to show that the Hilbert transform $\Hil$ is a Lax operator for $B_\Ub$. The claim then readily follows for $\Pi_{\pm}=\frac 1 2 (\id  \mp \ii \Hil)$, since $\id$ commutes with any operator. Since $\frac{d}{dt} \Hil \equiv 0$, we need to show that
$$
[B_\Ub,\Hil] = 0.
$$
From the well-known product identity 
$$
\Hil(fg) = (\Hil f) g + f (\Hil g) + \Hil(\Hil f \Hil g) 
$$ 
and using that $\Hil \Ds=-\pt_x$, we readily find that
$$
[\Hil,\mu_{\Ds \Ub}] = -\mu_{\pt_x \Ub} - \Hil \mu_{\pt_x \Ub} \Hil.
$$
Hence we get
\begin{align*}
[\Hil, \Ds \circ \mu_\Ub + \mu_\Ub \circ \Ds ] & = \Hil \circ ( \Ds \circ \mu_\Ub + \mu_\Ub \circ  \Ds ) - (\Ds \circ \mu_\Ub + \mu_\Ub \circ \Ds) \circ \Hil \\
& = -\pt_x \circ \mu_\Ub + \Hil \circ \mu_\Ub \circ \Hil \pt_x - \Hil \pt_x \circ  \mu_\Ub \circ \Hil + \mu_\Ub \circ \pt_x \\
& = \mu_\Ub \circ \pt_x - \pt_x \circ \mu_\Ub + \Hil  \circ \mu_\Ub \circ \pt_x \Hil - \Hil \pt_x \circ  \mu_\Ub \circ \Hil \\
& = [\mu_\Ub,\pt_x] + \Hil [\mu_\Ub,\pt_x] \Hil \\
& = -\mu_{\pt_x \Ub} - \Hil \circ  \mu_{\pt_x \Ub} \circ \Hil.
\end{align*}
Therefore, we find
\begin{align*}
[\Hil,B_\Ub] & = \frac{\ii}{2} [\Hil, \Ds \circ \mu_\Ub + \mu_\Ub \circ \Ds ] - \frac{\ii}{2} [\Hil, \mu_{\Ds \Ub}] \\
& = \frac{\ii}{2} (-\mu_{\pt_x \Ub} - \Hil  \circ \mu_{\pt_x \Ub} \circ \Hil ) - \frac{\ii}{2} ( -\mu_{\pt_x \Ub} - \Hil \circ \mu_{\pt_x \Ub} \circ \Hil ) = 0 \, .
\end{align*}
This completes the proof of Lemma \ref{lem:lax}.
\end{proof}

From the Leibniz rule for commutators $[X,YZ]=Y[X,Z] + [X,Y]Z$ and the corresponding rule for derivatives $\frac{d}{dt} (XY) = \dot{X}Y+ X \dot{Y}$,  we immediately observe from Lemma \ref{lem:lax} that {\em all finite linear combinations of products} involving the operators $\{ \mu_\Ub, \Pi_+, \Pi_- \}$ are Lax operators too. For instance, in view of $\Hil = -\ii \Pi_+ + \ii \Pi_-$,  we recover the following Lax operator of commutator-type with
$$
L_{\Ub} = [\Hil, \mu_\Ub]  = \Hil \mu_\Ub - \mu_\Ub \Hil  \, ,
$$
which was already found in \cite{GeLe-18}. By taking traces of powers of $L_\Ub$, we obtain the  conserved quantities
$$
\Tr(|L_{\Ub(t)}|^p) = \mbox{const.} \quad \mbox{for $0 < p < \infty$}.
$$
Thus, by adapting Peller's theorem, we obtain the a-priori bounds
$$
\Tr(|L_{\Ub(t)}|^p) \sim_p \| \uu(t) \|_{\dot{B}^{1/p}_{p}}^p \sim \| \uu(0) \|_{\dot{B}^{1/p}_p}^p 
$$
for the homogeneous Besov-type norms $\| \cdot \|_{\dot{B}^{1/p}_p}$ for solutions of (HWM$_d$).\footnote{For $0 < p < 1$, we only have that $\| \cdot \|_{\dot{B}^{1/p}_p}$ is a quasi-semi-norm, since the triangle inequality fails in this case. In our analysis, we only need the case $p=2$.} However, these a-priori bounds are  {\em not known} to provide sufficient control to deduce global-in-time existence of solutions.

In order to further exploit the Lax pair structure attached to (HWM$_d$), we make the following observation involving operator analysis on Hardy spaces. Notice that, for a bounded matrix-valued function $\mathbf{F} \in L^\infty(\R; M_d(\C))$, that the corresponding Toeplitz and Hankel operators with symbol $\Fb$ can be written as $T_{\mathbf{F}} f = \Pi_+ (\mu_\mathbf{F} f)$ and $H_{\mathbf{F}} f = \Pi_- (\mu_{\mathbf{F}} f)$, using $\mu_{\mathbf{F}}$ for the corresponding multiplication with symbol $\mathbf{F}$. Now, by using Lemma \ref{lem:lax} together with $T_{\Ub} = \Pi_+ \mu_{\Ub} \Pi_+$ and $[B_\Ub, \Pi_{\pm}] \equiv 0$ (by Lemma \ref{lem:lax} too), we can easily deduce the following fact.

\begin{cor}[Toeplitz Lax Structure] \label{cor:toep_hank_lax}
For $\Ub=\Ub(t,x)$ as in Lemma \ref{lem:lax}, we have the Lax equation
$$
\frac{d}{dt} T_{\Ub(t)} = \left [ B^+_{\Ub(t)}, T_{\Ub(t)} \right ] \, .
$$
Here $B^+_{\Ub} = \Pi_+ B_{\Ub} \Pi_+$ is the compression of $B_\Ub$ onto $L^2_+(\R;\VV)$ which is given by
$$
B^+_\Ub = -\frac{\ii}{2} ( T_{\Ub}  \circ D + D \circ T_{\Ub} ) + \frac{\ii}{2} T_{|D| \Ub}  
$$
with $D=-\ii \pt_x$.
\end{cor}

\begin{remarks*}
1) Note that the compressed operator $B^+_\Ub$ is a differential operator as $|D| f = Df  = -\ii \pt_x f$ for $f \in (H^1 \cap L^2_+)(\R; \VV)$. 

2) For $t \in [0,T]$, let $\UU(t) : L^2_+(\R; \VV) \to L^2_+(\R; \VV)$ denote the unitary operator generated by the skew-adjoint operator $B^+_{\Ub(t)}$ so that
\begin{equation} \label{eq:ode_B}
\frac{d}{dt} \UU(t) = B^+_{\Ub(t)} \UU(t) \quad \mbox{for $t \in [0,T]$}, \quad \UU(0) = \id \, .
\end{equation}
For existence and uniqueness of this operator-valued initial-value problem, we refer to Appendix \ref{app:LWP}. As a direct consequence of Corollary \ref{cor:toep_hank_lax}, we find that $T_{\Ub(t)}$ and $T_{\Ub(0)}$ are given by unitary conjugation:
$$
T_{\Ub(t)} = \UU(t) T_{\Ub(0)} \UU(t)^* \quad \mbox{for $t \in [0,T]$} \, .
$$
In particular, we obtain the invariance of the spectrum $\sigma(T_{\Ub(t)}) = \sigma(T_{\Ub(0)})$ for $t \in [0,T]$.

3) Of course, the Hankel operator $H_{\Ub(t)}$ also satisfies a corresponding Lax equation with $B_{\Ub(t)}^+$ replaced by the ``twisted'' compressed operator $\Pi_- B_{\Ub(t)} \Pi_+$. But in what follows we shall only work with the Lax equation for $T_{\Ub(t)}$, which allow us to conclude all the necessary facts for our arguments developed below.
\end{remarks*}

For later use, we record the following commutator relations, where we remind the reader that we occasionally use $A.B$ to denote matrix product $AB$ on $M_d(\C)$ for better readability. 

\begin{lem} \label{lem:commutator}
Let $\Ub = \Ub_\infty + \Vb \in M_d(\C) \oplus (L^\infty \cap L^2)(\R; M_d(\C))$. Then, for every $f \in \dom(X^*)$, we have $T_{\Ub} f \in \dom(X^*)$ and
$$
[X^*, T_{\Ub}] f = \frac{\ii}{2 \pi} \Pi_+ \Vb.I_+(f) \, . 
$$ 
Moreover, it holds that
$$
[X^*, T_{\Ub}^2] f = \frac{\ii}{2 \pi} T_{\Ub}( \Pi_+ \Vb.I_+(f)) + \frac{\ii}{2 \pi} \Pi_+ \Vb.I_+(T_{\Ub} f) \, . 
$$
\end{lem}

\begin{proof}
First, we note that $[X^*, T_{\Ub_\infty}] = 0$, since $\Ub_\infty \in M_d(\C)$ is a constant matrix. Also, we evidently have that $T_{\Ub_\infty} f \in \dom(X^*)$ whenever $f \in \dom(X^*)$.

Thus it remains to discuss the commutator $[X^*, T_{\Ub}] = [X^*, T_{\Vb}]$ with $\Vb \in (L^\infty \cap L^2)(\R; M_d(\C))$. Indeed, by adapting the proof in \cite{GePu-24}[Lemma 2.3] to the matrix-valued symbol $\Vb$, we find that
$$
[X^*,T_{\Vb}] f = \frac{\ii}{2 \pi} \Pi_+ \Vb.I_+(f)
$$
noticing that $T_{\Vb} f \in \dom(X^*)$ for any $f \in \dom(X^*)$.  We leave the details to the reader.

The  commutator identity  for $[X^*, T_{\Ub}^2]$ simply follows from the first identity and the fact that $[A,BC] = B[A,C] + [A,B]C$.
\end{proof}

%%%%%%%%%%%%%%% Section Spectral Analysis

\section{Spectral Analysis of $T_{\Ub}$}

\label{sec:spectral}

As in the previous sections, we let $\VV$ either stand for the Hilbert spaces $\C^d$ or $M_d(\C)$ for some given integer $d \geq 2$. The aim of this section is to derive some fundamental spectral properties of the Toeplitz operator 
$$
T_\Ub : L^2_+(\R; \VV) \to L^2_+(\R; \VV), \quad f \mapsto T_\Ub f= \Pi_+(\Ub f) \, .
$$ 
Throughout the following we will always assume that 
$$
\Ub(x) = \Ub_\infty + \Vb(x) \in H^{\frac 1 2}_\bullet(\R; M_d(\C))  \equiv M_d(\C) \oplus H^{\frac{1}{2}}(\R; M_d(\C))
$$ 
together with the pointwise algebraic constraints
$$
\Ub(x)^* = \Ub(x) \quad \mbox{and} \quad \Ub(x)^2 = \mathds{1}_d \quad \mbox{for a.\,e.~$x \in \R$} \, .
$$ 
As a consequence,  we see that the corresponding Toeplitz operator $T_{\Ub}=T_{\Ub}^*$ is self-adjoint and bounded with operator norm $\| T_{\Ub} \| \leq 1$. Moreover, we readily check that the following properties hold.
\begin{enumerate}
\item[(i)] $\Ub_\infty^* = \Ub_\infty$ and $\Ub_\infty^2 = \mathds{1}_d$.
\item[(ii)] $\Ub, \Vb \in L^\infty(\R; M_d(\C))$ with $\|\Ub \|_{L^\infty} = 1$ and $\|\Vb \|_{L^\infty} \leq \|\Ub_\infty\|_{L^\infty} + \| \Ub \|_{L^\infty} = 2$.
\item[(iii)] $\Vb = \Vb_{+} + \Vb_{+}^*$ with $\Vb_{+} = \Pi_+ \Vb$ and $\Vb_{+}^* = \Pi_- \Vb$.
\end{enumerate}

\subsection*{Fredholm Property and Invariant Subspaces}

Recall that 
$$
H_\Ub: L^2_+(\R; \VV) \to L^2_-(\R; \VV), \quad H_{\Ub} f= \Pi_-(\Ub f)
$$ 
denotes the corresponding (block) Hankel operator with matrix-valued symbol $\Ub$. For later use, we remark that the adjoint Hankel operator is given by
$$
H_\Ub^* : L^2_-(\R; \VV) \to L^2_+(\R;\VV), \quad H_\Ub^* f = \Pi_+(\Ub f) \, .
$$
\begin{remark*}
Here we used the fact that $\Ub(x)^* =  \Ub(x)$ almost everywhere. For general matrix-valued symbols $\mathbf{F}\in L^\infty(\R; M_d(\C))$, the adjoint Hankel operator is $H_{\mathbf{F}}^* f = \Pi_+(\mathbf{F}^* f)$ for $f\in L^2_-(\R;\VV)$.
\end{remark*}

We have the following general fact, for matrix-valued symbols $\Ub$ satisfying the assumptions stated above.

\begin{lem}[Key Identity and Fredholmness] \label{lem:key} We have the identity
$$
T_{\Ub}^2 = \id - K_{\Ub} \quad \mbox{on $L^2_+(\R; \VV)$} \, ,
$$
where the self-adjoint operator 
$$
K_\Ub := H^*_{\Ub} H_{\Ub} : L^2_+(\R; \VV) \to L^2_+(\R; \VV)
$$
satisfies $0 \leq K_{\Ub} \leq \id$ and it is trace-class with 
$$
\Tr(K_\Ub) = \Tr(H^*_\Ub H_\Ub) = \mathrm{const.} \cdot \| \Ub \|_{\dot{H}^{\frac 1 2}}^2  \, .
$$
Moreover, the Toeplitz operator $T_{\Ub}$ is Fredholm with index 0.
\end{lem}

\begin{proof}
Let us consider the case $\VV= \C^d$, where we remark that the proof for $\VV = M_d(\C)$ is analogous. Suppose that $f \in L^2_+(\R; \C^d)$ is given. Using that $\mu_\Ub^2 = \mu_{\Ub^2} = \mu_{\mathds{1}_d} = \id$ holds on $L^2_+(\R; \C^d)$ and $\Ub^*(x)=\Ub(x)$ for a.\,e.~$x \in \R$, we observe that
\begin{align*}
T_\Ub ( T_\Ub f) & = \Pi_+(\Ub \Pi_+(\Ub f)) = \Pi_+ ({\Ub} (\id - \Pi_-) \Ub f) \\ 
& = \Pi_+ f - \Pi_+ ( {\Ub} (\Pi_- \Ub f) ) = f - H_{\Ub}^* H_\Ub f,
\end{align*}
since we trivially have that $\Pi_+ f = f$ for $f \in L^2_+(\R; \C^d)$. This proves the claimed identity.

Consider now the bounded and  self-adjoint operator 
$$
K_{\Ub} := H_{\Ub}^* H_{\Ub} : L^2_+(\R; \C^d) \to L^2_+(\R; \C^d) \, .
$$ 
Clearly, we have that $K_{\Ub} \geq 0$ is non-negative. Also, we notice that $\| K_{\Ub} \| \leq \| H_{\Ub} \|^2 \leq \| \Ub \|_{L^\infty} = 1$, which shows that $K_{\Ub} \leq \id$ holds in the sense of operators. Next, we observe that $K_{\Ub}$ is trace-class with
\be \label{eq:K_trace}
\Tr(K_{\Ub}) = \Tr(H_\Ub^* H_\Ub) = \| H_{\Ub} \|_{HS}^2 = c  \cdot \| \Ub \|_{\dot{H}^{\frac 1 2}}^2 \, ,
\ee 
where $c>0$ is some numerical constant. Here $\| A \|_{HS}$ denotes the Hilbert--Schmidt norm of a bounded operator $A : H_1 \to H_2$ with separable Hilbert spaces $H_1, H_2$, i.\,e., we have 
$$
\| A \|_{HS}^2  = \sum_{n=1}^\infty \langle  A e_n, A e_n \rangle_{H_2} \, ,
$$  
where $(e_n)_{n \in \N}$ is an arbitrary ONB of $H_1$. For the last equation in \eqref{eq:K_trace}, we give an elementary proof taken from \cite{GeLe-18}. Using the orthogonal decomposition $L^2(\R;\C^d) = L^2_+(\R; \C^d) \oplus L^2_-(\R; \C^d)$, we consider the commutator of $\Ub$ viewed as multiplication operator on $L^2(\T; \C^d)$ with the Hilbert transform $\Hil$. This can be written as a $2 \times 2$-matrix of operators such that
$$
 [\Hil, \Ub] = \left ( \begin{array}{cc} 0 & -\ii \Pi_+ \Ub \Pi_- \\ \ii \Pi_- \Ub \Pi_+ & 0 \end{array} \right ) : L^2_+(\R; \C^d) \oplus L^2_-(\R; \C^d) \to L^2_+(\R; \C^d) \oplus L^2_-(\R;\C^d).
$$
On the other hand, from the singular integral formula for $\Hil$, we easily see that $[\Hil, \Ub]$ has the integral kernel $h_\Ub(x,y) = \frac{1}{\pi} \frac{\Ub(x)-\Ub(y)}{x-y} \in L^2(\R \times \R; M_d(\C))$. Hence its Hilbert-Schmidt norm as an operator acting on $L^2(\R; \VV)$ can be directly computed as
$$
\| [\Hil, \Ub] \|_{HS}^2 = \| h_\Ub \|_{L^2(\R \times \R; M_d(\C))}^2 = \frac{1}{\pi^2} \int_{\R} \! \int_{\R} \frac{|\Ub(x)- \Ub(y)|_F^2}{|x-y|^2} \, dx \,dy \, ,
$$
where $| \cdot |_F$ denotes the Frobenius norm of matrices in $M_d(\C)$. Next, by using that $\Ub= \Ub^*$ holds, we see that $\| h_\Ub \|_{HS}^2 = \| \Pi_+ \Ub \Pi_- \|_{HS}^2 + \| \Pi_- \Ub \Pi_+ \|_{HS}^2 =2 \| H_\Ub \|_{HS}^2$. Recalling the formula \eqref{eq:H_half}, we deduce that the last equation in \eqref{eq:K_trace} holds.

It remains to prove that $T_{\Ub}$ is Fredholm with index 0. Indeed, we readily see that $T_\Ub$ is Fredholm since $T_\Ub$ is invertible modulo compact operators, which directly follows from the identity $T_{\Ub} T_\Ub = T_\Ub^2 = \id - K_\Ub$. Since $T_{\Ub}$ is self-adjoint, its Fredholm index must be 0.
\end{proof}

In view of the general identity established in Lemma \ref{lem:key}, it is natural to introduce the following closed subspaces
\be
\boxed {\Hfr_0 := \ker(K_{\Ub}) \quad \mbox{and} \quad \Hfr_1 := \ov{\ran(K_{\Ub})}  }
\ee
which yields the orthogonal decomposition
\be
\boxed{L^2_+(\R; \VV) = \Hfr_0 \oplus \Hfr_1 }
\ee
As a direct consequence,  we obtain a decomposition of the Toeplitz operator $T_{\Ub}$ into invariant subspaces in the spirit of the celebrated {\em Sz.--Nagy--Foia\textcommabelow{s} decomposition} for contractions on Hilbert spaces  \cite{SzFoBeKe-10} (also referred to as {\em Langer's lemma} in \cite{La-61}), which in turn is a generalization of the well-known {\em Wold decomposition} for isometries on Hilbert spaces. Note that $T_{\Ub}$ is a contraction, because its operator norm satisfies $\|T_{\Ub}\| \leq \| \Ub \|_{L^\infty} = 1$, where in fact we have equality in view of the identity in Lemma \ref{lem:key} above.

\begin{prop} \label{prop:nagy}
The subspaces $\Hfr_0$ and $\Hfr_1$ are invariant under $T_{\Ub}$. Moreover, the restriction $T_{\Ub} |_{\Hfr_0}$ is unitary, whereas the restriction $T_{\Ub} |_{\Hfr_1}$ is completely non-unitary (c.n.u.), i.\,e., there is no non-trivial invariant subspace in $\Hfr_1$ on which $T_{\Ub}$ is unitary.
\end{prop} 

\begin{proof}
Since the operator $K_\Ub : L^2_+(\R; \VV) \to L^2_+(\R; \VV)$ with $0 \leq K_\Ub \leq 1$ is compact and self-adjoint, we can write
$$
K_\Ub = \sum_{j=1}^N \lambda_j \langle \cdot, \phi_j\rangle  \phi_j
$$
with $N = \rank(K_\Ub) \in \N \cup \{\infty\}$, $\lambda_j \in (0, 1]$ for $j=1, \ldots, N$, and the corresponding eigenvectors $(\phi_j)_{j=1}^N$ form an ONB of $\Hfr_1 = \ov{\ran(K_\Ub)}$. By the identity $T_\Ub^2 = \id - K_\Ub$ and from elementary spectral calculus for the self-adjoint restrictions $T_\Ub |_{E_{\lambda_j}}$ on the finite-dimensional subspaces $E_{\lambda_j}=\ker(K_\Ub-\lambda_j \id) \subset \Hfr_1$, we deduce that
\be \label{eq:T_U_spectral}
T_\Ub = \sum_{j=1}^N \eps_j \sqrt{1- \lambda_j} \langle \cdot, \phi_j \rangle \phi_j \quad \mbox{on $\Hfr_1$} 
\ee
with some $\eps_j \in \{ \pm 1 \}$ for $j=1, \ldots, N$. Evidently, we have that $\Hfr_1$ is invariant under $T_\Ub$ and we see that $T_\Ub |_{\Hfr_1}$ is c.\,n.\,u. Because otherwise $T_\Ub |_{\Hfr_1}$ would have an eigenvalue $\mu \in \{ \pm 1\}$ on some finite-dimensional subspace $E_{\lambda_j}$, contradicting the above explicit formula since $\lambda_j > 0$ holds.

Since $\Hfr_0 = \ker(K_\Ub) = \Hfr_1^\perp$ and by self-adjointness of $T_\Ub$, we see that $T_\Ub(\Hfr_0) \subset \Hfr_0$. Furthermore from $T_\Ub^2 = \id - K_\Ub$, we readily find
$T^2_{\Ub} |_{\Hfr_0} = \id |_{\Hfr_0}$, which implies that the self-adjoint operator $T_\Ub |_{\Hfr_0}$ is also unitary.
\end{proof}

Thanks to the formula $T_\Ub^2 = \id - K_\Ub$ and the decomposition obtained in Proposition \ref{prop:nagy}, we deduce that the spectrum of $T_\Ub$ decomposes as
$$
\sigma(T_\Ub) = \sigma_{\mathrm{e}}(T_\Ub) \sqcup \sigma_{\mathrm{d}}(T_\Ub) \, ,
$$
where the essential and discrete spectra of $T_\Ub$ are given by
$$
\sigma_{\mathrm{e}}(T_{\Ub}) = \sigma(T_{\Ub} |_{\Hfr_0})  \subset \{ \pm 1 \} 
$$
$$
\sigma_{\mathrm{d}}(T_{\Ub}) = \sigma_{\mathrm{d}}( T_{\Ub} |_{\Hfr_1}) = \{ \eps_j \sqrt{1-\lambda_j} \mid j=1,\ldots, \rank(K_\Ub) \}  \, ,
$$
with the sequences $(\eps_j)\subset \{ \pm 1 \}$ and $(\lambda_j) \subset (0, 1]$  taken from  \eqref{eq:T_U_spectral} above.

\begin{remark*}
As an aside, we remark that the property of the Toeplitz operator $T_{\Ub}$ having non-empty discrete spectrum is due to the fact its symbol $\Ub(x)$ is matrix-valued. By contrast, a classical result due to Widom \cite{Wi-64} states that any Toeplitz operator $T_\phi : L^2_+(\R; \C) \to L^2_+(\R; \C)$ with a scalar-valued symbol $\phi \in L^\infty(\R; \C)$ has a spectrum $\sigma(T_\phi)$ which must be a connected subset in $\C$, which shows that $\sigma_{\mathrm{d}}(T_\phi) = \emptyset$ in this case.
\end{remark*}

As a next step, we find some explicit elements in the invariant subspace $\Hfr_1$. The use of this fact will become clear later on when proving our global well-posedness result for (HWM$_d$). Recall our assumption that
\be \label{eq:U_ass}
\Ub = \Ub_\infty + \Vb \in M_d(\C) \oplus H^{\frac 1 2}(\R; M_d(\C)) \, .
\ee
For later use, we make the following observation. For better readability, we use $B.A$ to denote the product $BA$ of two matrices $B, A\in M_d(\C)$.

\begin{prop} \label{prop:king}
Let $\VV = M_d(\C)$. For any constant matrix $A \in M_d(\C)$, it holds that $\Pi_+ \Vb.A \in \Hfr_1$.
\end{prop}

\begin{proof}
Since $K_{\Ub} = H^*_{\Ub} H_{\Ub}$, we find that $\ker(K_{\Ub}) = \ker(H_{\Ub})$, which yields that $\Hfr_1= (\Hfr_0)^\perp = \ov{\ran(H^*_\Ub)}$.  Hence we have to show that $\Pi_+ \Vb.A \in \ov{\ran(H^*_{\Ub})}$ holds for any constant matrix $A \in M_d(\C)$. Indeed, by recalling $\chi_\eps = \frac{1}{1-\ii \eps x} \in L^2_+(\R; \C)$ for $\eps> 0$ and hence $\ov{\chi}_\eps \in L^2_-(\R; \C)$, we notice
\begin{align*}
 \lim_{\eps \to 0} H^*_{\Ub}(\ov{\chi}_\eps A) & = \lim_{\eps \to 0} \Pi_+( \Ub. \ov{\chi}_\eps A) \\
& = \lim_{\eps \to 0} \Pi_+ ( (\Ub_\infty + \Pi_+ \Vb + \Pi_- \Vb). \ov{\chi}_\eps A ) \\
& = \lim_{\eps \to 0} \Pi_+( (\Pi_+\Vb) \ov{\chi_\eps}).A = \Pi_+ \Vb.A,
\end{align*}
because of $\lim_{\eps \to 0} \Pi_+ ( f \ov{\chi}_\eps) = f$ in $L^2_+(\R; \VV)$ by dominated convergence. This shows that $\Pi_+ \Vb.A$ belongs to $\Hfr_1=\ov{\ran(H^*_{\Ub})}$.
\end{proof}

Next, by using the well-known fact that kernels of Hankel operators are invariant under the Lax--Beurling semigroup $\{ S(\eta) \}_{\eta \geq 0}$, we obtain the following result.

\begin{lem} \label{lem:S_inv}
It holds that $S(\eta)\Hfr_0 \subset \Hfr_0$ and $S(\eta)^* \Hfr_1 \subset \Hfr_1$ for all $\eta \geq 0$.
\end{lem}

\begin{proof}
Let $f \in \Hfr_0 = \ker(K_{\Ub}) = \ker(H_{\Ub})$. For any $\eta \geq 0$, we immediately observe that
\begin{align*}
H_{\Ub}(S(\eta) f) & = \Pi_-(\Ub S(\eta) f) = \Pi_- (\eu^{\ii \eta x} \Ub f) \\
& = \Pi_- (\eu^{\ii \eta x} \Pi_-(\Ub f)) = \Pi_- (\eu^{\ii \eta x} H_{\Ub}(\Ub f)) = 0 \, .
\end{align*}
Thus we find $S(\eta) f \in \Hfr_0$ for any $f \in \Hfr_0$. This proves that $S(\eta) \Hfr_0 \subset \Hfr_0$. 

Since $\Hfr_0 \perp \Hfr_1$, we directly see that $S(\eta)^* f \in \Hfr_1$ for any $f \in \Hfr_1$ and $\eta \geq 0$ with the adjoint Lax--Beurling semigroup $\{ S(\eta)^* \}_{\eta \geq 0}$ acting on $L^2_+(\R; \VV)$.
\end{proof}

\begin{remark*}
As a direct consequence of the well-known {\em Lax--Beurling theorem} (see the version in \cite{Lax-59} for a direct application to our setting) about invariant subspaces of $S(\eta)$, we can deduce the following fact: If $\Hfr_0 = \ker(K_\Ub) \neq \{ 0 \}$ is non-trivial, there exist a subspace $\VV' \subseteq \VV$ and a function 
$$
\Theta \in L^\infty_+(\R; \mathrm{End}(\VV'; \VV)) \quad \mbox{with} \quad  \Theta(x)^* \Theta(x) = \id_{\VV'} \quad \mbox{for a.\,e.~$x \in \R$}
$$
such that
$$
\Hfr_0 = \Theta L^2_+(\R; \VV') \quad \mbox{and} \quad \Hfr_1 = (\Theta L^2_+(\R; \VV'))^\perp \, .
$$
The matrix-valued function $\Theta$ is called a (left) {\em inner function} and the subspace $\Hfr_1$ is thus the {\em model space} generated by $\Theta$. However, we will not exploit this fact in the present paper.
\end{remark*}

\subsection*{Spectral Properties for Rational Data}

Recall that $\Hfr_1= \ov{\ran(K_{\Ub})}$. We have the following characterization when the subspace $\Hfr_1$ is finite-dimensional, corresponding to the fact that the compact operator $K_{\Ub} : L^2_+(\R; \VV) \to L^2_+(\R; \VV)$ has finite rank. 

\begin{lem}[Kronecker-type theorem]  \label{lem:kronecker} Let $\Ub \in L^\infty(\R; M_d(\C))$ be of the form \eqref{eq:U_ass} with $\Ub(x) = \Ub(x)^*$ for a.\,e.~$x \in \R$. Then $K_\Ub : L^2_+(\R; \VV) \to L^2_+(\R;\VV)$ has finite rank (i.\,e.~we have $\dim \Hfr_1 < \infty$) if and only if $\Ub$ is a rational function.
\end{lem}

\begin{remark}
Since $K_{\Ub}=H_{\Ub}^* H_{\Ub}= \id-T_{\Ub}^2$ is a Lax operator for (HWM$_d$), we see that rationality is preserved along the flow. For (HWM) with target $\Ss^2 \cong \Gr_1(\C^2)$, this feature was already observed in \cite{GeLe-18}.
\end{remark}

\begin{proof}
Since $\dim \ran(H_{\Ub}^* H_{\Ub}) = \dim \ran(H_{\Ub})$, it suffices to consider the Hankel operator $H_{\Ub} : L^2_+(\R; \VV) \to L^2_-(\R; \VV)$. Furthermore, since $\Ub_\infty \in M_d(\C)$ is constant, we see that $H_{\Ub_\infty}=0$ and hence $H_{\Ub}= H_{\Ub_\infty + \Vb} = H_{\Vb}$. Thus it remains to discuss $H_{\Vb}$ with the matrix-valued symbol $\Vb \in (H^{\frac 1 2} \cap L^\infty)(\R; M_d(\C))$ for the rest of the proof.

We first recall the following general Kronecker-type theorem valid for Hankel operators acting on the Hardy space $L_+^2(\T; \mathcal{H})$ on the torus $\T \cong \pt \mathbb{D}$, where $\mathcal{H}$ is a given separable complex Hilbert space (not necessarily finite-dimensional). Correspondingly, we use $\mathbb{P}_+$ and $\mathbb{P}_- = \id - \mathbb{P}_+$ to denote the Cauchy--Szeg\H{o} projections on $L^2(\T; \mathcal{E})$; see \cite{Pe-03} for a general background. As usual, we use $\mathcal{B}(\mathcal{H}, \mathcal{K})$ to denote the Banach space of bounded linear operators from $\mathcal{H}$ to another complex Hilbert space $\mathcal{K}$. From \cite{Pe-03}[Chapter 2, Theorem 5.3] we directly deduce the following result. 

\begin{thm}[Kronecker's theorem on $L^2_+(\T; \mathcal{H})$]
Let $\mathcal{H}, \mathcal{K}$ be separable complex Hilbert spaces and assume $\Phi \in L^\infty(\T; \mathcal{B}(\mathcal{H}, \mathcal{K}))$. Define the Hankel operator $H_\Phi : L^2_+(\T; \mathcal{H}) \to L^2_-(\T; \mathcal{K})$ by $H_\Phi f = \mathbb{P}_-(\Phi f)$. Then $\rank \, H_\Phi < \infty$ if and only if $\mathbb{P}_- \Phi : \T \to \mathcal{B}(\mathcal{H}, \mathcal{K})$ is a rational map of the form
$$
\mathbb{P}_- \Phi = \sum_{\lambda \in \Lambda} \sum_{n=1}^{k(\lambda)} \frac{T_{\lambda, n}}{(z- \lambda)^n} \, ,
$$
where $\Lambda$ is a finite subset in $\mathbb{D}$ and the $k(\lambda)$ are positive integers and $T_{n, \lambda} \in \mathcal{B}(\mathcal{H}, \mathcal{K}) \setminus \{ 0 \}$ are finite-rank operators, $\lambda \in \Lambda$, and $1 \leq n \leq k(\lambda)$.
\end{thm}

Let us now take the finite-dimensional spaces $\mathcal{H}=\mathcal{K} = \VV$ in the previous result with either $\VV = \C^d$ or $\VV = M_d(\C)$. For any $\Phi \in L^\infty(\T; M_d(\C))$ given, we deduce  the equivalence
$$
\mbox{$H_{\Phi} = \mathbb{P}_- \Phi \mathbb{P}_+$ has finite rank if and only if $\mathbb{P}_- \Phi \in L^\infty(\T; M_d(\C))$ is rational.}
$$
Now using the standard conformal map $\omega : \mathbb{D} \to \C_+$ with $\omega(\zeta) = \ii \frac{1 + \zeta}{1-\zeta}$, let us define the map
$$
(\UU f)(x) = \frac{1}{\sqrt{\pi}} \frac{(f \circ \omega^{-1})(x)}{x+\ii} = \frac{1}{\sqrt{\pi}} \frac{1}{x+ \ii} f \left ( \frac{x-\ii}{x+\ii} \right ) \quad \mbox{for $f \in L^2(\T; \VV)$},
$$
which is known to be unitary operator from $L^2(\T; \VV)$ to $L^2(\R; \VV)$ with the property that $\mathcal{U}(L^2_+(\T; \VV)) = L^2_+(\R; \VV)$; see \cite{Pe-03}[Appendix 2.1]. We easily verify that
$$
H_{\Phi} = \UU^* H_{\Vb} \UU \quad \mbox{with $\Phi = \Vb \circ \omega$},
$$
see e.\,g.~\cite{Pe-03}[Chapter 1, Lemma 8.3]. Since compositions with $\omega$ and $\omega^{-1}$ preserve rationality and in view of the identity $\Pi_-  \Vb = (\mathbb{P}_- ( \Vb \circ \omega)) \circ \omega^{-1}$, we deduce 
$$
\mbox{$H_{\Vb}$ has finite rank if and only if $\Pi_- \Vb : \R \to M_d(\C)$ is rational.}
$$ 
Finally, by recalling that $\Pi_+ \Vb = (\Pi_- \Vb)^*$, we conclude that $\Vb = (\Pi_- \Vb)^* + \Pi_- \Vb$ is a rational function if and only if $H_{\Vb}$ has finite rank.
\end{proof}

We now show that, for rational matrix-valued symbols $\Ub$, the subspace $\Hfr_1$ is also an invariant subspace for the unbounded operator $X^*$, which is the generator of the adjoint Lax--Beurling semigroup $\{ S(\eta)^* \}_{\eta \geq 0}=\{ \eu^{-\ii \eta X^*}  \}_{\eta \geq 0}$.

\begin{prop} \label{prop:G_invariant}
If $\Ub \in \mathcal{R}at(\R; \Gr_k(\C^d))$, then $\Hfr_1 \subset \dom(X^*)$ and $X^*(\Hfr_1) \subset \Hfr_1$.
\end{prop}

\begin{proof}
By Lemma \ref{lem:S_inv}, we recall that the adjoint Lax--Beurling semigroup $S(\eta)^*$ acts invariantly on $\Hfr_1$. Moreover, by Lemma \ref{lem:kronecker}, we know that $\dim \Hfr_1 < \infty$. By standard arguments from semigroup theory, it follows that the generator $X^*$ restricted to the finite-dimensional invariant subspace $\Hfr_1$ is bounded and thus its domain $\dom (X^*|_{\Hfr_1})$ is thus all of $\Hfr_1$. In particular, we have $\Hfr_1 \subset \dom (X^*)$ with $X^*(\Hfr_1) \subset \Hfr_1$.
\end{proof}

%%%%%%%%%%%%%%%% Section  Explicit Flow Formula

\section{Local Well-Posedness and Explicit Flow Formula} 
\label{sec:gwp}

In this section, we derive the explicit flow formula valid for (HWM$_d$) for sufficiently smooth solutions. In fact, this formula will play an essential r\^ole for obtaining the main results of this paper. Let us also remark that similar explicit flow formulae have been recently derived for other completely integrable equations which feature a Lax pair structure on Hardy spaces such as the cubic Szeg\H{o} equation \cite{GeGr-15}, the Benjamin--Ono equation \cite{Ge-23} and the Calogero--Moser derivative NLS \cite{GeLe-24,Ba-24,KiLaVi-23}.

\subsection*{Local Well-Posedness for Sufficiently Regular Data}
We start with the a result on local well-posedness for the matrix-valued (HWM$_d$) for sufficiently regular initial data of the form
\be \label{eq:Ub_0}
\Ub_0(x) = \Ub_\infty + \Vb_0(x) \in M_d(\C) \oplus H^s(\R; M_d(\C))   \, ,
\ee
satisfying the constraints
\be \label{eq:cond_u}
\Ub_0(x) = \Ub_0(x)^*, \quad \Ub_0(x)^2 = \mathds{1}_d \quad \mbox{for a.\,e.~$x \in \R$} \, .
\ee
In what follows, we will always assume that 
$$
s > \frac{3}{2} \, .
$$
In particular, the initial datum $\Ub_0 : \R \to M_d(\C)$ is of class $C^1$ by Sobolev embeddings. In view of \eqref{eq:cond_u}, we easily conclude that $\Tr(\Ub_0(x))$ can only attain integer values, whence it follows $\Tr(\Ub_0(x)) = \mbox{const}.$ on $\R$ by continuity.\footnote{That fact that $\Tr(\Ub_0(x)) = \mbox{const}.$ almost everywhere is even true for $s=1/2$, since any integer-valued map $\Tr(\Ub_0) \in \dot{H}^{\frac{1}{2}}(\R;\R)$ necessarily satisfies $\Tr(\Ub_0((x) )= \mbox{const}.$ almost everywhere; see e.\,g.~\cite{Br-02}.} As a consequence, we deduce that there exists some integer $0 \leq k \leq d$ such that
$$
\Ub_0(x) \in \Gr_k(\C^d) \quad \mbox{for $x \in \R$}.
$$
We have the following result.

\begin{lem} \label{lem:lwp}
Let $s > \frac{3}{2}$, $d \geq 2$, and assume  $\Ub_0 : \R \to M_d(\C)$ satisfies \eqref{eq:Ub_0} and \eqref{eq:cond_u}. Then, for any $R> 0$, there exists some $T=T(R) >0$ such that for every $\Ub_0=\Ub_\infty + \Vb_0$ as above with $\| \Vb_0 \|_{H^s} < R$, there exists a unique solution of \eqref{eq:hwm_d} of the form
 $$
 \Ub(t) = \Ub_\infty + \Vb(t) \in  M_d(\C) \oplus C([0,T]; H^s(\R;M_d(\C)))
 $$
 and we have $\Ub(t,x) \in \Gr_k(\C^d)$ for all $t \in [0,T]$ and $x \in \R$ with some integer $0 \leq k \leq d$.
 
Furthermore, the $H^\sigma$-regularity of $\Vb_0$ with $\sigma > s$ is propagated on the whole maximal time-interval of existence of $\Ub(t)$, and the flow map $\Vb_0 \mapsto \Vb(t)$ is continuous in the $H^\sigma$-topology.
\end{lem}

\begin{remark*}
For proving the above local well-posedness result, the Hermitian constraint in \eqref{eq:cond_u} is the relevant one. However, the second constraint in \eqref{eq:cond_u} will be essential to obtain a global well-posedness result below based on the Lax pair structure, which involves the use of both pointwise constraints stated in \eqref{eq:cond_u}. 
\end{remark*}

\begin{proof}
We postpone the detailed proof of Lemma \ref{lem:lwp} to Appendix \ref{app:LWP}. 
\end{proof}

\subsection*{Explicit Flow Formula}
Inspired by the very recent work \cite{Ge-23} on the Benjamin--Ono equation, we next derive an {\em explicit flow formula} for solutions of \eqref{eq:hwm_d} based on its Lax pair structure acting on the Hardy space. Note that, in this formula, we choose the vector space $\VV = M_d(\C)$ for the Hardy space $L^2_+(\R; \VV)$.

\begin{lem}[Explicit Flow Formula] 
\label{lem:explicit}
Let $s > \frac{3}{2}$, $d \geq 2$, and $\Ub(t) = \Ub_\infty + \Vb(t) \in  M_d(\C) \oplus C([0,T]; H^s(\R, M_d(\C))$ be as in Lemma \ref{lem:lwp} above. Then it holds that
$$
\Pi_+ \Vb(t,z) = \frac{1}{2 \pi \ii} I_+ \left [ (X^*+tT_{\Ub_0} - z \id)^{-1} \Pi_+ \Vb_0 \right ] \quad \mbox{for $z \in \C_+$ and $t \in [0,T]$.} 
$$
Here $T_{\Ub_0} : L^2_+(\R; \VV) \to L^2_+(\R; \VV)$ denotes the Toeplitz operator $T_{\Ub_0} f = \Pi_+(\Ub_0 f)$ with $\VV = M_d(\C)$.
\end{lem}

Before we turn to the proof of Lemma \ref{lem:explicit}, we need some commutator identities as follows. Recall that
$$
B_{\Ub} = -\frac{\ii}{2} ( \mu_{\Ub} |D| + |D| \mu_{\Ub} ) + \frac{\ii}{2} \mu_{|D| \Ub} \, .
$$
In fact, since we can restrict to the Hardy space $L^2_+(\R; \VV)$, it will be convenient to work with  the {\em compression of $B_{\Ub}$} to the Hardy space $L^2_+(\R; \VV)$ denoted by
$$
B^+_{\Ub} := \Pi_+ B_{\Ub} \Pi_+ = -\frac{\ii}{2} (T_{\Ub} D + D T_{\Ub}) + \frac{\ii}{2} T_{|D| \Ub} \quad \mbox{with $D=-\ii \pt_x$} \, .
$$ 
Note that $D \geq 0$ on $L^2_+(\R; \VV)$ with its operator domain $\dom(D) = H^1_+(\R; \VV)$. The Lax equation for $T_{\Ub(t)} : L^2_+(\R, \VV) \to L^2_+(\R, \VV)$ can thus be written as
$$
\frac{d}{dt} T_{\Ub(t)} = [B_{\Ub(t)}^+, T_{\Ub(t)} ] \, .
$$

We have the following key commutator identity.
\begin{prop} \label{prop:G_B}
For any $f \in \dom(X^*) \cap H^1_+(\R; \VV)$, it holds that
$$
[X^*, B^+_{\Ub}] f = T_{\Ub} f \, .
$$ 
\end{prop}

\begin{proof}
Using the fact that $[X^*, D] = \ii \, \id$ and by Lemma \ref{lem:commutator}, we calculate
\begin{align*}
[X^*, T_{\Ub} D + D T_{\Ub}] f & = [X^*, T_{\Ub}] Df + T_{\Ub} [X^*,D] f + [X^*,D] T_{\Ub} f + D [X^*,T_{\Ub}] f \\
& = \frac{\ii}{2 \pi} \Pi_+ \Vb.I_+(Df) + \ii T_{\Ub} f + \ii T_{\Ub} f + \frac{\ii}{2 \pi} D (\Pi_+ \Vb.I_+(f)) \\
& = 2 \ii T_{\Ub} f + \frac{\ii}{2 \pi} \Pi_+(D \Vb).I_+(f) \, ,
\end{align*}
where also used that $I_+(Df)=0$ holds. By applying Lemma \ref{lem:commutator} once again, 
$$
[X^*,T_{|D| \Ub}] f = \frac{\ii}{2 \pi} \Pi_+(|D| \Vb).I_+(f) = \frac{\ii}{2 \pi} \Pi_+(D \Vb).I_+(f) \, .
$$
In view of these identities, we easily conclude the claimed identity.
\end{proof}

We are now ready to turn to the proof of Lemma \ref{lem:explicit}.

\begin{proof}[Proof of Lemma \ref{lem:explicit} (Explicit Flow Formula)]
We divide the proof into the following steps. We remind the reader that we take $\VV = M_d(\C)$ in the following.

\medskip
\textbf{Step 1.} Recalling identity \eqref{eq:reproduce}, we  write
$$
\Pi_+ \Vb(t,z) = \frac{1}{2 \pi \ii} I_+((X^*-z \id)^{-1} \Pi_+\Vb(t)) \quad \mbox{for $z \in \C_+$ and $t \in [0,T]$} \, .
$$
Let $F \in M_d(\C)$ and $z \in \C_+$ be fixed from now on. We find
\begin{align*}
\langle \Pi_+ \Vb(t,z), F \rangle_\VV & = \frac{1}{2 \pi \ii} \langle I_+((X^*-z \id)^{-1} \Pi_+ \Vb(t), F \rangle_\VV \\
&= \frac{1}{2 \pi \ii} \lim_{\eps \to 0} \left \langle (X^*-z \id)^{-1} \Pi_+ \Vb(t), F \chi_\eps \right \rangle \\
& = \frac{1}{2 \pi \ii} \lim_{\eps \to 0} \left \langle \UU(t)^* (X^*- z \id)^{-1} \Pi_+ \Vb(t), \UU(t)^* (F \chi_\eps) \right \rangle \, ,
\end{align*}
where we also use that $\UU(t)^* : L^2_+(\R; \VV) \to L^2_+(\R; \VV)$ is a unitary map for any $t \in [0,T]$, which is given by the solution of the initial-value problem
$$
\frac{d}{dt} \UU(t) = B_{\Ub(t)}^+ \UU(t) \quad \mbox{for $t \in [0,T]$}, \quad \UU(0) = \id \, .
$$
See Appendix \ref{app:LWP} for details. Using the identity
$$
\UU(t)^* (X^*-z \id)^{-1} = (\UU(t)^* X^* \UU(t) - z \id)^{-1} \UU(t)^* \, ,
$$
we conclude 
$$
\langle \Pi_+ \Vb(t,z), F \rangle_\VV = \frac{1}{2 \pi \ii} \lim_{\eps \to 0} \left \langle (\UU(t)^* X^* \UU(t) - z \id)^{-1} \UU(t)^* (\Pi_+ \Vb(t)), \UU(t)^*(F \chi_\eps) \right  \rangle 
$$
for any $z \in \C_+$, $t \in [0,T]$ and $F \in M_d(\C)$.

\medskip
\textbf{Step 2.} We will now discuss the individual terms which appear in the expression derived in \textbf{Step 1} above. First, we notice that
$$
\frac{d}{dt} \UU(t)^* X^* \UU(t) = \UU(t)^* [X^*, B_{\Ub(t)}^+] \UU(t) = \UU(t)^* T_{\Ub(t)} \UU(t) = T_{\Ub_0} \, ,
$$ 
where we used Proposition \ref{prop:G_B} together with the fact that $T_{\Ub(t)} = \UU(t) T_{\Ub_0} \UU(t)^*$ holds thanks to the Lax evolution. By integration on the interval $[0,t]$, we get
\be \label{eq:explicit1}
\UU(t)^* X^* \UU(t) = X^* + t T_{\Ub_0} \, .
\ee

Next, we observe 
$$
\frac{d}{dt} ( \UU(t)^* (F \chi_\eps)) = -\UU(t)^*(B_{\Ub(t)}^+ (F \chi_\eps)) = o(1) \, ,
$$
where $o(1) \to 0$ in $L^2$ as $\eps \to 0$ uniformly with respect to $t \in [0,T]$. To see this, we remark
\begin{align*}
B_{\Ub}^+ (F \chi_\eps) & = \frac{\ii}{2} T_{\Ub}(F D\chi_\eps) + \frac{\ii}{2} D(T_{\Ub} F \chi_\eps) - \frac{\ii}{2} T_{\Ds \Ub}(F \chi_\eps) \\
& \to \frac{\ii}{2} \Pi_+(D \Ub). F -\frac{\ii}{2} \Pi_+(D\Ub). F = 0 \quad \mbox{as $\eps \to 0$}
\end{align*}
in $L^2(\R; \VV)$ uniformly in $t \in [0,T]$. Therefore, by integrating in $t$, we conclude 
\be \label{eq:explicit2}
\UU(t)^*(F \chi_\eps) = F \chi_\eps + o(1)
\ee
with $o(1) \to 0$ in $L^2_+(\R, \VV)$ as $\eps \to 0$ uniformly in $t \in [0,T]$. 

It remains to discuss the last term from \textbf{Step 1}. Here we claim that
\be \label{eq:explicit3}
\UU(t)^* (\Pi_+ \Vb(t)) = \Pi_+ \Vb_0 \, .
\ee
Since $\UU(0)^* = \id$, we need to show that the time derivative of the left-hand side vanishes. Indeed, we note
$$
\frac{d}{dt} \left ( \UU(t)^*(\Pi_+ \Vb(t)) \right ) = \UU(t)^* \left ( -B_{\Ub(t)}^+ \Pi_+ \Vb(t) +  \pt_t  \Pi_+ \Vb(t) \right ) \, .
$$ 
Now, by the Lax equation $\frac{d}{dt} T_{\Ub(t)} = [B_{\Ub(t)}^+, T_{\Ub(t)}]$ and if we let $E=\mathds{1}_d$ denote the identity matrix in $M_d(\C)$, we find
$$
\frac{d}{dt} T_{\Ub(t)} (E \chi_\eps) = B_{\Ub(t)}^+ T_{\Ub(t)} (E \chi_\eps) - T_{\Ub(t)} B_{\Ub(t)}^+ (E \chi_\eps)
$$
For the first term on right-hand side, we observe
$$
B_{\Ub(t)}^+ T_{\Ub(t)} (E \chi_\eps) \to B_{\Ub(t)}^+ (\Pi_+ \Vb(t)) \quad \mbox{in $L^2_+$ as $\eps \to 0$}
$$
uniformly in $t \in [0,T]$. Furthermore, in the same way as in the discussion showing that $B_{\Ub(t)}^+ (F \chi_\eps) \to 0$ as $\eps \to 0$ for any constant matrix $F \in M_d(\C)$, we conclude
$$
T_{\Ub(t)} B_{\Ub(t)}^+(E \chi_\eps) \to 0 \quad \mbox{in $L^2_+$ as $\eps \to 0$}
$$
uniformly in $t$. On the other hand, we have 
$$
\frac{d}{dt} T_{\Ub}(t) E \chi_\eps \to \pt_t \Pi_+ \Vb  \quad \mbox{in $L^2_+$ as $\eps \to 0$}
$$
 uniformly in $t \in [0,T]$. In summary, we infer that $\pt_t \Pi_+ \Vb(t) = B_{\Ub(t)} \Pi_+ \Vb(t)$ holds, whence it follows 
$$
\frac{d}{dt}  \left ( \UU(t)^* (\Pi_+ \Vb(t)) \right ) = \UU(t)^* (-B_{\Ub(t)}^+ \Pi_+ \Vb(t) + \pt_t \Pi_+ \Vb(t) ) =  0 \, .
$$
This completes the proof of \eqref{eq:explicit3}.

\medskip
\textbf{Step 3.} Combining the results from \textbf{Step 1} and \textbf{Step 2} above, we conclude, for any $F \in M_d(\C)$ and $z \in \C_+$, that
\begin{align*}
\langle \Pi_+ \Vb(t,z), F \rangle_\VV & = \frac{1}{2 \pi \ii} \lim_{\eps \to 0} \left \langle (\UU(t)^* X^* \UU(t) - z \id)^{-1} \UU(t)^* (\Pi_+ \Vb(t)), \UU(t)^*(F \chi_\eps) \right  \rangle \, \\
& = \frac{1}{2 \pi \ii} \lim_{\eps \to 0} \left \langle (X^*+ t T_{\Ub_0} - z \id)^{-1} \Pi_+ \Vb_0, F \chi_\eps \right \rangle \\
& = \frac{1}{2 \pi \ii}\left \langle  I_+[(X^*+tT_{\Ub_0} - z \id)^{-1} \Pi_+ \Vb_0)], F \right \rangle_\VV \, .
\end{align*}
Since $F \in M_d(\C)$ is arbitrary, we deduce the claimed formula for $\Pi_+ \Vb(t,z) \in \VV$. 

The proof of Lemma \ref{lem:explicit} is now complete.
\end{proof}

%%%%%%%%%%%%% 

\section{Global Well-Posedness for Rational Data}

We are now ready to prove global well-posedness for (HWM$_d$) with rational initial data
$$
\Ub_0 \in \mathcal{R}at(\R; \Gr_k(\C^d))
$$
 for any $d \geq 2$ and $0 \leq k \leq d$. The main argument rests on exploiting the explicit flow formula derived above. First, we start with the following general result, which in fact does not require rational initial data.

\begin{lem} \label{lem:injective}
Let $d \geq 2$ be an integer. Suppose $\Wb \in L^\infty(\R; M_d(\C))$ has the following properties
$$
\Wb(x) = \Wb(x)^* \ \ \mbox{a.\,e.}, \quad \Wb(x) = \Wb_\infty + \Vb_0(x) \in M_d(\C) \oplus L^2(\R; M_d(\C)) \, .
$$
Then $X^*+ T_{\Wb}$ acting on $L^2_+(\R; M_d(\C))$ has no real eigenvalues, i.\,e., its point spectrum satisfies $\sigma_{\mathrm{p}}(X^*+ T_{\Wb}) \cap \R = \emptyset$.
\end{lem} 

\begin{proof}
Let $x \in \R$. Since $X^*$ is closed, we find that $\EE := \mathrm{ker}(X^*+T_{\Wb} - x\id)$ is a closed subspace in $L^2_+(\R; M_d(\C))$; see also Section \ref{sec:prelim} for general properties of $X^*$ as well as \ref{GePu-24}. Moreover, from the eigenvalue equation
$$
(X^*+ T_{\Wb} - x) f = 0 
$$
we see that $\EE \subset \mathrm{dom}(X^*)$. By taking the imaginary part of the inner product with $f$ and using that $T_{\Wb}^* = T_{\Wb}$ is self-adjoint and that $x$ is a real number, we conclude that $\mathrm{Im} \langle X^*f, f \rangle = 0$. Recalling the identity \eqref{eq:G_I}, we deduce
$$
I_+(f) = 0 \quad \mbox{for $f \in \EE$} \, .
$$
In view of Lemma \ref{lem:commutator}, we also notice
$$
[X^*, T_{\Wb}] f = \frac{\ii}{2 \pi} \Pi_+ \Vb_0.I_+(f) = 0 \quad \mbox{for $f \in \EE$} \, ,
$$
which shows that $X^* f \in \EE$ for all $f \in \EE$. Thus $\EE$ is an invariant subspace for $X^*$. For the semigroup $\{ S(\eta)^* \}_{\eta \geq 0}$ generated by $X^*$, we thus deduce
$$
S(\eta)^* f = e^{-\ii \eta X^*} f \in \EE \quad \mbox{for all $f \in \EE$ and all $\eta \geq 0$} \, .
$$
But this implies that, for every $f \in \EE$,
$$
0 = I_+(S(\eta)^* f) = \widehat{f}(\eta) \quad \mbox{for all $\eta \geq 0$}.
$$
Hence we see that $f = 0$ for all $f \in \EE$. Therefore, the subspace 
$$
\EE = \mathrm{ker}(X^*+T_{\Wb} - x\id ) = \{ 0 \}
$$
is trivial for any $x \in \R$.
\end{proof}

\subsection*{Proof of Theorem \ref{thm:gwp}}

We are now ready to prove global well-posedness for (HWM$_d$) for rational initial data
$$
\Ub_0 \in \mathcal{R}at(\R; \Gr_k(\C^d))
$$
where $d \geq 2$ and $0 \leq k \leq d$ are given integers. We note that
$$
\Ub_0 = \Ub_\infty + \Vb_0 \in \Gr_k(\C^d) \oplus H^\infty(\R; M_d(\C))
$$
holds. Hence, by the local well-posedness result from Lemma \ref{lem:lwp}, there exists a unique maximal solution 
$$
\Ub(t) = \Ub_\infty + \Vb(t) \in C([0,T_{\max}); \Gr_k(\C^d) \oplus H^\infty(\R; M_d(\C))
$$
of (HWM$_d$) with initial datum $\Ub(0) = \Ub_0$ and maximal (forward) time of existence $T_{\max} \in (0, +\infty]$ such that the following implication holds:
\be \label{eq:blowup_alt}
T_{\max} < +\infty \quad \Rightarrow \quad \lim_{t \nearrow T_{\max}} \| \Vb(t) \|_{H^2} = +\infty \, .
\ee

Thus to show that $T_{\max} = +\infty$ holds true we argue by contradiction and we suppose that $T_{\max} < +\infty$. We now claim that
\be \label{ineq:apriori_first}
\sup_{t \in [0,T_{\max})} \| \Vb(t) \|_{H^2} < +\infty \, ,
\ee
which  implies that $T_{\max} = +\infty$ must hold by \eqref{eq:blowup_alt}. To prove \eqref{ineq:apriori_first},  we first note that $\Vb(t,x)=\Vb(t,x)^*$ for $t \in [0,T_{\max})$ and $x \in \R$. Therefore $\Vb(t) = \Pi_+ \Vb(t) + (\Pi_+ \Vb(t))^*$ and hence it suffices to show that
\be \label{ineq:apriori}
\sup_{t \in [0,T_{\max})} \| \Pi_+ \Vb(t) \|_{H^2} < +\infty \, .
\ee

In view of the explicit flow formula in Lemma \ref{lem:explicit}, we define
\be \label{eq:explicit4}
\mathrm{EF}[\Ub_0](t,z) := \frac{1}{2 \pi \ii} I_+ \left [( X^*+ t T_{\Ub_0} -z \id)^{-1} \Pi_+ \Vb_0 \right ] \quad \mbox{for $t \geq 0$ and $z \in \ov{\C}_+$} \, ,
\ee
where $T_{\Ub_0} : L^2_+(\R; M_d(\C)) \to L^2_+(\R; M_d(\C))$. Let us check that $\mathrm{EF}[\Ub_0]$ is indeed well-defined for all $t \geq 0$ and $z \in \ov{\C}_+$. By Lemma \ref{lem:kronecker} (Kronecker-type theorem), the subspace $\Hfr_1 = \ov{\ran(K_{\Ub_0})} \subset \dom(X^*)$ is finite-dimensional. By Propositions \ref{prop:king} and \ref{prop:G_invariant}, we deduce $
\Pi_+ \Vb_0 \in \Hfr_1$ and that
$$
M(t) := X^* + t T_{\Ub_0} : \Hfr_1 \to \Hfr_1
$$ 
is an endormophism on the finite-dimensional subspace $\Hfr_1$. Moreover, by Lemma \ref{lem:injective} with $\Wb = t \Ub_0$, we see that the eigenvalues of $M(t)$ cannot be real, i.\,e., 
$\sigma(M(t)) \cap \R = \emptyset$ holds, which implies that 
$$
\sigma(M(t)) \subset \C_- \quad \mbox{for all $t \geq 0$} \, .
$$ 
Hence the resolvent $(X^* + t T_{\Ub_0}- z \id)^{-1} : \Hfr_1 \to \Hfr_1$ exists for all $t \geq 0$ and $z \in \ov{\C}_+$. Moreover, by continuity of eigenvalues of $M(t)$ with respect to $t$, we deduce that, for any compact interval $I \subset [0,\infty)$, it holds that
$$
\| (X^* + t T_{\Ub_0}-z\id)^{-1} \|_{\Hfr_1 \to \Hfr_1} \leq C(I, \Ub_0) \quad \mbox{for all $t \in I$ and $z \in \ov{\C}_+$} \, ,
$$
with some finite constant $C(I, \Ub_0) >0$. Since $I_+ : \Hfr_1 \subset \dom(X^*) \to M_d(\C)$ is bounded (as a linear map on a finite-dimensional Hilbert space), we deduce $\mathrm{EF}[\Ub_0](t,z)$ is a rational function in $z$ for any $t \geq 0$, whose poles belong to a compact subset $K =K(I, \Ub_0) \subset \C_-$ when $t \in I$ for any given compact time interval $I \subset [0,\infty)$. 

To summarize, we have shown that, for any given compact interval $I \subset [0,\infty)$, there exists some constant $C = C(I, \Ub_0)> 0$ such that
$$
|\alpha | + \frac{1}{|\mathrm{Im} \, \alpha|} \leq C
$$ 
whenever $\alpha$ is a pole of the rational map $z \mapsto \mathrm{EF}[\Ub_0](t,z)$ with $t \in  I$. By possibly enlarging the constant $C>0$, we obtain the $L^\infty$-bound with
$$
\sup_{x \in \R} | \mathrm{EF}[\Ub_0](t,x) |_F \leq C \quad \mbox{for $t \in I$} \, .
$$
Since $\Hfr_1$ has finite dimension, we easily deduce that the degree of the denominator of the rational functions $z \mapsto \mathrm{EF}[\Ub_0](t,z)$ can be uniformly bounded for $t \in I$. Hence, by applying Lemma \ref{lem:rational_family} below, we deduce
$$
\sup_{t \in I} \| \mathrm{EF}[\Ub_0](t) \|_{H^2} \leq C(I, \Ub_0)
$$
with some finite constant $C(I, \Ub_0)$. 

Since $\Pi_+ \Vb(t) = \mathrm{EF}[\Ub_0](t)$ for $t \in [0, T_{\max})$ and by taking a compact interval $I \subset [0,\infty)$ with $[0,T_{\max}) \subset I$, we conclude that \eqref{ineq:apriori} holds true. This completes the proof that the maximal (forward) time of existence must be $T_{\max} = +\infty$. 

Finally, by the time reversal symmetry of (HWM$_d$) with
$$
\Ub(t,x) \mapsto -\Ub(-t,-x) \, ,
$$ 
which maps solutions to solutions (and evidently preserves rationality in $x$), we deduce that solutions of (HWM$_d$) with rational initial data also uniquely extend to all negative times $t \in (-\infty, 0]$. 

This completes the proof of Theorem \ref{thm:gwp}. \hfill $\qed$

\subsection*{Proof of Theorem \ref{thm:gwp_S2}}
This a direct consequence of Theorem \ref{thm:gwp}. Indeed, let $\uu_0 \in \mathcal{R}at(\R; \Ss^2)$ be given and define $\Ub_0 = \uu_0 \cdot \bm{\sigma} \in \mathcal{R}at(\R; \Gr_1(\C^2))$. By Theorem \ref{thm:gwp}, there exists a unique global solution $\Ub = \Ub(t,x)$ of (HWM$_2$) with initial datum $\Ub(0) = \uu_0$. Hence 
$$
\uu(t,x) = \frac{1}{2} \Tr( \Ub(t,x) \bm{\sigma} )=\frac{1}{2} ( \Tr(\Ub(t,x) \sigma_1), \Tr(\Ub(t,x) \sigma_2), \Tr(\Ub(t,x) \sigma_3))
$$ 
is the claimed unique global-in-time solution of (HWM) with initial datum $\uu(0)=\uu_0$. \hfill $\qed$

\medskip
We close this section with the following auxiliary result used above.

\begin{lem} \label{lem:rational_family}
Let $\mathcal R \subset \C(X)$ be a subset of rational functions. We assume that there exists $C>0$ such that the following properties hold.
\begin{enumerate}
\item If $\alpha $ is a pole of some $R\in \mathcal R$, then 
$$|\alpha | +\frac{1}{|{\rm Im}(\alpha )|} \leq C\ .$$
\item For every $R\in \mathcal R$, $R(x)\to  0$ as $x\to \infty $ and 
$$\| R \|_{L^\infty (\R )}\leq C\ .$$
\item There exists an integer $N$ such that the degree of the denominator of every $R\in \mathcal R$ is at most $N$.
\end{enumerate}
Then, for every integer $k\ge 0$, it holds that
$$\sup_{R\in \mathcal R}\| R\|_{H^k(\R )}<\infty \ .$$
\end{lem}

\begin{proof}
Given $R\in \mathcal R$, write
$$R(x)=\frac{P(x)}{Q(x)}\ ,\ Q(x)=\prod_{j=1}^D (x-\alpha _j)\ ,\ P\in \C[X]\ ,\ {\rm deg}(P)<D\leq N\ .$$
Because of properties (1) and (2), 
$$\max_{0\leq x\leq 1} |P(x)|\leq C\max_{0\leq x\leq 1}\left |  \prod_{j=1}^D (x-\alpha _j)\ \right |\leq C(1+C)^N\ .$$
Consequently, all the coefficients $a_j$ of $P$ satisfy
$$\sup_{j<D}|a_j|\leq B(N,C)\ ,$$
for some constant $B(N,C)$ depending only on $N$ and $C$. 
Similarly, from property (1), all the coefficients of $Q$ are uniformly bounded by a constant depending only on $C$ and $N$.
Moreover, from property (1), for every $x\in \R$,
 $$|Q(x)|\geq \left ( ||x|-C|^2+C^{-2}  \right )^{D/2}\ .$$
Notice that the $k$--th derivative $R^{(k)}$ is a sum of  a finite number -- depending only on $k$ --- of terms of the form
$$\frac{ P^{(m)}Q^{(m_1)}\dots Q^{(m_r)}  }{Q^{r+1}}$$
where $ 0\leq r\leq k$ and $m+m_1+\dots m_r =k$. Notice that the degree of the numerator is at most $(r+1)D-k-1$, and that its coefficients are all bounded by a constant depending only on $k$, $N$ and $C$. Consequently,
$$\| R^{(k)}\|_{L^2}^2\leq A(k,N,C) \max _{r\leq k}\max _{\ell \leq (r+1)D-k-1}\int_\R \frac{x^{2\ell}}{( ||x|-C|^2+C^{-2} )^{D(r+1)}}\, dx $$
with some constant $A(k,N,C)>0$. This completes the proof.
\end{proof}

%%%%%%%%%%%%%%%%%%%%%% Soliton Resolution and Non-Turbulence

\section{Soliton Resolution and Non-Turbulence}

In this section we prove our next main result Theorem \ref{thm:soliton}, which shows soliton resolution and non-turbulence for rational solutions of (HWM$_d$) under the spectral assumption that the Toeplitz operator $T_{\Ub_0} : L^2_+(\R; \C^d) \to L^2_+(\R;\C^d)$ has simple discrete spectrum.

\subsection*{Preliminaries}
Let $d \geq 2$ and $0 \leq k \leq d$ be given integers. In what follows, we suppose that $\Ub_0 \in \mathcal{R}at(\R; \Gr_k(\C^d))$ holds, i.\,e., the map $\Ub_0 : \R \to M_d(\C)$ is a rational matrix-valued function satisfying the pointwise constraints
$$
\Ub_0(x)^* = \Ub_0(x), \quad \Ub_0(x)^2 = \mathds{1}_d, \quad \Tr(\Ub_0(x)) = d-2k \quad \mbox{for $x \in \R$} \, .
$$
In the trivial case of constant initial data $\Ub_0(x) \equiv \Ub_\infty$, we directly obtain Theorem \ref{thm:soliton} with $N=0$. Hence for the rest of the proof, we will assume that $\Ub_0$ is non-constant. 

For the following discussion, we need to clearly distinguish between the Toeplitz operator $T_{\Ub_0}$ acting on the Hardy space $L^2_+(\R; \VV)$ with $\VV=\C^d$ or $\VV=M_d(\C)$, respectively.

From Lemma \ref{lem:key}, we recall the general formula
\be \label{eq:T_key}
T_{\Ub_0}^2 = \id - K_{\Ub_0} \quad \mbox{on $L^2_+(\R; \VV)$} \, ,
\ee
with the trace-class operator $K_{\Ub_0} = H_{\Ub_0}^* H_{\Ub_0} : L^2_+(\R; \VV) \to L^2_+(\R; \VV)$. Since $\Ub_0$ is rational, the operator $K_{\Ub_0}$ is finite-rank by Lemma \ref{lem:kronecker} and we have the finite-dimensional invariant subspace for $T_{\Ub_0}$ given by
\be
\Hfr_1(\VV) := \ran ( K_{\Ub_0} : L^2_+(\R; \VV) \to L^2_+(\R; \VV) ) \, ,
\ee
where we use the notation $\Hfr_1(\VV)$ instead of $\Hfr_1$ to keep track whether we choose $\VV = \C^d$ or $\VV = M_d(\C)$. We introduce the following short-hand notations
\be
\Ts:= T_{\Ub_0} |_{\Hfr_1(M_d(\C))} \quad \mbox{and} \quad \Tst := T_{\Ub_0} |_{\Hfr_1(\C^d)} \, .
\ee
Note that $\Ts = \Ts^*$ and $\Tst = \Tst^*$ are self-adjoint endomorphisms on the finite-dimensional spaces $\Hfr_1(M_d(\C))$ and $\Hfr_1(\C^d)$, respectively. From Proposition \ref{prop:G_invariant} we recall that the generator $X^*$ of the adjoint Lax--Beurling semigroup also acts invariantly on the finite-dimensional subspace $\Hfr_1(\VV)$. Likewise, we use the following notation
\be
\Gs := X^* |_{\Hfr_1(M_d(\C))} \quad \mbox{and} \quad \Gst := X^* |_{\Hfr_1(\C^d)} 
\ee
for the generator $X^*$ of adjoint Lax--Beurling semigroup restricted to the invariant subspaces $\Hfr_1(M_d(\C))$ and $\Hfr_1(\C^d)$, respectively. 

Let us now assume $\Tst$ has \textbf{simple} spectrum, i.\,e., we have
\be \label{ass:Tst}
\sigma(\Tst) = \{ v_1, \ldots, v_N \} \quad \mbox{with} \quad N = \dim \Hfr_1(\C^d) \, .
\ee
Note that $v_n \in (-1,1)$ for $n = 1, \ldots, N$. Let $\phi_n \in \Hfr_1(\C^d) \subset L^2_+(\R; \C^d)$ be a choice of the corresponding normalized eigenfunctions of $\Tst$ such that
$$
\Tst \phi_n = v_n \phi_n \quad \mbox{with} \quad \| \phi_n \|_{L^2} = 1
$$
for $n=1, \ldots, N$. Clearly, the family $( \phi_n )_{1 \leq n \leq N}$ forms an orthonormal basis for $\Hfr_1(\C^d)$. 

We can easily construct an orthonormal basis of eigenfunction for $\Ts$ acting on the matrix-valued finite-dimensional Hilbert space $\Hfr_1(M_d(\C))$ as follows. For $1 \leq n \leq N$ and $1 \leq j \leq d$, we define the matrix-valued functions $\Phi_{n,j} \in L^2_+(\R; M_d(\C))$ by setting
\be
\Phi_{n,j}:= \left (0, \ldots, \underbrace{\phi_n}_{\mbox{$j$-th column}}, \ldots, 0 \right )  \, .
\ee
We readily check that
\be
\Ts \Phi_{n,j} = v_n \Phi_{n,j} \quad \mbox{for $n=1, \ldots, N$ and $j=1, \ldots, d$} \, . 
\ee
Thus the eigenvalues $v_n$ for $\Ts$ are $d$-fold degenerate in a trivial manner by changing the columns in the matrix-valued functions $\Phi_{n,j}$. 

We have the following fact, whose elementary proof we omit. 
\begin{prop}
The functions $\{ \Phi_{n,j} \}_{1 \leq n \leq N, 1 \leq j \leq d}$ form an orthonormal basis of eigenfunctions for $\Ts : \Hfr_1(M_d(\C)) \to \Hfr_1(M_d(\C))$. 
\end{prop}

\subsection*{Perturbation Analysis as $|t| \to \infty$}

From Theorem \ref{thm:gwp} we know that the corresponding solution of (HWM$_d$) with rational initial datum $\Ub_0$ is global in time and satisfies
\be
\Ub(t,x) = \Ub_\infty + \Pi \Vb(t,x) + (\Pi \Vb(t,x))^* \, ,
\ee
where here and in the following we write $\Pi \equiv \Pi_+$ for the Cauchy--Szeg\H{o} projection for notational simplicity. By the following explicit flow formula from Lemma \ref{lem:explicit}, we have
\be
\Pi \Vb(t,x) = \frac{1}{2 \pi \ii} I_+ \left [ (\Gs+ t \Ts - x \id)^{-1} \Pi \Vb_0 \right ]  \quad \mbox{for $(t,x) \in \R \times \R$} \, ,
\ee
using our definitions of $\Gs$ and $\Ts$ acting on the finite-dimensional subspace $\Hfr_1(M_d(\C))$. Note that we can take $x \in \R$ here, since we have already shown that the rational function $\Pi \Vb(t,z)$ for $z \in \C_+$ has no poles on the real axis for all $t \in \R$. Recall also that $\Pi \Vb_0 \in \Hfr_1(M_d(\C))$ holds thanks to Proposition \ref{prop:king} above.

In order to study the large time limit $t \to \pm \infty$, it will be convenient to define 
\be
\Ms(\eps) := \eps \Gs + \Ts \quad \mbox{with} \quad \eps := \frac{1}{t} 
\ee
for $t \neq 0$. In terms these definition, we can write the explicit flow formula as  
\be \label{eq:explicit_scat}
\Pi V(\eps^{-1}, x) = \frac{\eps}{2 \pi \ii} I_+ [ (\Ms(\eps)- \eps x \id)^{-1} \Pi V_0 ] \, .
\ee

Inspired by the analysis in \cite{GeLe-24} for the study of $N$-solitons for the Calogero--Moser derivative NLS, we carry out a perturbation analysis of the non-self-adjoint endomorphisms $\Ms(\eps) : \Hfr_1(M_d(\C)) \to \Hfr_1(M_d(\C))$ in the limit $\eps \to 0$. We have the following facts, where we recall that we always suppose that the nondegeneracy assumption \eqref{ass:Tst} for $\Tst : \Hfr_1(\C^d) \to \Hfr_1(\C^d)$ holds true.

\begin{lem} \label{lem:M_pert}
There exists some $\eps_0 > 0$ sufficiently small such that the following  holds.
\begin{itemize}
\item[(i)] For $1 \leq n \leq N$ and $1 \leq j \leq d$, there exist analytic functions $\eps \mapsto v_n(\eps) \in \C$ and $\eps \mapsto \Phi_{n,j}(\eps) \in \Hfr_1(M_d(\C))$ for $|\eps| \leq \eps_0$ with
$$
v_n(\eps) = v_n + \eps w_n + O(\eps^2), \quad \Phi_{n,j}(\eps) = \Phi_{n,j} + O(\eps) \, ,
$$
$$
\Ms(\eps) \Psi_{n,j}(\eps)= v_n(\eps) \Psi_{n,j}(\eps) \, .
$$
The functions $\{ \Psi_{n,j}(\eps) \}_{1 \leq n \leq N, 1 \leq j \leq d}$ form a basis for $\Hfr_1(M_d(\C))$ for $|\eps| \leq \eps_0$.
\item[(ii)] For $1 \leq n \leq N$, we have 
$$
w_n = \langle \Gst \phi_n, \phi_n \rangle \quad \mbox{and} \quad \mathrm{Im} \, w_n <  0 \, .
$$
\end{itemize}
\end{lem} 

\begin{remark*}
The fact that all complex numbers $w_n$ have non-vanishing imaginary part will play a fundamental role to obtain a-priori bound on all higher Sobolev norms for $\Ub(t,x)$, i.\,e., it rules out the phenomenon of turbulence in the limit $t \to \pm \infty$. This is in striking contrast to the analysis of $N$-soliton solutions for the Calogero--Moser derivative NLS studies in \cite{GeLe-24}, where the corresponding perturbative analysis yields the vanishing of the imaginary parts in the limit $t \to \pm \infty$ (which corresponds to the limit $\eps \to 0$).  
\end{remark*}

\begin{proof}
We divide the proof of Lemma \ref{lem:M_pert} into the following steps.

\medskip
{\bf Step 1.} Let $\eps_0> 0$ be a constant chosen later. For $|\eps| \leq \eps_0$, we define the endomorphisms 
\be
\Mst(\eps) := \Tst + \eps \Gst : \Hfr_1(\C^d) \to \Hfr_1(\C^d) \, .
\ee
Note that $\Mst(0) = \Tst = \Tst^*$ is self-adjoint with simple spectrum $\sigma(\Tst) = \{ v_1, \ldots, v_N \}$ with a corresponding orthonormal basis of eigenfunctions $(\phi_n)_{1 \leq n \leq N}$. By standard analytic perturbation theory, there exist analytic functions $\eps \mapsto v_n(\eps) \in \C$ and $\eps \mapsto \phi_n(\eps) \in \Hfr_1(\C^d)$ for $1 \leq n \leq N$ such that
\be
\Tst(\eps) \phi_n(\eps) = v_n(\eps) \phi_n(\eps) 
\ee
for  $|\eps| \leq \eps_0$, where $\eps_0 > 0$ is some sufficiently small constant.  We have
\be
v_n(\eps) = v_n + \eps w_n + O(\eps^2), \quad \phi_n(\eps) = \phi_n + O(\eps) \, ,
\ee
\be
w_n = \langle \Gst \phi_n, \phi_n \rangle \, .
\ee
Since $(\phi_n)_{1 \leq n \leq N}$ forms an orthonormal basis for $\Hfr_1(\C^d)$ and by continuity with respect to $\eps$, we readily see that the perturbed eigenvectors $(\phi_n(\eps))_{1 \leq n \leq N}$ also form a (not necessarily orthonormal) basis of $\Hfr_1(\C^d)$, provided that $\eps_0 > 0$ is sufficiently small. By defining 
$$
\Phi_{n,j}(\eps) := \left (0, \ldots, \underbrace{\phi_n(\eps)}_{\mbox{$j$-th column}}, \ldots, 0 \right )  \, ,
$$
we easily verify that (i) holds true.

\medskip
{\bf Step 2.} It remains to prove item (ii). Thus we claim, for any $1 \leq n \leq N$, 
\be \label{eq:w_n}
\mathrm{Im} \, w_n = \mathrm{Im} \, \langle \Gst \phi_n, \phi_n \rangle < 0 \, .
\ee
Indeed, let $1 \leq n \leq N$ be given. From the general idenitity \eqref{eq:G_I}, we recall that
\be
\mathrm{Im} \, \langle \Gst \phi_n, \phi_n \rangle = - \frac{1}{4 \pi} |I_+(\phi_n)|^2 \leq 0 \, ,
\ee
with $I_+(f) = \lim_{\xi \to 0^+} \widehat{f}(\xi)$ and $f \in \dom(X^*)$. [Note that $\phi_n \in \dom(X^*)$, since $\phi_n$ is a rational function.] To prove \eqref{eq:w_n}, we argue by contradiction as follows. Let us assume that
\be
I_+(\phi_n) = 0 \, .
\ee
By the commutator formula in Lemma \ref{lem:commutator}, we deduce 
\be
[\Gst, \Tst] \phi_n = \frac{\ii}{2 \pi} \Pi \Vb_0.I_+(\phi_n) = 0 \, .
\ee
Thus from $\Tst \phi_n = v_n \phi_n$ we see that $\Tst \Gst \phi_n = v_n \Gst \phi_n$.  But since $\Tst$ has simple spectrum by assumption, we conclude $\Gst \phi_n = \alpha \phi_n$ for some constant $\alpha \in \C$. By taking the Fourier transform, this yields
\be
\ii \frac{d}{d \xi} \widehat{\phi}_n(\xi) = \alpha \widehat{\phi}_n(\xi) \quad \mbox{for $\xi \geq 0$} \, .
\ee  
Thus we find $\widehat{\phi}(\xi) = A \eu^{-\ii \alpha \xi}$ for $\xi \geq 0$ with some constant $A \neq 0$ (since $\phi_n \not \equiv 0$). Moreover, we infer that $\mathrm{Im} \, \alpha < 0$ since $\widehat{\phi}_n \in L^2(\R_+)$. But this implies that
$$
I_+(\phi_n) = \lim_{\xi \to 0^+} \widehat{\phi}_n(\xi) = A \neq 0 \, ,
$$ 
contradicting our assumption that $I_+(\phi_n)=0$ holds. This shows \eqref{eq:w_n} and completes the proof of Lemma \ref{lem:M_pert}. 
\end{proof}

\subsection*{Proof of Theorem \ref{thm:soliton}}

We are now ready to give the proof of Theorem \ref{thm:soliton}. Adapting the notation from above, we proceed as follows.

\subsubsection*{Asymptotic Behavior as $t \to \pm \infty$}

Recall that  $\eps = t^{-1}$ for $t \neq 0$. In what follows, we shall always assume that $|\eps| \leq \eps_0$ with the constant $\eps_0> 0$ from Lemma \ref{lem:M_pert} above, which amounts to considering times $t$ with $|t| \geq T_0$ where $T_0 = \eps^{-1}$. 

We expand $\Pi \Vb_0 \in \Hfr_1(M_d(\C))$ in terms of the basis $(\Phi_{n,j}(\eps))_{1 \leq n \leq N, 1 \leq j \leq d}$ from Lemma \ref{lem:M_pert}, i.\,e., we write
\be \label{eq:Vb0_expansion}
\Pi \Vb_0 = \sum_{n=1}^N \sum_{j=1}^d \alpha_{n,j}(\eps) \Phi_{n,j}(\eps) 
\ee 
with some coefficients $\alpha_{n,j}(\eps) \in \C$. From Lemma \ref{lem:M_pert} and the fact that $\eps x \in \R$ does not belong to the spectrum $\sigma(\Ms(\eps)) = \{ v_1(\eps), \ldots, v_N(\eps) \} \subset \C_-$, we conclude that
\be
(\Ms(\eps) - \eps x \id)^{-1} \Pi \Vb_0 = \sum_{n=1}^N \sum_{j=1}^d \frac{\alpha_{n,j}(\eps)}{v_n(\eps) - \eps x} \Phi_{n,j}(\eps) \quad \mbox{for $|\eps| \leq \eps_0$ and $x \in \R$} \, .
\ee
In view of the explicit formula \eqref{eq:explicit_scat} together with $\eps = t^{-1}$ and the properties stated in Lemma \ref{lem:M_pert}, we obtain that 
\be
\Pi \Vb(t,x) = \frac{\eps}{2 \pi \ii} I_+ \left [ \sum_{n=1}^N \sum_{j=1}^d \frac{\alpha_{n,j}(\eps)}{v_n(\eps) - \eps x} \Phi_{n,j}(\eps) \right ] = \sum_{n=1}^N \frac{A_n(t)}{x-z_n(t)} \, .
\ee
Here we set
\be
z_n(t) := t v_n(t^{-1}) = t v_n + w_n + O(t^{-1}) \in \C_- \quad \mbox{for $|t| \geq T_0$} \, ,
\ee
and $A_n(t) \in M_d(\C)$  are the matrix-valued functions defined as
\be \label{eq:An_expansion}
A_n(t) := - \frac{1}{2 \pi \ii} \sum_{j=1}^d \alpha_{n,j}(t^{-1}) I_+[\Phi_{n,j}(t^{-1})] \quad \mbox{for $|t| \geq T_0$} \, .
\ee 
Note that, by choosing $T_0 > 0$ possibly larger, we can henceforth ensure that 
\be \label{ineq:zn_zm} 
|z_n(t) - z_m(t)| \geq \frac{|t|}{2} \cdot \min_{n \neq m} |v_n - v_m|  > 0 \quad \mbox{for $|t| \geq T_0$ and $n \neq m$},
\ee
implying that the poles $z_1(t), \ldots, z_N(t) \in \C_-$ are pairwise distinct whenever $|t| \geq T_0$. 

Next we show, after discarding possibly trivial zero terms, that all the matrices $A_j(t) \in M_d(\C)$ are nonzero and nilpotent of degree 2. Moreover, their limits as $t \to \pm \infty$ both exist, coincide, and are nonzero as well.

\begin{prop} \label{prop:scat}
There exists an integer $1 \leq M \leq N$ such that, after possibly relabelling $ \{ A_n(t), z_n(t) \}_{n=1}^N$, it holds that
$$
\Pi_+ \Vb(t,x) = \sum_{n=1}^M \frac{A_n(t)}{x-z_n(t)} \quad \mbox{for $|t| \geq T_0$ and $x \in \R$} \, .
$$
Here the matrices $A_n(t) \in M_d(\C)$ satisfy $A_n(t) \neq 0$ and $A_n(t)^2 = 0$ for $|t| \geq T_0$. 

In addition, it holds
$$
 A_n(t) = A_n + O(t^{-1})   
$$
with some non-zero limits $A_n \neq 0$ satisfiying $A_n^2 = 0$ for $1 \leq n \leq M$.
\end{prop}

\begin{remark*}
In the special case of (HWM) with target $\Ss^2$, corresponding to the target $\Gr_1(\C^2)$ in the matrix-valued case, we will see below that actually $M=N$ must always hold. This observation is based on the simple algebraic fact that nonzero matrices $A \in M_2(\C)$ with $A^2=0$ must have $\rank(A) = 1$. See below for more details.
\end{remark*}

\begin{proof}
{\bf Step 1.} We first show that $A_n(t)^2 = 0$ holds for $1 \leq n \leq N$ and $|t| \geq T_0$. Indeed, we know that
\be \label{eq:Ub_scat}
\Ub(t,x) = \Ub_\infty + \sum_{n=1}^N \frac{A_n(t)}{x-z_n(t)} + \sum_{n=1}^N \frac{A_n(t)^*}{x-\ov{z}_n(t)} \quad \mbox{for $|t| \geq T_0$}
\ee
with the pairwise distinct poles $z_1(t), \ldots, z_n(t) \in \C_-$ and some constant matrix $\Ub_\infty \in \Gr_k(\C^d)$. From the algebraic constraint $\Ub(t,x)^2 = \mathds{1}_d$ and by equating the terms proportional to $(x-z_n(t))^{-2}$ to zero, we  conclude that
$$
A_n(t)^2 = 0 
$$
for $|t| \geq T_0$ and $1 \leq n \leq N$. 

Furthermore, we readily see that we have existence and equality of the limits
$$
\lim_{t \to -\infty} A_n(t) = \lim_{t \to +\infty} A_n(t) =: A_n \in M_d(\C) \, .
$$
This directly follows from the properties in Lemma \ref{lem:M_pert} which yields that
$$
\lim_{|t| \to \infty} A_n(t) = -\frac{1}{2 \pi \ii} \lim_{\eps \to 0} \sum_{j=1}^d \alpha_{n,j}(t^{-1}) I_+[\Phi_{n,j}(t^{-1})] = -\frac{1}{2 \pi \ii} \sum_{j=1}^d \alpha_{n,j} I_+[\Phi_{n,j}] = A_n
$$
with the coefficients $\alpha_{n,j} = \langle \Pi \Vb_0, \Phi_{n,j} \rangle$. Moreover, since $\Phi_{n,j}(t^{-1}) = \Phi_{n,j} + O(t^{-1})$, we readily deduce that
$$
A_n(t) = A_n + O(t^{-1}) \, .
$$
Moreover, from $A_n(t)^2 = 0$ for $|t| \geq T_0$, we readily deduce that the limits satisfy $A_n^2=0$ as well.

\medskip
{\bf Step 2.} By plugging \eqref{eq:Ub_scat} into (HWM$_d$), we obtain the following differential equations for the matrix-valued functions $A_n(t)$:
\be \label{eq:A_dot}
\dot{A}_n(t) = \ii \sum_{m \neq n}^N \frac{[A_n(t), A_m(t)]}{(z_n(t)-z_m(t))^2} \quad \mbox{for $|t| \geq T_0$ and $1 \leq n \leq N$} \, ,
\ee 
where $[X,Y]$ denotes the commutator of matrices in $M_d(\C)$. For details of the calculation that derives \eqref{eq:A_dot}, we refer to the proof of \cite{BeKlLa-20}[Theorem 2.1]; the generalization to (HWM$_d$) is straightforward. We also note that the expression on the right-hand side in \eqref{eq:A_dot} is non-singular for $|t| \geq T_0$ thanks to \eqref{ineq:zn_zm}.

We now claim that
\be \label{claim:nonzero}
A_n(T_0) \neq 0 \quad \Rightarrow \quad \mbox{$A_n(t) \neq 0$ for $t \geq T_0$ and $\displaystyle \lim_{t \to +\infty} A_n(t) \neq 0$} \,.
\ee
Indeed, let $\| A \| = (\Tr(A A^*))^{1/2}$ denote the Frobenius norm of a matrix $A \in M_d(\C)$. Since $\| A_m(t) \| \leq C$ for $t \geq T_0$ and $1 \leq m \leq N$ with some constant $C>0$ (by existence of limits shown in \textbf{Step 1}) and from \eqref{ineq:zn_zm}, we obtain from \eqref{eq:A_dot} the estimate
\be
\| \frac{d}{dt}  A_n(t) \| \lesssim \frac{1}{t^2} \| A_n(t) \| \quad \mbox{for $t \geq T_0$} \, .
\ee
 Suppose now that $A_n(T_0) \neq 0$ and let $T \in (T_0,+\infty]$. Then by integrating the estimate above, we conclude that
 $$
 \int_{T_0}^T \frac{d}{dt} \log \|A_n(t)\| \, dt = \log(\| A_n(T) \| ) - \log(\| A_n(T_0) \|) \lesssim \int_{T_0}^T \frac{dt}{t^2}  \lesssim \frac{1}{T_0} < +\infty$$
 which rules out $A_n(T) = 0$ for $T \in (T_0, +\infty]$. This proves the implication \eqref{claim:nonzero}. 
 
 \medskip
 {\bf Step 3.} Define the integer $0 \leq K \leq N$ by  setting 
 $$
 K := \# \{ 1 \leq n \leq N : A_n(T_0) = 0\}
 $$ 
 and we let $M := N-K$. Now if $M=0$, then $\Ub(t,x) = \Ub_0 = \Ub_\infty$ is a constant solution to (HWM$_d$). But this implies that $K_{\Ub_0} \equiv 0$ and hence $\Hfr_1(\C^d) = \{ 0 \}$ is trivial, which contradicts our assumption that $N = \dim \Hfr_1(\C^d) \geq 1$. Thus we see that $M \geq 1$ holds. 
 
 Thus, after relabelling $\{ A_n(T_0), z_n(T_0) \}_{n=1}^N$ if necessary, we see that
 $$
 \Pi \Vb(T_0, x) = \sum_{n=1}^M \frac{A_n(T_0)}{x-z_n(T_0)} 
 $$
 with $A_n(T_0) \neq 0$ for $1 \leq n \leq M$. By \eqref{claim:nonzero}, we deduce that $A_n(t) \neq 0$ for all $t \geq T_0$ and $1 \leq n \leq M$ and $\lim_{t \to +\infty} A_n(t) = A_n \neq 0$ for $1 \leq n \leq M$. This proves statement of Proposition \ref{prop:scat} for positive times $t \geq T_0$.
 
 Finally, since $\lim_{t \to -\infty} A_n(t) = \lim_{t \to +\infty} A_n(t)$ for all $1 \leq n \leq N$ by \textbf{Step 1}, we complete the proof of Proposition \ref{prop:scat} for negative times $t \leq -T_0$.
 \end{proof}
 
 \subsubsection*{Completing the Proof of Theorem \ref{thm:soliton}}
 We are now ready to complete the proof of Theorem \ref{thm:soliton}, which we divide into the following steps.
 
 \medskip
 \textbf{Step 1.} In view of Proposition \ref{prop:scat} above, we define 
 \be
\Ub^{\pm}(t,x) := \sum_{n=1}^M \Qv_{v_n} (x-v_n t) - (N-1) \Ub_\infty 
\ee
with the rational functions
\be
\Qv_{v_n}(x) := \Ub_\infty + \frac{A_n}{x-y_n + \ii \delta_n} + \frac{A_n^*}{x-y_n - \ii \delta_n} \, .
\ee
 with the non-zero matrices $A_n = \lim_{|t| \to \infty} A_n(t) \in M_d(\C)$ and where we set
 \be
y_n := \mathrm{Re} \, w_n, \quad \delta_n := -\mathrm{Im} \, w_n > 0 \, \quad \mbox{for} \quad n=1, \ldots, M \, .
 \ee
For the difference
$$
\mathbf{R}(t) := \Ub(t) - \Ub^{\pm}(t) \in H^\infty(\R; M_d(\C))
$$
we claim that
\be \label{eq:R_limit} 
\lim_{t \to \pm\infty} \|\mathbf{R}(t) \|_{H^s} = 0 \quad \mbox{for any $s \geq 0$} \, .
\ee
Indeed, since $\mathbf{R}(t,x)^* = \mathbf{R}(t,x)$ and thus $\mathbf{R} = \Pi \mathbf{R} + (\Pi \mathbf{R})^*$, it suffices to consider $\Pi \mathbf{R}$.  We note that
 \begin{align*}
\Pi \mathbf{R}(t,x) & = \sum_{n=1}^M \left ( \frac{A_n(t)}{x-z_n(t)} - \frac{A_n}{x-y_n - v_n t + \ii \delta_n} \right ) \\
& = \sum_{n=1}^M \frac{A_n(t)-A_n}{x-z_n(t)} + \sum_{n=1}^M \left ( \frac{A_n}{x-z_n(t)} - \frac{A_n}{x-y_n - v_n t + \ii \delta_n} \right ) \\
& =: r_1(t) + r_2(t) \, .
 \end{align*}
By recalling that $A_n(t)-A_n = O(t^{-1})$ and taking the Fourier transform, we see that
$$
\| r_1(t) \|_{H^s}  \leq O(t^{-1}) \sum_{n=1}^M \left ( \int_0^{\infty} \langle \xi \rangle^{2s} \eu^{-2\delta \xi} \, d\xi \right)^{1/2} \to 0 \quad \mbox{as} \quad t \to \pm \infty \, ,
$$
where we also used that $\mathrm{Im} \, z_n(t) \leq -\delta < 0$ for $|t| \geq T_0$ and $1 \leq n \leq M$ with some constant $\delta > 0$. Furthermore, we find
\begin{align*}
\| r_2(t) \|_{H^s} & \leq C \sum_{n=1}^M  \left ( \int_0^{\infty} \langle \xi \rangle^{2s} | \eu^{-\ii z_n(t) \xi} - \eu^{-\ii (y_n+v_n t + \ii \delta_n) \xi}|^2 \, d \xi \right)^{1/2} \\
& \leq C \sum_{n=1}^M \left ( \int_0^{\infty} \langle \xi \rangle^{2s} \eu^{-2 \delta \xi} | \eu^{O(t^{-1}) \xi}-1|^2 \, d \xi \right)^{1/2}  \to 0 \quad \mbox{as} \quad t \to \pm \infty,
\end{align*}
by dominated convergence and by making use of the fact that  $z_n(t) = y_n + v_n t - \ii \delta_n + O(t^{-1})$ and $\delta_n \geq \delta > 0$ for $n=1, \ldots, M$. This completes the proof of  \eqref{eq:R_limit}.

\medskip
\textbf{Step 2.} Next, we show that each rational functions $\Qv_{v_n}$ yields a profile for a traveling solitary wave for \eqref{eq:HWMd} with velocity $v_n$.

First, we verify that $\Qv_{v_n} : \R \to \Gr_k(\C^d)$ holds. Indeed, for any $1 \leq n \leq N$ and $x \in \R$ fixed, we observe that
\begin{align*}
\Ub(t,x + v_n t) & = \Qv_{v_n}(x) + \sum_{j \neq n}^N \left ( \frac{A_j}{x-y_j-(v_j-v_n)t} + \frac{A_j^*}{x-y_j - (v_j -v_n)t} \right ) + \mathbf{R}(t,x) \\
& \to \Qv_{v_n}(x)  \quad \mbox{as} \quad |t| \to +\infty \, ,
\end{align*}
which follows from \eqref{eq:R_limit} and the fact that $v_j \neq v_n$ for $j \neq n$. From this we easily conclude that $\Qv_{v_n}(x) \in \Gr_k(\C^d)$ for all $x \in \R$.

Next, we prove that each $\Qv_{v_n} \in \mathcal{R}at(\R; \Gr_k(\C^d))$ is a traveling solitary wave profiles for the velocity $v_n$. By taking the limit $\eps = t^{-1} \to 0$ in \eqref{eq:Vb0_expansion} and \eqref{eq:An_expansion}, we obtain (using the notation in the proof of Proposition \ref{prop:scat} above) that
$$
\Pi \Vb_0=\sum_{n=1}^N \sum_{j=1}^d \alpha_{n,j}\Phi_{n,j}\, , \quad A_n = -\frac{1}{2\pi \ii}\sum_{j=1}^d \alpha_{n,j}I_+(\Phi_{n,j})\, .
$$
Let $(e_1,\dots,e_d)$ be the canonical basis of $\C^d$. Then $\Phi_{n,j}=e_j^T \phi_n$ and therefore
$$A_n =-\frac{1}{2\pi \ii}\left (\sum_{j=1}^d \alpha_{n,j}e_j\right )^T I_+(\varphi_n),$$
or, equivalently, we can write
$$
A_n=\langle .,\eta_n\rangle_{\C^d} I_+(\varphi_n) \quad \mbox{with $\eta_n:=\displaystyle \frac{1}{2\ii \pi}\sum_{j=1}^d \overline \alpha_{n,j}e_j \in \C^d$ for $n=1, \ldots, M$} \, .
$$
Note that $A_n \neq 0$ with $A_n^2 = 0$. Hence $\eta_n \in \C^d$ and $I_+(\varphi_n) \in \C^d$ are non-zero vectors with $\langle \eta_n, I_+(\phi_n) \rangle_{\C^d} = 0$. In particular, we see that $\rank(A_n) = 1$ for $1 \leq n \leq M$.

Now we reformulate the eigenfunction identity 
$$
T_{\Ub_0} \varphi_n =v_n \varphi_n 
$$ 
for the Toeplitz operator $T_{\Ub_0} : L^2_+(\R; \C^d) \to L^2_+(\R; \C^d)$. Indeed, let us apply $I_+$ to both sides while using the following elementary lemma.

\begin{lem} \label{lem:super_simple}
Let $f,g\in L^2_+$ be rational functions. Then
$$I_+(fg)=0 \quad \mbox{and} \quad I_+(\Pi (f\overline g))= \int_{\R} f \ov{g} \, dx  \, .$$
\end{lem}

\begin{proof} This simply follows from by using the following fact: For all $h \in \dom(X^*)$, we have
$I_+(h)=\lim_{\eps \to 0^+} \langle h, \chi_\eps \rangle _{L^2}$ with $\chi_\eps (x):=\frac{1}{1-\ii \eps x}$.
\end{proof}

From Lemma \ref{lem:super_simple}, we infer
$$I_+(T_{\Ub_0}\varphi_n)=\Ub_\infty I_+(\varphi_n)+\int_\R (\Pi \Vb_0)^* \varphi_n\, dx = \Ub_\infty I_+(\varphi_n)+2\ii \pi \eta_n\ ,$$
because, using that $\varphi_n$ is normalized in $L^2$, 
$$\int_\R \Phi_{p,j}^*\varphi_n \, dx = e_j \delta_{n,p}\ .$$
The eigenfunction identity $T_{\Ub_0} \varphi_n =v_n \varphi_n$ therefore implies
\be \label{U1}
\Ub_\infty I_+(\varphi_n)=v_n I_+(\varphi _n) -2\pi \ii \eta_n\ .
\ee
Applying the matrix $\Ub_\infty$ to both sides of the above identity, we get
\be \label{U2}
\Ub_\infty \eta_n =-v_n\eta _n -\frac{1}{2 \pi \ii}(1-v_n^2)I_+(\varphi_n)\ .
\ee
Recall that $I_+(\varphi_n)$ and $\eta _n$ are non-zero vectors in $\C^d$ with $\langle \eta_n, I_+(\phi_n) \rangle_{\C^d} = 0$. We denote by $P_n = \mathrm{span} \{ \eta_n, I_+(\phi_n) \}$  the two-dimensional plane in $\C^d$ generated by these two vectors. We notice that $\Ub_\infty$ preserves $P_n$ and hence it preserves $P_n^\perp $, since $\Ub_\infty ^*=\Ub_\infty $. It is now easy to check that the kernel of $H_{\Qv_{v_n}}$ is given by
$$\ker( H_{\Qv_{v_n}})= \frac{x-y_n-\ii\delta_n}{x-y_n+\ii\delta _n}L^2_+(\R )I_+(\varphi_n) \oplus L^2_+(\R )\eta_n \oplus (L^2_+(\R )\otimes P_n^\perp )$$
and that its orthogonal subspace  in $L^2_+(\R; \C^d)$ is generated by 
$$\psi _n(x):=\frac{1}{x-y_n+\ii \delta_n}I_+(\varphi_n)\ .$$
Furthermore, from  \eqref{U1}, \eqref{U2} and the identity 
$$\Ub_\infty A_n +A_n \Ub_\infty =\frac{A_nA_n^*+A_n^*A_n}{2\ii \delta _n} \, ,$$
we get $\Vert I_+(\varphi_n)\Vert_{\C^d}^2=4\pi \delta_n$ and 
$$T_{\Qv_{v_n}}\psi_n =v_n \psi_n\ .$$
Finally, a direct calculation using again \eqref{U1} and \eqref{U2} leads to 
$$-2\ii v_n\Qv_{v_n}'(x)=[\Qv_{v_n},\Ds \Qv_{v_n}](x)\ ,$$
which precisely means that $\Qv_{v_n} (x-v_nt)$ is a traveling solitary wave for \eqref{eq:HWMd} with velocity $v_n$.

\medskip
\textbf{Step 3.} We next show that the integer $1 \leq M \leq N$ given in Proposition \ref{prop:scat} must satisfy
$$
M=N 
$$
where  $\sigma_{\mathrm{d}}(T_{\Ub_0}) = \{v_1, \ldots, v_N\}$. To see this, we recall from Proposition \ref{prop:scat} that
$$
\Ub(t,x) = \Ub_\infty + \sum_{n=1}^M \frac{A_n(t)}{x-z_n(t)} + \sum_{n=1}^M \frac{A_n(t)^*}{x-\ov{z}_n(t)} \quad \mbox{for $|t| \geq T_0$} 
$$
with non-zero matrices $A_1(t), \ldots, A_M(t) \in M_d(\C)$ such that $A_n(t)^2 =0$ and pairwise distinct poles $z_1(t), \ldots, z_M(t) \in \C_-$. Furthermore, from \eqref{eq:An_expansion} and the arguments in the beginning of \textbf{Step 2} above, we deduce that
$$
A_n(t) = \langle \cdot, \eta_n(t) \rangle_{\C^d} I_+(\phi_n(t)) \quad \mbox{for $|t| \geq T_0$}
$$ 
with nonzero vectors $\eta_n(t),I_+(\phi_n(t)) \in \C^d$ such that $\langle \eta_n(t), I_+(\phi_n(t)) \rangle_{\C^d} = 0$. In particular, we conclude that $\rank(A_n(t)) = 1$ for $|t| \geq T_0$. Hence we can apply Lemma \ref{lem:simple_poles} (see also the remark there) to deduce that $\rank(K_{\Ub(T_0)}) = M$.

On the other hand, thank to the Lax evolution, we get that $K_{\Ub(T_0)}= \UU(T_0) K_{\Ub_0} \UU(T_0)^*$ with some unitary map $\UU(T_0) : L^2_+(\R; \C^d) \to L^2_+(\R; \C^d)$. This implies  $\rank(K_{\Ub(T_0)}) = \rank(K_{\Ub_0}) = N$, whence it follows that $M=N$.

\medskip
\textbf{Step 4.} Finally, we observe that $\| \Ub^{(\infty)}(t) \|_{\dot{H}^s} \leq C$ for all $t \in \R$ with some constant $C>0$ depending on $s>0$. Furthermore, in view of \eqref{eq:R_limit} and $\Ub \in C(\R; \dot{H}^\infty)$, we readily deduce the a-priori bounds
$$
\sup_{t \in \R} \| \Ub(t) \|_{\dot{H}^s} \leq C(\Ub_0,s) < \infty
$$
for any $s >0$.

The proof of Theorem \ref{thm:soliton} is now complete.  \hfill $\qed$

%%%%%%%%%%%%%%%%%% Section Refined Analysis for S^2

\section{Refined Analysis for Target $\Ss^2$}
\label{sec:target_S2}

We now consider (HWM) with target $\Ss^2$. The goal of this section is to refine the general Theorem \ref{thm:soliton} on soliton resolution for the target $\Ss^2 \cong \Gr_1(\C^2)$, leading to Theorem \ref{thm:soliton_S2}. Moreover, we will establish that the spectral condition of simplicity of the discrete spectrum $\sigma_{\mathrm{d}}(T_{\Ub_0})$ holds for a {\em dense} subset of rational initial data in the case of the target $\Ss^2$, as formulated in Theorem \ref{thm:generic}. The proof of this density result will make essential use of the stereographic projection $\Ss^2 \to \C \cup \{ \infty \}$ to find a suitable parametrization of rational maps from $\uu : \R \to \Ss^2$ and the corresponding Toeplitz operators $T_\Ub$ with rational matrix-valued symbol $\Ub = \uu \cdot \bm{\sigma}$. Our arguments will be based on analyticity properties to finally conclude Theorem \ref{thm:generic}. We expect that the density  result stated in Theorem \ref{thm:generic} can be generalized to (HWM$_d$) with target $\Gr_k(\C^d)$. However, the algebraic and analytic challenges would require a vast extension of the following analysis, which we haven chosen not to pursue here.

For the reader's convenience, we recall (HWM) with target $\Ss^2$ is equivalent to (HWM$_d$) with $d=2$ for matrix-valued maps of the form
$$
\Ub(x) = \uu(x) \cdot \bm{\sigma} =  \left ( \begin{array}{cc} u_3(x) & u_1(x) - \ii u_2(x) \\ u_1(x) + \ii u_2(x) & -u_3(x) \end{array} \right ) \in \Gr_1(\C^2) 
$$
where  $\uu = (u_1,u_2,u_3) : \R \to \Ss^2$. 

\subsection*{Parametrization  by Stereographic Projection}

Let $\uu : \R \to \Ss^2$ be a map and, as usual, we set $\Ub = \uu \cdot \bm{\sigma}$. For the rest of this subsection, we will consider the case $\VV=\C^2$, i.\,e., we consider the Toeplitz operator
$$
T_\Ub : L^2_+(\R; \C^2) \to L^2_+(\R; \C^2)
$$
acting on $\C^2$-valued functions in the Hardy space $L^2_+$. Likewise, the operators $H_{\Ub}$ and $K_{\Ub} = H_{\Ub}^* H_{\Ub}$ act on $L^2_+(\R; \C^2)$ throughout the following. By using the (inverse) stereographic projection 
$$
\hat{\C} = \C \cup \{ \infty \} \to \Ss^2, \quad z \mapsto \left ( \frac{2 \, \mathrm{Re} \, z}{z \ov{z} + 1}, \frac{2 \, \mathrm{Im} \, z}{z \ov{z} + 1}, \frac{z \ov{z}-1}{z \ov{z} + 1} \right ) \, ,
$$
we obtain the following explicit description in the case of rational maps from $\R$ to $\Ss^2$.
 
\begin{thm} \label{thm:toeplitz_stereo}
Let $\uu =(u_1, u_2, u_3) : \R \to \Ss^2$ be a rational map. Given an integer $N \geq 1$, the following statements are equivalent.
\begin{itemize}
\item[(i)] $\dim \Hfr_1 = \rank(K_{\Ub}) = N$.
\item[(ii)] The least common denominator of $u_1, u_2, u_3$ has degree $2N$.
\item[(iii)] There exists a rational function $R \in \C(X)$ of the form
$$
R(x) = \frac{P(x)}{Q(x)} \, ,
$$
where $P \in \C[X]$ is a polynomial of degree $N$ and $Q \in \C[X]$ is a non-zero polynomial of degree at most $N-1$, such that $P$ and $Q$ have no common factor such that, up to rotation on the sphere $\Ss^2$, we have
$$
u_1(x) + \ii u_2(x) = \frac{2 R(x)}{R(x) \ov{R}(x)+1} \, , \quad u_3(x) = \frac{R(x) \ov{R}(x) -1}{R(x) \ov{R}(x) + 1} \, .
$$
\end{itemize}
\end{thm}  

\begin{remarks*}
1) We use $\C(X)$ to denote the field of rational functions with one variable with coefficients in $\C$. Likewise, we use $\C[X]$ to denote the ring of complex polynomials over $\C$. The variable $X$ either represents an element $x \in \R$ or $z \in \C$. 

2) For a polynomial $T \in \C[X]$ with $T(x) = \sum_{j=0}^N t_j x^j$, we denote its complex conjugate by $\ov{T}(x) = \sum_{j=0}^N \ov{t}_j x^j$ obtained by complex conjugation of its coefficients. Likewise, for a rational function $R = P/Q \in \C(X)$, we denote its complex conjugate by $\ov{R}= \ov{P}/\ov{Q}$.
\end{remarks*}

\begin{proof}
The proof of Theorem \ref{thm:toeplitz_stereo} is given in Appendix \ref{app:stereo}.
\end{proof}

In view of Theorem \ref{thm:toeplitz_stereo} we introduce, for an integer $N \geq 1$, the following subsets of rational functions 
$$
\mathcal{R}_N :=  \left \{ \frac{P(x)}{Q(x)} \in \C(X) \mid \deg P = N, \deg Q \leq N-1, Q \not \equiv 0,  \mathrm{gcd}(P, Q) = 1 \right \} \, .
$$
For $\uu \in \mathcal{R}at(\R; \Ss^2)$, we can henceforth assume with $\uu(\infty) = \mathbf{e}_3$ by rotational symmetry on $\Ss^2$. By Theorem \ref{thm:toeplitz_stereo}, we have the canonical equivalence of sets
$$
\mathcal{K}_N := \{ \uu \in \mathcal{R}at(\R; \Ss^2) \mid \uu(\infty) = \mathbf{e}_3, \; \rank(K_{\Ub}) = N \} \cong \mathcal{R}_N
$$
 by means of the (inverse) stereographic projection in Theorem \ref{thm:toeplitz_stereo} (iii) above.

\medskip
Next, we analyze the topological properties of $\mathcal{R}_N$ more closely. For $P/Q \in \mathcal{R}_N$, we can assume without loss of generality that $P$ is a monic polynomial, i.\,e., we denote
$$
P(x) = x^N + p_1 x^{N-1} + \ldots + p_N \quad \mbox{with $p_k \in \C$ for $k=1, \ldots N$} \, .
$$
The polynomials $Q \in \C[X]$ will be written as
$$
Q(x) = q_1 x^{N-1} + \ldots + q_N \quad \mbox{with $q_k \in \C$ for $k=1, \ldots, N$} \, ,
$$
where $(q_1, \ldots, q_N) \neq (0,\ldots, 0)$. Evidently, we can identify the pair of polynomials $(P, Q) \in \C[X] \times \C[X]$ above uniquely by elements in $\C^N \times (\C^N \setminus \{ 0 \})$. In particular, the set $\mathcal{R}_N$ can be naturally regarded as a subset in $\C^{2N}$. We have the following result. 

\begin{lem} \label{lem:R_n}
The set $\mathcal{R}_N \subset \C(X)$ can be canonically identified with a non-empty, open and connected subset $\mathcal{A}_N$ in $\C^{2N}$. 
\end{lem}

\begin{proof}
We divide the proof into the following steps.

\medskip
\textbf{Step 1.}  Elements $P/Q \in \mathcal{R}_N$ can be canonically identified with pairs $(P,Q) \in \C^{2N}$ of the form
$$
P = (p_1, \ldots, p_N) \in \C^N, \quad Q = (q_1, \ldots, q_N) \in \C^N \setminus \{ 0 \} \, ,
$$
such that $P$ and $Q$ have no common factor as polynomials. Let $\mathcal{A}_N \subset \C^{2N}$ denote the set of such pairs $(P,Q)$. By the fundamental theorem of algebra, we can write $P(x) = \prod_{j=1}^N (x-\xi_j)$ where $\xi_1, \ldots, \xi_N \in \C$ denote the roots of $P$ counted with their multiplicity. In order to take into account possible permutations of the roots, we introduce the quotient space 
$$\C^N_{\mathrm{sym}} = \C^N / \sim$$ 
with the equivalence relation $(\xi_1, \ldots, \xi_N) \sim (\xi_{\sigma(1)}, \ldots, \xi_{\sigma(N)})$ for all permutations $\sigma \in S_N$. We use $[\xi_1, \ldots, \xi_N]$ to denote elements in $\C^N_{\mathrm{sym}}$. It is a classical fact that the map which assigns to any polynomial $P$ of degree $N$ its roots modulo permutations,
$$
\tau : \C^N \to \C^N_{\mathrm{sym}}, \quad P \mapsto [\xi_1, \ldots, \xi_N],
$$
is continuous. Let us define the map
$$
F : \C^N \times (\C^{N} \setminus \{ 0 \} ) \to \C, \quad (P,Q) \mapsto \prod_{j=1}^N Q(\xi_j(P)), 
$$
where $[\xi_1(P), \ldots, \xi_N(P)] \in \C^N_{\mathrm{sym}}$ denote the roots (modulo permutations) of the polynomial $P$. Clearly, we have
$$
F(P,Q) \neq 0 \quad \Leftrightarrow \quad \mbox{$P$ and $Q$ have no common factor}.
$$
By continuity of the map $F$, we deduce that the set
$$
\mathcal{A}_N = \{ (P,Q) \in \C^N \times \C^N \setminus \{ 0 \} : F(P,Q) \neq 0 \}
$$
is an open subset in $\C^{2N}$. Moreover, it is evident that $\mathcal{A}_N$ is non-empty.

\medskip
\textbf{Step 2.} Next, we prove that $\mathcal{A}_N \subset \C^{2N}$ is connected. Since $\mathcal{A}_N$ is open, this is equivalent to being pathwise connected. For $(P,Q) \in \mathcal{A}_N$, we define the set
$$
V_P = \{ Q \in \C^N : F(P,Q)=0 \} = \{ Q \in \C^N \mid \prod_{j=1}^N Q(\xi_j(P))=0 \} \, .
$$
As a zero set of a non-trivial polynomial in $Q=(q_1, \ldots, q_N) \in \C^N$, we see that $V_P$ is an algebraic set in $\C^N$ with $0 \in V_P$.\footnote{In fact, we verify that $V_P= \bigcup_{j=1}^N V_j$ with the linear subspaces $V_j = \ker  \, \ell_j$ with the linear forms $\ell_j : \C^N \to \C$ given by $\ell_j(Q) = Q_1 \xi_j(P)^{N-1} + \ldots + Q_{N-1} \xi_j(P) + Q_N$.} Regarding its complement, we claim that 
\be \label{eq:connected}
\mbox{$\C^N \setminus V_P$ is connected} \, .
\ee
Since $\C^N \setminus V_P$ is open, this claim is equivalent to pathwise connectedness of this set. Let $Q, \tilde{Q} \in \C^N \setminus V_P$ with $Q \neq \tilde{Q}$ be given and consider the set
$$
L = \{ Q + \zeta (\tilde{Q}-Q) \mid \zeta \in \C \} \, ,
$$ 
which corresponds to the complex line in $\C^N$ that  connects $Q$ and $\tilde{Q}$. Since we have $L \not \subseteq V_P$ and $V_P$ is the zero set of a polynomial in $Q \in \C^N$, there are only finitely many points of intersections of $L$ with $V_P$, i.\,e., 
$$
L \cap V_P = \{ z_1, \ldots, z_K \}
$$ 
for some $z_1, \ldots, z_K \in \C^N$. However, the set $L \setminus \{ z_1, \ldots, z_K \} \simeq \R^2 \setminus \{ p_1, \ldots, p_K \}$ with finitely many points $p_1, \ldots, p_K \in \R^2$ is pathwise connected. Thus there exists a continuous map $\gamma: [0,1] \to \C^N \setminus V_P$ with $\gamma(0)= Q$ and $\gamma(1) = \tilde{Q}$. This proves \eqref{eq:connected}.

Next, we suppose  $(P, Q) \in \mathcal{A}_N$ and $(\tilde{P}, \tilde{Q}) \in \mathcal{A}_N$  are given. We prove that $(P,Q)$ and $(\tilde{P}, \tilde{Q})$ can be connected by a continuous path in $\mathcal{A}_N$ as follows. We consider the sets 
$$
W = \{ P \} \times (\C^N \setminus V_P) \quad  \mbox{and}  \quad \tilde{W} = \{\tilde{P} \} \times (\C^N \setminus V_{\tilde{P}}) \, . 
$$
Evidently, we have that $(P,Q) \in W$ and $(\tilde{P}, \tilde{Q}) \in \tilde{W}$. Let $Q_*=(0, \ldots, 0, 1) \in \C^N$ corresponding to the constant polynomial $Q_*(x) \equiv 1$. By \eqref{eq:connected} and the evident fact that $Q_* \in (\C^N \setminus V_P) \cap (\C^N \setminus V_{\tilde{P}})$, we can find two continuous paths in $\mathcal{A_N}$ that connect $(P,Q)$ with $(P, Q_*)$ and $(\tilde{P}, \tilde{Q})$ with $(\tilde{P}, Q_*)$, respectively. Furthermore, we easily construct a continuous path in $\mathcal{A}_N$ which connects $(P,Q_*)$ and $(\tilde{P},Q_*)$. This shows that $\mathcal{A}_N \subset \C^{2N}$ is pathwise connected.

This completes the proof of Lemma \ref{lem:R_n}.
\end{proof}

With the results derived above, we are now ready to give the proofs of Theorems \ref{thm:soliton_S2} and \ref{thm:generic} for \eqref{eq:HWM} with target $\Ss^2$.

\subsection*{Proof of Theorem \ref{thm:soliton_S2} (Soliton Resolution for Target $\Ss^2$)}
Suppose $\uu_0 \in \mathcal{R}at(\R;\Ss^2)$ satisfies the assumptions of Theorem \ref{thm:soliton_S2} and let $\Ub_0 = \uu_0 \cdot \bm{\sigma} \in \mathcal{R}at(\R; \Gr_1(\C^2))$ be the corresponding initial datum for \eqref{eq:HWMd} with $d=2$. 

By applying Theorem \ref{thm:soliton} and using the identification $\Gr_1(\C^2) \cong \Ss^2$ via the use of the Pauli matrices $\bm{\sigma}=(\sigma_1, \sigma_2, \sigma_3)$, we obtain that
$$
\lim_{t \to \pm \infty} \| \uu(t) - \uu^{\pm}(t) \|_{\dot{H}^s} = 0 \quad \mbox{for any $s > 0$} \, ,
$$
with
$$
\uu^{\pm}(t,x) = \sum_{j=1}^N \qv_{v_j}(x-v_jt) - (N-1) \uu_\infty \, .
$$
Here each $\qv_{v_j} \in \mathcal{R}at(\R; \Ss^2)$ is a profile of a ground state traveling solitary wave for (HWM) with velocity $v_j$ and it is given by
$$
\qv_{v_j}(x) = \uu_\infty + \frac{A_j}{x-y_j+ \ii \delta_j} + \frac{A_j^*}{x-y_j - \ii \delta_j} \, .
$$
%By a direct calculation, we obtain $E(\qv_{v_j}) = (1-v_j^2) \cdot \pi$. From the classification result \cite{LeSc-18} we see that $\qv_{v_j}$ corresponds to a {\em ground state} traveling solitary wave.

The proof of Theorem \ref{thm:soliton_S2} is now complete. \hfill $\qed$ 

\subsection*{Proof of Theorem \ref{thm:generic} (Density of Rational Data with Simple Discrete Spectrum)}
Let $\uu \in \mathcal{R}at(\R; \Ss^2)$ be given and set $\Ub = \uu \cdot \bm{\sigma} \in \mathcal{R}at(\R; \Gr_1(\C^2))$ as usual. We recall that the discrete spectrum $\sigma_{\mathrm{d}}(T_\Ub)$ of  the Toeplitz operator $T_{\Ub} : L^2_+(\R; \C^2) \to L^2_+(\R; \C^2)$ is found to be
$$
\sigma_{\mathrm{d}}(T_{\Ub}) = \sigma(T_{\Ub} |_{\Hfr_1}) 
$$
with the finite-dimensional subspace $\Hfr_1 = \ran(K_{\Ub})=\ran(\id-T_{\Ub}^2)$. We are interested in the case when $\sigma_{\mathrm{d}}(T_\Ub)$ is simple and therefore we define the set
$$
\mathcal{R}at_{\mathrm{s}}(\R; \Ss^2) := \{ \uu \in \mathcal{R}at(\R; \Ss^2) \mid \mbox{$\sigma_{\mathrm{d}}(T_\Ub)$ is simple} \} \, .
$$
We have the following result (stated as Theorem \ref{thm:generic} in the introduction).

\begin{thm} \label{thm:simple}
The subset $\mathcal{R}at_{\mathrm{s}}(\R; \Ss^2)$ is dense in $\dot{H}^{\frac 1 2}(\R, \Ss^2)$.
\end{thm}

\begin{proof}
 We divide the proof into the following steps.

\medskip
\textbf{Step 1.} For a given integer $N \geq 1$, we define the set
$$
\mathcal{K}_N := \left \{ K_{\Ub}  : \mbox{$\mathrm{Rank}(K_\Ub) = N$ with $\uu \in \mathcal{R}at(\R; \Ss^2)$ and $\uu(\pm \infty) = \mathbf{e}_3$} \right \} \, .
$$
From Theorem \ref{thm:toeplitz_stereo} part (iii), we recall that $\mathcal{K}_N$ is canonically identified with set of rational functions $\mathcal{R}_N \subset \C(X)$ via the (inverse) stereographic projection. By Lemma \ref{lem:R_n}, we can canonically identify $\mathcal{R}_N$  with a non-empty, open and connected subset $\mathcal{A}_N \subset \C^{2N}$. Let us write $R=P/Q \equiv (P,Q) \in \mathcal{A}_N$ in what follows. 

Next, we define the map $\mathsf{u} : \mathcal{A}_N \to L^\infty(\R; \R^3)$ with
\begin{align*}
(\mathsf{u}(P,Q))(x) & :=\left ( \frac{ 2 \mathrm{Re} (P(x) \ov{Q}(x))}{P(x) \ov{P}(x) + Q(x) \ov{Q}(x)}, \frac{2 \mathrm{Im}( P(x) \ov{Q}(x))}{  \ov{P}(x) P(x) + \ov{Q}(x) Q(x)}, \right . \\
& \qquad \left . \frac{P(x) \ov{P}(x) - Q(x) \ov{Q}(x)}{P(x) \ov{P}(x) + Q(x) \ov{Q}(x)} \right ) \quad \mbox{with $x \in \R$} \, .
\end{align*}
Note that, for any $(P,Q) \in \mathcal{A}_N$, the map $x \mapsto (\mathsf{u}(P,Q))(x)$  belongs to $\mathcal{R}at(\R; \Ss^2)$ and it evidently satisfies $(\mathsf{u}(P,Q))(\pm \infty) = \mathbf{e}_3$. Correspondingly, we obtain a map $\mathsf{U} : \mathcal{A}_N \to L^\infty(\R; M_2(\C))$ by setting
\be \label{eq:analytic1}
\mathsf{U}(P,Q) := \mathsf{u}(P,Q) \cdot \bm{\sigma} \, .
\ee
By Theorem \ref{thm:toeplitz_stereo}, the map
\be \label{eq:analytic2}
\mathsf{K} : \mathcal{A}_N \to \mathcal{B}(L^2_+(\R; \C^2)), \quad (P,Q) \mapsto \mathsf{K}(P,Q) := H_{\mathsf{U}(P,Q)}^* H_{\mathsf{U}(P,Q)} 
\ee
is injective and and its image satisifes $\mathsf{K}(\mathcal{A}_N) = \mathcal{K}_N$. 

\medskip
\textbf{Step 2.} We claim that 
$$
\mbox{$\mathsf{K} : \mathcal{A}_N \to \mathcal{B}(L^2_+(\R; \C^2))$ is real analytic} 
$$
with the usual identification that $\mathcal{A}_N \subset \C^{2N} \cong \R^{4N}$. Indeed, since the expressions in \eqref{eq:analytic1} and \eqref{eq:analytic2} are linear and quadratic, respectively, this amounts to showing that
$$
\mbox{$\mathsf{u} : \mathcal{A}_N \to L^\infty(\R; \R^3)$ is a real analytic map} \, .
$$
Indeed, let $(P,Q) \in \mathcal{A}_N$ be given. We show that $\mathsf{u}$ is real analytic in an open neighborhood around $(P,Q)$  by showing that is a restriction a complex analytic mapping. For $\eps > 0$, we consider the open set
$$
\Omega_{\eps} := \{ (P_1, P_2, Q_1, Q_2) \in \C^{4N} \mid |(P_1,P_2, Q_1, Q_2) - (P,\ov{P}, Q, \ov{Q})| < \eps \} \
$$
and the map $\tilde{\mathsf{u}} : \Omega_\eps \to L^\infty(\R; \R^3)$ defined as
\begin{align*}
\tilde{\mathsf{u}}(P_1, P_2, Q_1, Q_2)(x) & := \left ( \frac{ P_1(x) Q_2(x) + P_2(x) Q_1(x)}{P_1(x) P_2(x) + Q_1(x) Q_2(x)}, \frac{1}{2 \ii} \frac{P_1(x) Q_2(x)-P_2(x) Q_1(x)}{  P_1(x) P_2(x) + Q_1(x) Q_2(x)}, \right . \\
& \qquad \left . \frac{P_1(x) P_2(x) - Q_1(x) Q_2(x)}{P_1(x) P_2(x) + Q_1(x) Q_2(x)} \right ) \quad \mbox{with $x \in \R$} \, .
\end{align*}
Note that if $\eps > 0$ is sufficiently small, the denominator $P_1(x) P_2(x) + Q_1(x) Q_2(x) \neq 0$ for all $x \in \R$ for $(P_1,P_2,Q_1, Q_2) \in \Omega_\eps$ and hence the map $\tilde{\mathsf{u} }: \Omega_\eps \to L^\infty(\R; \R^3)$ is well-defined. Clearly, the map $\tilde{\mathsf{u}} : \Omega_\eps \to L^\infty(\R; \R^3)$ is $C^1$ and satisfies the Cauchy--Riemann equations and hence its is complex analytic. In view of the fact that
$$
\mathsf{u}(\eta,\zeta) = \tilde{\mathsf{u}}(\eta,\ov{\eta}, \zeta, \ov{\zeta}) \quad \mbox{for $(\eta, \ov{\eta}, \zeta, \ov{\zeta}) \in \Omega_\eps$} \, ,
$$
we conclude that $\mathsf{u} : \mathcal{A}_N \to L^\infty(\R; \R^3)$ is real analytic.
%Clearly, the map $(P,Q,x) \mapsto \mathsf{u}(P,Q)(x)$ is real analytic from $\mathcal{A}_N \times \R$ with values in $\R^3$, since it is given by a rational functions in $(P,Q,x)$ with no singularities. However, we need to show that $\mathsf{u} : \mathcal{A}_N \to L^\infty(\R, \R^3)$ is real analytic with values in the (real) Banach space $L^\infty(\R, \R^3)$. To this end, we notice that $\uu : \mathcal{A}_N \to L^\infty(\R, \R^3)$ is clearly $C^\infty$. Moreover, for any compact set $K \subset \mathcal{A}_N$, we readily check that there exists a constant $C=C(K) > 0$ such that
%$$
%\max_{(P,Q) \in K} \| D^{\alpha} \mathsf{u}(P,Q) \|_{L^\infty} \leq C^{|\alpha|+1} \alpha! \quad \mbox{for any multi-index $\alpha$} \ .
%$$ 
%Hence it follows that $\mathsf{u} : \mathcal{A}_N \to L^\infty(\R; \R^3)$ is real analytic, which implies that $\mathcal{K} : \mathcal{A}_N \to \mathcal{B}(L^2_+(\R; \C^2))$ is real analytic as well.

\medskip
\textbf{Step 3.} Since the image $\mathsf{K}(\mathcal{A}_N) = \mathcal{K}_N$ belongs to the subspace $\mathcal{F}_N$ of bounded operators in $\mathcal{B}(L^2_+(\R; \C^2))$ with finite rank $N$, we see that the maps 
$$
\mathcal{A}_N \to \R, \quad (P,Q) \mapsto \Tr(\mathsf{K}(U,P)^m) 
$$ 
are well-defined for any integer $m \geq 1$. In fact, these maps are real analytic as being  the composition of real analytic maps. 

Let $p_{\mathsf{K}(P,Q)}(\lambda)$ denote the characteristic polynomial of the endomorphism $\mathsf{K}(P,Q) : \Hfr_1 \to \Hfr_1$ on the $N$-dimensional subspace $\Hfr_1 = \ran (\mathsf{K}(P,Q))$. Applying the Plemelj--Smithies formula (see e.\,g.~\cite{GoGoKr-00}) in the theory of Fredholm determinants, we obtain that
$$
p_{\mathsf{K}(P,Q)}(\lambda) = \det(\lambda \id - \mathsf{K}(P,Q)) = \sum_{k=0}^N (-1)^k C_k(\mathsf{K}(P,Q)) \lambda^{N-k} \, ,
$$ 
with the coefficients
$$
C_k(\mathsf{A}) = \frac{1}{k!} \det \left ( \begin{array}{lllll} \Tr (\mathsf{A}) & k-1 & 0 & \cdots & 0 \\ \Tr (\mathsf{A}^2) & \Tr (\mathsf{A} ) & k-2 & \ddots  & \vdots \\ \vdots & \vdots & \ddots & \ddots & 0 \\
\Tr ( \mathsf{A}^{k-1}) & \Tr ( \mathsf{A}^{k-2}) & \cdots & \Tr(\mathsf{A}) & 1 \\ \Tr(\mathsf{A}^k) & \Tr(\mathsf{A}^{k-1}) & \cdots & \Tr(\mathsf{A}^2) & \Tr(\mathsf{A}) \end{array} \right ) \, ,
$$
where $k=0, \ldots, N$. This shows that the coefficients of $p_{\mathsf{K}(P,Q)}(\lambda)$ are real analytic functions of $(P,Q) \in \mathcal{A}_N$. As a consequence, the discriminant function
$$
\mathfrak{d} : \mathcal{A}_N \to \R, \quad (P,Q) \mapsto \mathfrak{d}(P,Q) := \mathrm{disc}(p_{\mathsf{K}(P,Q)}) 
$$
is also a real analytic function on the open and connected set $\mathcal{A}_N \subset \C^{2N} \cong \R^{4N}$. Moreover, we have $\mathfrak{d}(P,Q) \neq 0$ if and only if $\mathsf{K}(P,Q) : \Hfr_1 \to \Hfr_1$ has simple eigenvalues, which by the identity in Lemma \ref{lem:key}, is equivalent to having simple spectrum of $T_{\mathsf{U}(P,Q)}^2=\id - \mathsf{K}(P,Q)$ on $\Hfr_1$. Thus we find
$$
\mbox{$\mathfrak{d}(P,Q) \neq 0$ if and only if the discrete spectrum $\sigma_{\mathrm{d}}(T_{\mathsf{U}(P,Q)}^2)$ is simple} \, .
$$
Defining the set
$$
\widetilde{\mathcal{A}}_N := \{ (P,Q) \in \mathcal{A}_N \mid \mathfrak{d}(P,Q) \neq 0 \} \, ,
$$
we conclude from the real analyticity of the function $\mathfrak{d}$ on the connected set $\mathcal{A}_N$ that either
$$
\mbox{$\widetilde{\mathcal{A}}_N$ is a dense and open subset in $\mathcal{A}_N$} \ ,
$$
or it holds $\widetilde{\mathcal{A}}_N = \emptyset$, in which case we must have $\mathfrak{d} \equiv 0$ on $\mathcal{A}_N$. However, by an explicit construction in Lemma \ref{lem:T_simple_exist} below, we conclude that $\mathfrak{d} \not \equiv 0$ on $\mathcal{A}_N$. Hence we have shown that 
$$
\mbox{$\sigma_{\mathrm{d}}(T_{\mathsf{U}(P,Q)}^2)$ is simple for all $(P,Q) \in \widetilde{\mathcal{A}}_N$} 
$$
with some dense and open subset $\widetilde{\mathcal{A}}_N \subset \mathcal{A}_N$. Note that, by self-adjointness of $T_{\Ub}$, the simplicity of $\sigma_{\mathrm{d}}(T_{\Ub}^2)$ implies that $\sigma_{\mathrm{d}}(T_{\Ub})$ is simple as well. Hence we deduce that
$$
\mbox{$\sigma_{\mathrm{d}}(T_{\mathsf{U}(P,Q)})$ is simple for all $(P,Q) \in \widetilde{\mathcal{A}}_N$} \, .
$$

\medskip
\textbf{Step 4.} We are now ready to finish the proof of Theorem \ref{thm:simple}. Let $\uu \in \mathcal{R}at(\R; \Ss^2)$ be given. Note that $\lim_{x \to \pm \infty}\uu(x) = \mathbf{p}$ for some unit vector $\mathbf{p} \in \Ss^2$. By rotational symmetry, we can henceforth assume that
$$
\mathbf{p} = \mathbf{e}_3 \, .
$$
Let $N = \mathrm{Rank}(K_\Ub)$ and $\Hfr_1 = \ran(K_\Ub)$. If $N=0$ (which corresponds to the constant map $\uu \equiv \mathbf{e}_3$) then $\dim \Hfr_1 = 0$ and thus $\sigma_{\mathrm{d}}(T_{\Ub}) = \emptyset$ which is trivially simple. Also if $N=1$, we have $\dim \Hfr_1 = 1$ and thus $\sigma_{\mathrm{d}}(T_{\Ub})$ is evidently simple.

Henceforth we assume that $N \geq 2$ holds. Note that there is a (unique) point $(P,Q) \in \mathcal{A}_N$ such that
$$
\Ub = \mathsf{U}(P,Q) \quad \mbox{and} \quad K_\Ub = \mathsf{K}(P,Q) \in \mathcal{K}_N \, .
$$
By density $\widetilde{\mathcal{A}}_N \subset \mathcal{A}_N$, we can find a sequence $(P_k, Q_k) \in \widetilde{\mathcal{A}}_N$ such that $(P_k, Q_k) \to (P,Q)$ in $\C^{2N}$. Letting $\Ub_k = \mathsf{U}(P_k,Q_k)$, we conclude that
$$
\mbox{$\sigma_{\mathrm{d}}(T_{\Ub_k})$ is simple for all $k \in \N$} \, .
$$ 
Moreover, from $(P_k,Q_k) \to (P,Q)$ in $\C^{2N}$ it is easy to see that $\| \Ub_k - \Ub \|_{\dot{H}^{\frac 1 2}} \to 0$ as $k \to \infty$. Equivalently, in terms of the rational functions $\uu_k =( u_{k,1}, u_{k_2}, u_{k,3}) \in \mathcal{R}at(\R; \Ss^2)$ with
$$
\uu_{k,j} =  \frac{1}{2} \Tr_{\C^2} (\Ub_k \sigma_j) \quad \mbox{for $j=1,2,3$ and $k \in \N$},
$$
we deduce that $\| \uu_k - \uu \|_{\dot{H}^{\frac 1 2}} \to 0$ as $k \to \infty$. This proves  the density of  $\mathcal{R}at_{\mathrm{s}}(\R; \Ss^2) \subset \mathcal{R}at(\R; \Ss^2)$ as stated above. The proof of Theorem \ref{thm:simple} is now complete.
\end{proof}

%%%%%%%%%%%% Appendix A

\begin{appendix} 

\section{Density of Rational Maps} \label{sec:app_density}

Let $d \geq 2$ and $0 \leq k \leq d$ be given integers. Recall that
$$
\mathcal{R}at(\R; \Gr_k(\C^d)) = \left \{ \Ub : \R \to \Gr_k(\C^d) \mid \mbox{$\Ub(x)$ is rational in $x\in \R$} \right \}
$$
denotes the set of rational maps from $\R$ into the complex Grassmannian $\Gr_k(\C^d)$, which we identify with the set of matrices\footnote{Recall also that, via $U = \mathds{1}_d - 2 P$, we have the canonical equivalence $\Gr_k(\C^d) \cong \{ P \in \C^{d \times d} \mid \mbox{$P^* = P = P^2$ and $\Tr(P) = k$} \}$ in terms of self-adjoint projections $P$ on $\C^d$ with $\rank(P)=k$.}
$$
\Gr_k(\C^d) = \{ U \in \C^{d \times d} \mid \mbox{$U^* = U$, $U^2 = \mathds{1}_d$ and $\Tr(U) = d-2k$} \} \, .
$$
Furthermore, we recall the space
$$
\dot{H}^{\frac 1 2}(\R; \Gr_k(\C^d)) = \left  \{ \Ub \in \dot{H}^{\frac 1 2}(\R; \C^{d \times d}) \mid \mbox{$\Ub(x) \in \Gr_k(\C^d)$ for a.\,e.~$x \in \R$} \right \} \, ,
$$
equipped with Gagliardo semi-norm $\| \cdot \|_{\dot{H}^{\frac 1 2}}$ given through 
$$
\| \Ub \|_{\dot{H}^{\frac 1 2}}^2 = \| |D|^{\frac 1 2} \Ub \|_{L^2}^2 =  \frac{1}{2 \pi} \int_{\R} \int_{\R} \frac{|\Ub(x)- \Ub(y)|_F^2}{|x-y|^2} \, dx \, dy \, ,
$$
where $|A|_F = (\Tr(A^* A))^{1/2}$ denotes the Frobenius norm of a matrix $A \in \C^{d \times d}$. 

\begin{thm} \label{thm:app_density}
 $\mathcal{R}at(\R; \Gr_k(\C^d))$ is dense in $\dot{H}^{\frac 1 2}(\R; \Gr_k(\C^d))$. That is, for every $\Ub \in \dot{H}^{\frac 1 2}(\R; \Gr_k(\C^d))$, there exists a sequence $\Ub_n \in \mathcal{R}at(\R; \Gr_k(\C^d))$ such that $\| \Ub_n - \Ub \|_{\dot{H}^{\frac 1 2}} \to 0$ as $n \to \infty$.
\end{thm}

Before we give the proof of Theorem \ref{thm:app_density} below, we obtain the following fact.

\begin{cor}
$\mathcal{R}at(\R; \Ss^2)$ is  dense in $\dot{H}^{\frac 1 2}(\R; \Ss^2)$.
\end{cor}

\begin{proof}
By Theorem \ref{thm:app_density}, the set $\mathcal{R}at(\R; \Gr_1(\C^2))$ is dense in $\dot{H}^{\frac 1 2}(\R; \Gr_1(\C^2))$. Recall that, thanks to  the linear relation $\Ub = \uu \cdot \bm{\sigma}$ with $\bm{\sigma}=(\sigma_1, \sigma_2, \sigma_3)$ denoting the standard Pauli matrices, we easily check the equivalence of norms $\| \Ub \|_{\dot{H}^{\frac 1 2}} \sim \| \uu \|_{\dot{H}^{\frac 1 2}}$ and we thus conclude.
\end{proof}

Next, we turn to the proof of Theorem \ref{thm:app_density}. Here it is convenient to first prove the corresponding result in the periodic setting as follows. Let $\T = \R/2 \pi \Z$ denote the one-dimensional torus. Correspondingly, we define the space
$$
H^{\frac 1 2}(\T; \Gr_k(\C^d)) := \{ \Ub \in H^{\frac 1 2}(\T; \C^{d \times d}) \mid \mbox{$\Ub(t) \in \Gr_k(\C^d)$ for a.\,e.~$t \in \T$} \} \, ,
$$
endowed with the $H^{\frac 1 2}$-norm for maps from $\T$ into $\C^{d \times d}$. Likewise, we also define
$$
\mathcal{R}at(\T; \Gr_k(\C^d)) := \{ \Ub : \T \to \Gr_k(\C^d) \mid \mbox{$\Ub(t)$ is rational in $z = \eu^{\ii t}$ with $t \in \T$} \} \, .
$$
It is easy to see that $\mathcal{R}at(\T; \Gr_k(\C^d)) \subset H^{\frac 1 2}(\T; \Gr_k(\C^d))$ holds. In fact, we will show the following result.

\begin{thm} \label{thm:dense_torus}
$\mathcal{R}at(\T; \Gr_k(\C^d))$ is dense in $H^{\frac 1 2}(\T; \Gr_k(\C^d))$. 
\end{thm}

\begin{proof}[Proof of Theorem \ref{thm:dense_torus}]
First, we recall the following general result due to Brezis--Nirenberg \cite{BrNi-95} for Sobolev spaces of functions with values in smooth and closed (i.\,e.~compact with no boundary) manifolds. Indeed, we have that $\Gr_k(\C^d)$ is a smooth and closed manifold of real dimension $2k(d-k)$. Now from \cite{BrNi-95}[Lemma A.12] we obtain the following result; see also \cite{MaSc-23}[Section 2] for a recent and detailed discussion of density of smooth maps in Sobolev spaces in the setting of manifolds.

\begin{lem} \label{lem:VMO}
$C^\infty(\T; \Gr_k(\C^d))$ is dense in $H^{\frac 1 2}(\T; \Gr_k(\C^d))$.
\end{lem}

To complete the proof of Theorem \ref{thm:dense_torus}, it remains to establish the following result.

\begin{lem} \label{lem:rational}
For every $\Ub \in C^\infty(\T; \Gr_k(\C^d))$, there exists a sequence $$\Ub_N \in \mathcal{R}at(\T; \Gr_k(\C^d))$$ such that $\| \Ub_N - \Ub \|_{H^{\frac 1 2}} \to 0$ as $N \to \infty$. 
\end{lem}

\begin{remark*}
The proof below can actually be used to prove density with respect to the $\| \cdot \|_{H^s}$-norm for all $s \geq 0$.
\end{remark*}

\begin{proof}[Proof of Lemma \ref{lem:rational}]
Let $\Ub \in C^\infty(\T; \Gr_k(\C^d))$ be given. We define the map $\Pb \in C^\infty(\T; \Gr_k(\C^d))$ by setting $\Pb(t) := \frac{1}{2} (\mathds{1}_d - \Ub(t))$. We have 
$$
\Pb(t) = \Pb(t)^* = \Pb(t)^2 \quad \mbox{and} \quad \rank(\Pb(t)) = k \quad \mbox{for all $t \in \T$}.
$$
We claim that there exists a smooth map $\Gb \in C^\infty(\T; \C^{d \times k})$ such that
\be \label{eq:aux_P}
 \Pb(t) \Gb(t) = \Gb(t) \quad \mbox{and} \quad \rank(\Gb(t)) = k  \quad \mbox{for $t \in \T$} \, .
\ee

To prove this claim, we use a result in \cite{Si-65}[Theorem 6], where the following result is shown (up to trivial modifications of notation and changing the period of 1 to $2\pi$).
\begin{prop} \label{prop:sibuya}
Let $\Ab \in C^\infty(\R; \C^{d \times d})$ with $\Ab(t+2\pi) = \Ab(t)$ for all $t \in \R$ and assume that 
$$
\rank(\Ab(t)) = m \quad \mbox{for all $t \in \R$}
$$
with some constant $m \leq d$. Then there exists $\Bb \in C^\infty(\R; \C^{d \times (d-m)})$ such that
$$
\Bb(t+2 \pi) = \Bb(t), \quad \Ab(t) \Bb(t) = 0, \quad \rank(\Bb(t))=d-m \quad \mbox{for $t \in \R$} \, .
$$
\end{prop}
By applying Proposition \ref{prop:sibuya} to $\Ab(t) = \mathds{1}_d - \Pb(t)$ where $m=d-k$, we complete the proof of claim \eqref{eq:aux_P} by setting $\Gb(t) := \Bb(t)$.

Let us now return to the proof of Lemma \ref{lem:rational}. We claim that
\be \label{eq:id2}
\Pb(t) = \Gb(t) [\Gb(t)^* \Gb(t)]^{-1} \Gb(t)^* \quad \mbox{for $t \in \T$} \, ,
\ee
Note that, since $\rank(\Gb(t)) = k$ for $\Gb(t) \in \C^{d \times k}$, we obtain that $\Gb(t)^* \Gb(t) \in \C^{k \times k}$ is invertible for any $t \in \T$. 

To show \eqref{eq:id2}, let $\tilde{\Pb}(t)$ denote its right-hand side. Evidently, we have $\tilde{\Pb}(t)^* = \tilde{\Pb}(t)$ and $\tilde{\Pb}(t) = \tilde{\Pb}(t)^2$. 

Notice that $v \in \ker(\tilde{\Pb}(t))$ if and only if $(\Gb(t)^* \Gb(t))^{-1}(\Gb(t)^* v) \in \ker(\Gb(t))$. Hence $\ker(\tilde{\Pb}(t))=\ker(\Gb(t)^*)$ and by orthogonal complements we find $\ran(\tilde{\Pb}(t)) = \ran(\Gb(t))$. 

On the other hand, we have $\rank(\Pb(t))=k=\rank(\Gb(t))$ and $\ran(\Gb(t)) \subset \ran(\Pb(t))$ since $\Pb(t) \Gb(t) = \Gb(t)$. Hence $\ran(\Pb(t))=\ran(\Gb(t))$. 

We readily conclude that $\ran (\tilde{\Pb}(t)) = \ran(\Pb(t))$. But this implies that the self-adjoint projections $\tilde{\Pb}(t)$ and $\Pb(t)$ must be identical. Hence \eqref{eq:id2} holds true.

For $N \in \N$, we let $\Gb_N(t)$ be the truncated Fourier series of $\Gb \in C^\infty(\T; \C^{d \times k})$, i.\,e.,
$$
\Gb_N(t) = \sum_{|n| \leq N} \widehat{\Gb}_n \eu^{\ii n t} 
$$
with coefficients $\widehat{\Gb}_n = \frac{1}{2 \pi} \int_0^{2 \pi} \Gb(t) \eu^{-\ii n t} \, dt \in \C^{d \times k}$ for $n \in \Z$. Clearly, we have $\Gb_N \in \mathcal{R}at(\T; \C^{d \times k})$ together with the fact that
\be \label{eq:conv:GN}
\| \Gb_N - \Gb \|_{H^1} \to 0 \quad \mbox{as} \quad N \to \infty \, .
\ee
By Sobolev embeddings, we have the uniform convergence $\| \Gb_N - \Gb \|_{L^\infty} \to 0$ as $N \to \infty$. Recall that $\Gb(t)^* \Gb(t) \in \C^{k \times k}$ is invertible for all $t \in \T$. Thus we deduce 
$$
\mbox{$\Gb_N(t)^* \Gb_N(t) \in \C^{k \times k}$ is invertible for all $t \in \T$ and $N \geq N_0$} \, ,
$$
with some sufficiently large constant $N_0 \geq 1$. Also, this shows that $\rank(\Gb_N(t)) = k$ for all $t \in \T$ and $N \geq N_0$.

For $N \geq N_0$, we now define the sequence $\Pb_N : \T \to \C^{d \times d}$ by
$$
\Pb_N(t) := \Gb_N(t) [ \Gb_N(t)^* \Gb_N(t) ]^{-1} \Gb_N(t)^* \, .
$$
Evidently, we have $\Pb_N(t) = \Pb_N(t)^* = \Pb_N(t)^2$ for any $t \in \T$. Moreover, we find that $\rank (\Pb_N(t)) = k$ for $t \in \T$ and $N \geq N_0$. Thus $\Pb_N : \T \to \Gr_k(\C^d)$ for all $N \geq N_0$.

Now, recall that $\Gb_N \in \mathcal{R}at(\T; \C^{d \times k})$. But this implies that the right-hand side in the definition of the maps $\Pb_N(t)$ is also rational in $z= \eu^{\ii t} \in \Ss$, i.\,e., we have
$$
\Pb_N \in \mathcal{R}at(\T; \Gr_k(\C^d)) \quad \mbox{for all $N \geq N_0$} \, .
$$
Now, from the convergence \eqref{eq:conv:GN} together with the fact that the Sobolev space $H^1(\T)$ is an algebra, it is straightforward to derive
\begin{align*}
\Pb_N(t) & =  \Gb_N(t) [ \Gb_N(t)^* \Gb_N(t) ]^{-1} \Gb_N(t)^* \\
& \quad \to \Gb(t) [\Gb(t)^* \Gb(t)]^{-1} \Gb(t) = \Pb(t) \quad \mbox{in $H^1(\T; \C^{d \times d})$} \, .
\end{align*}
Thanks to the elementary embedding $H^{1} \subset H^{\frac 1 2}$ this implies that $\| \Pb_N - \Pb \|_{H^{\frac 1 2}} \to 0$ as $N \to \infty$. 

Finally, we see that the sequence $\Ub_N= \mathds{1}_d - 2 \Pb_N \in \mathcal{R}at(\T; \Gr_k(\C^d))$ satisfies $\| \Ub_N - \Ub \|_{H^{\frac 1 2}} = 2 \| \Pb_N - \Pb \|_{H^{\frac 1 2}} \to 0$ as $N \to \infty$. The proof of Lemma \ref{lem:rational} is now complete.
\end{proof}

The proof of Theorem \ref{thm:dense_torus} now follows immediately from Lemmas \ref{lem:VMO} and \ref{lem:rational}.
\end{proof}

With the help of Theorem \ref{thm:dense_torus}, we are now ready to prove Theorem \ref{thm:app_density}.

\begin{proof}[Proof of Theorem \ref{thm:app_density}]
We will make use of the known conformal invariance of the Gagliardo semi-norm $\| \cdot \|_{\dot{H}^{\frac 1 2}}$. In what follows, we will identify maps defined on $\T$ as maps defined on $\Ss^1$ by means of $z= \eu^{\ii t} \in \Ss^1$ with $t \in \T$.

Let
$$
\mathcal{S} : \R \to \Ss^1 \setminus \{ \ii \} ,\quad x \mapsto \ii \frac{x-\ii}{x+\ii} \, .
$$
denote the inverse stereographic projection from $\R$ to $\Ss \setminus \{ \ii \}$. Assume that $\Ub : \R \to \Gr_k(\C^d)$ and $\tilde{\Ub} : \Ss \to \Gr_k(\C^d)$ are related by $\Ub = \tilde{\Ub} \circ \mathcal{S}$. A well-known calculation\footnote{This can be traced back to J.~Douglas' seminal work on the Plateau problem \cite{Do-31}.} shows that
\begin{align*}
\| \Ub \|_{\dot{H}^{\frac 1 2}(\R)}^2  & = \frac{1}{2 \pi} \int_{\R} \int_\R \frac{|\Ub(x) - \Ub(y)|_F^2}{|x-y|^2 } \, dx \, dy \\ 
& = \frac{1}{2 \pi} \int_{\T} \int_{\T} \frac{|\tilde{\Ub}(t)- \tilde{\Ub}(s)|_F^2}{2-2 \cos(t-s) } \, dt \, ds = \| \tilde{\Ub} \|_{\dot{H}^{\frac 1 2}(\T)}^2
\end{align*}
Thus for a given map $\Ub \in \dot{H}^{\frac{1}{2}}(\R; \Gr_k(\C^d))$, we set $\tilde{\Ub}(z)  = (\Ub \circ \mathcal{S}^{-1})(z)$ which is defined for almost every $z \in \Ss$. Then $\tilde{\Ub} \in H^{\frac 1 2}(\T; \Gr_k(\C^d))$ by the above integral identity. By Theorem \ref{thm:dense_torus}, there exists a sequence $\tilde{\Ub}_N \in \mathcal{R}at(\T; \Gr_k(\C^d))$ with
$$
0 \leq \| \tilde{\Ub}_N - \tilde{\Ub} \|_{\dot{H}^{\frac 1 2}(\T)} \leq  \| \tilde{\Ub}_N - \tilde{\Ub} \|_{H^{\frac 1 2}(\T)} \to 0 \quad \mbox{as} \quad N \to \infty \, .
$$
Note that the sequence of functions
$$
\Ub_N := \tilde{\Ub}_N \circ \mathcal{S} \in \mathcal{R}at(\R; \Gr_k(\C^d)) \, 
$$
since $\mathcal{S}$ preserves rationality. Finally, we deduce that
$$
\| \Ub_N - \Ub \|_{\dot{H}^{\frac 1 2}(\R)} = \| \tilde{\Ub}_N - \tilde{\Ub} \|_{\dot{H}^{\frac 1 2}(\T)} \to 0 \quad \mbox{as} \quad N \to \infty \, .
$$
This completes the proof of Theorem \ref{thm:app_density}. 
\end{proof}

%%%%%%%%%%%%%%%%% Section Stereographic Parametrization

\section{Stereographic Parametrization} \label{app:stereo}

In this section, we give the proof of Theorem \ref{thm:toeplitz_stereo}. Hence we always assume that $\uu : \R \to \Ss^2$ is a rational map and, as usual, we denote $\Ub = \uu \cdot \bm{\sigma} : \R \to \Gr_1(\C^2)$ for the corresponding rational matrix-valued map. Note that here we always consider $T_{\Ub} : L^2_+(\R; \C^2) \to L^2_+(\R; \C^2)$, i.\,e., we take $\VV = \C^2$. Also, the operators $H_\Ub$ and  $K_\Ub = H_{\Ub}^* H_\Ub$ are always understood as acting on $L^2_+(\R; \C^2)$ in what follows.

We first collect some auxiliary results as follows. 

\begin{lem} \label{lem:simple_poles}
Assume $\uu : \R \to \Ss^2$ is a rational function of the form
$$
\uu(x) = \uu_\infty + \sum_{j=1}^N \left ( \frac{\sv_j}{x-z_j} + \frac{\ov{\sv}_j}{x-\ov{z}_j} \right )
$$
with some integer $N \geq 1$, $\uu_\infty \in \Ss^2$, $\sv_1, \ldots, \sv_N \in \C^3 \setminus \{ 0 \}$, and pairwise distinct poles $z_1, \ldots, z_N \in \C_-$. Then it holds $\mathrm{Rank} (K_{\Ub}) = N$.
\end{lem}

\begin{remark*}
By a straightforward extension of the proof below, we obtain the following result: Let $\Ub \in \mathcal{R}at(\R; \Gr_k(\C^d))$ be of the form
$$
\Ub(x) = \Ub_\infty + \sum_{j=1}^N \frac{A_j}{x-z_j} + \sum_{j=1}^N \frac{A_j^*}{x-\ov{z}_j}
$$
with some integer $N \geq 1$, $\Ub_\infty \in \Gr_k(\C^d)$, non-zero matrices $A_1, \ldots, A_N \in M_d(\C)$ with $A_j^2=0$ and $\rank(A_j)=1$, and pairwise distinct poles $z_1, \ldots, z_N \in \C_-$. Then we have $\rank(K_{\Ub}) = N$ for the operator $K_{\Ub} = H_{\Ub}^* H_{\Ub} : L^2_+(\R; \C^d) \to L^2_+(\R; \C^d)$.
\end{remark*}

\begin{proof}
Since $K_{\Ub} = H_{\Ub}^* H_\Ub$, we have $\dim \ran(H_{\Ub}^*) = \dim \ran (K_\Ub)$. Therefore, we need to determine the rank of the adjoint Hankel operator $H_\Ub^* : L^2_-(\R; \C^2) \to L^2_+(\R; \C^2)$ with
$$
H_{\Ub}^*f = \Pi_+ (\Ub f) = \Pi_+ \left ( \sum_{j=1}^N \frac{A_j}{x-z_j} f \right ) \quad \mbox{for $f \in L^2_-(\R; \C^2)$}
$$
with the matrices $A_j = \sv_j \cdot \bm{\sigma} \in M_2(\C)$.  From the constraint $\uu(x) \cdot \uu(x) = 1$, we readily deduce that the non-zero vectors $\sv_j \in \C^3 \setminus \{ 0 \}$ satisfy $\sv_j \cdot \sv_j = 0$ for all $j =1, \ldots, N$. To see this, we recall $\uu_\infty \cdot \uu_\infty =1$ and that the poles $\{ z_j \}_{j=1}^N$ are pairwise distinct, so that an elementary expansion in partial fractions yields
\begin{align*}
1 & = \uu(x) \cdot \uu(x) = \left (   \uu_\infty + \sum_{j=1}^N \left ( \frac{\sv_j}{x-z_j} + \frac{\ov{\sv}_j}{x-\ov{z}_j} \right ) \right ) \cdot \left (  \uu_\infty + \sum_{k=1}^N \left ( \frac{\sv_k}{x-z_k} + \frac{\ov{\sv}_k}{x-\ov{z}_k} \right ) \right ) \\
& = 1 + \sum_{j=1}^N \frac{\sv_j \cdot \sv_j}{(x-z_j)^2} + \mbox{rational terms not containing $\displaystyle \frac{1}{(x-z_j)^{2}}$ for any $j=1, \ldots, N$}.
\end{align*}
Hence we conclude that $\sv_j \cdot \sv_j = 0$ for all $j=1, \ldots, N$. Next, by elementary algebra for the Pauli matrices, we find $A_j^2=(\sv_j \cdot \sv_j) \mathds{1}_2 =0$ and hence each matrix $A_j \in M_2(\C)$ has exactly rank one. On the other hand, we easily verify that
$$
 \Pi_+ \left ( \frac{1}{x-\zeta} f \right ) = \frac{f(\zeta)}{x-\zeta} \quad \mbox{for $f \in L^2_-(\R;\C^2)$ and $\zeta  \in \C_-$} \, .
$$
In particular, we see that $f \mapsto \Pi_+((x-\zeta)^{-1} f)$ has rank one for $\zeta \in \C_-$. Since each matrix $A_j^*$ has rank one and in view of
$$
H^*_{\Ub} f = \Pi_+ (\Ub f) = \sum_{j=1}^N A_j \Pi_+\left ( \frac{1}{(x-z_j)} f \right ) \, ,
$$
we deduce the upper bound $\rank(H^*_\Ub) \leq N$.

It remains to show that $\rank(H^*_{\Ub}) \geq N$ holds. Take vectors $v_j \in \C^2$ with $A_j v_j \neq 0$ for $j=1, \ldots, N$. Now we consider the functions $f_1, \ldots, f_N \in L^2_-(\R;\C^2)$ given by
$$
f_j(x) = \prod_{k=1, \, k \neq j}^N \frac{x-z_k}{x-\ov{z}_k} \frac{v_j}{x-\ov{z}_j} \, .
$$
An explicit calculation shows that
$$
H_{\Ub}^* f_j = \frac{A_j f_j(z_j)}{x-z_j} \, .
$$
Since $A_j f_j(z_j) \neq 0$ and $z_1, \ldots, z_N \in \C_-$ are pairwise distinct, we see that $\rank(H_{\Ub}^*) \geq N$.  This completes the proof.
\end{proof}

The next lemma addresses the case of non-simple poles occurring in the rational map $\uu : \R \to \Ss^2$ and we derive a lower bound for $\rank(K_\Ub)$.

\begin{lem} \label{lem:mult_poles}
Suppose that $\uu : \R \to \Ss^2$ is of the form
$$
\uu(x) = \uu_\infty + \sum_{j=1}^p \sum_{k=1}^{m_j} \left ( \frac{\sv_{j,k}}{(x-z_j)^k} + \frac{\ov{\sv}_{j,k}}{(x-\ov{z}_j)^k} \right )
$$
with some integers $N \geq 1$, $1 \leq p,m_j \leq N$, vectors $\sv_{j,k} \in \C^3 \setminus \{ 0 \}$, and pairwise distinct $z_1, \ldots, z_N \in \C_-$. Then it holds that
$$
\mathrm{Rank}(K_\Ub) \geq N = \sum_{j=1}^p m_j \, .
$$
\end{lem}

\begin{proof}
As before, we need to bound $\rank(H^*_\Ub)$.  We adapt the second part of the proof of Lemma \ref{lem:simple_poles} as follows. For any $\zeta \in \C_-$ and any integer $k \geq 1$, we obtain by Taylor's formula that
$$
\Pi_+\left ( \frac{1}{(x-\zeta)^k} f \right ) = \sum_{\ell=0}^{k-1} \frac{f^{(\ell)}(\zeta)}{\ell! (x-\zeta)^{k-l}} 
$$
for any $f  \in L^2_-(\R; \C^2)$. Now we choose $j \in \{ 1, \ldots, p \}$ and $k \in \{ 1, \ldots, m_j \}$. We claim that there exists $f_{j,k} \in L^2_-(\R; \C^2)$ such that
$$
f_{j,k}^{(\ell)} (z_i) = 0 \quad \mbox{for $i \neq j$ and $\ell \in \{0, \ldots, m_i-1 \}$} \, ,
$$
$$
\mbox{$f_{j,k}^{(\ell)}(z_j) = 0$ and $A_{j,m_j} f^{(k-1)}_{j,k}(z_j) \neq 0$} \quad \mbox{for $\ell \in \{0, \ldots, m_j-1 \}$ and $\ell \neq k-1$} \, ,
$$
with the rank-one matrices $A_{j,k} = \sv_{j,k} \cdot \bm{\sigma} \in M_2(\C)$ . Indeed, just choose
$$
f_{j,k}(x) = \prod_{i=1, \, i \neq j}^N \left ( \frac{x-z_i}{x-\ov{z}_i} \right ) \sum_{r=k}^{m_j} \frac{(x-z_j)^{r-1}}{(x-\ov{z}_j)^{r}} v_{j,k,r}
$$
with non-zero vectors $v_{j,k,r} \in \C^2$ such that $A_{j,m_j} v_{j,k,k} \neq 0$ and with the other $v_{j,k,r}$ determined by induction on $r$. Then
$$
H^*_{\Ub}(f_{j,k}) = \frac{1}{(k-1)!} \sum_{r=k}^{m_j} \frac{A_{j,r} f_{j,k}(z_j)}{(x-z_j)^{r-k+1}} \, .
$$
It remains to observe that these rational functions are linearly independent as $j \in \{1, \ldots, p \}$ and $k \in \{1, \ldots, m_j \}$, which is elementary in view of the leading singularity in $(x-z_j)$.
\end{proof}

We are now ready to give the proof of Theorem \ref{thm:toeplitz_stereo}. For the reader's convenience, we recall the statement of Theorem \ref{thm:toeplitz_stereo}, which is now labeled as Theorem B.1 here.

\begin{thm} \label{thm:toeplitz_stereo_app}
Let $\uu =(u_1, u_2, u_3) : \R \to \Ss^2$ be a rational map. Given an integer $N \geq 1$, the following statements are equivalent.
\begin{itemize}
\item[(i)] $\dim \Hfr_1 = \rank(K_{\Ub}) = N$.
\item[(ii)] The least common denominator of $u_1, u_2, u_3$ has degree $2N$.
\item[(iii)] There exists a rational function $R \in \C(X)$ of the form
$$
R(x) = \frac{P(x)}{Q(x)} \, ,
$$
where $P \in \C[X]$ is a polynomial of degree $N$ and $Q \in \C[X]$ is a non-zero polynomial of degree at most $N-1$, such that $P$ and $Q$ have no common factor such that, up to rotation on the sphere $\Ss^2$, we have
$$
u_1(x) + \ii u_2(x) = \frac{2 R(x)}{R(x) \ov{R}(x)+1} \, , \quad u_3(x) = \frac{R(x) \ov{R}(x) -1}{R(x) \ov{R}(x) + 1} \, .
$$
\end{itemize}
\end{thm}

\begin{proof}[Proof of Theorem \ref{thm:toeplitz_stereo_app}]
We divide the proof into the following steps.

\medskip
\textbf{$(ii) \Rightarrow (iii)$}. Assume $\uu = (u_1, u_2, u_3) : \R \to \Ss^2$ is a rational map with the least common denominator given by a polynomial $D \in \R[X]$ of degree $2N$. (Note that $D$ must have even degree, since $u_1, u_2, u_3$ are real-valued rational functions with no poles in $\R$.) Moreover, up to a rotation on $\Ss^2$, we may assume that $u_3(x) \to 1$ as $|x| \to \infty$, so that there exist polynomials $Q_j \in \R[X]$ such that
$$
u_j(x) = \frac{Q_j(x)}{D(x)} \quad \mbox{for $j=1,2,3$},
$$
where $Q_1, Q_2$ have degree at most $2N-1$ and $Q_3$ has degree $2N$, with the same leading coefficient as $D$. Now the condition $u_1^2 + u_2^2 + u_3^2 = 1$ means that
$$
Q_1^2 + Q_2^2 + Q_3^2 = D^2 \, , 
$$
or equivalently
$$
(Q_1 + \ii Q_2)(Q_1 - \ii Q_2) = (D-Q_3)(D+Q_3) \, .
$$
Since $Q_3$ and $D$ have the same leading coefficient, the degree of $D+Q_3$ is $2N$ and the degree $\delta$ of $D-Q_3$ is at most $2N-1$. Denote by $d$ the degree of $Q_1 + \ii Q_2$. Since $Q_1$ and $Q_2$ are real polynomials, $d$ is also the degree of $Q_1 - \ii Q_2$ and hence
$$
2d = \delta + 2N \, .
$$
This implies
$$
N \leq d \leq 2N-1 \, .
$$
Furthermore, we recall that $D$ is the least common denominator of $u_1, u_2, u_3$ which means that $Q_1, Q_2, Q_3, D$ have no common factor, or equivalently the polynomials
$$
Q_1+ \ii Q_2, Q_1 - \ii Q_2, D-Q_3, D + Q_3
$$
have no common factor.

Now, we claim that $Q_1 + \ii Q_2$ and $D-Q_3$ have at least $d-N$ common zeros -- counted with multiplicities. Indeed, assume that $\alpha \in \C$ is a zero of $D-Q_3$ of multiplicity $m \geq 1$. We distinguish the following cases depending whether $\alpha \in \R$ or $\alpha \not \in \R$.

If $\alpha$ is real, then $\alpha$ is in fact a zero of $Q_1$ and $Q_2$, hence it is a zero of $Q_1+\ii Q_2$ and $Q_1 -\ii Q_2$ with the same multiplicity $\mu$. Since $\alpha$ cannot be a zero of $D+Q_3$ (otherwise $Q_1+\ii Q_2, Q_1 - \ii Q_2, D-Q_3, D+ Q_3$ would have common factor), we infer that $2 \mu = m$. 

If $\alpha$ is not real, then $\alpha$ is a zero of $Q_1 + \ii Q_2$ or $Q_1 - \ii Q_2$. Since $\ov{\alpha}$ is also a zero of the real polynomial $D-Q_3$, we can choose the zero $\beta \in \{ \alpha, \ov{\alpha} \}$ having the maximal multiplicity $\mu$ as zero of $Q_1 + \ii Q_2$. This shows $\mu \geq \frac{m}{2}$.

Summing up, we have found a common factor of $D-Q_3$ and $Q_1 + \ii Q_2$ with degree at least equal to half of the degree of $D-Q_3$, namely $d-N$. Therefore we can write
\be \label{eq:PQ}
\frac{Q_1 + \ii Q_2}{D-Q_3} = \frac{P}{Q}
\ee
where $P$ and $Q$ are polynomials in $\C[X]$ with no common factor, and $P$ has degree $d-r$, $Q$ has degree $2(d-N)-r$ for some $r \geq d-N$. Notice that $\deg P > \deg Q$.

Next, we prove that equality $r=d-N$ holds. Indeed, from \eqref{eq:PQ}, we conclude
$$
\frac{Q_1}{D} = \frac{P \ov{Q} + \ov{P} Q}{P \ov{P}+ Q \ov{Q}}, \quad \frac{Q_2}{D} = \frac{P \ov{Q}- \ov{P}Q}{\ii(P \ov{P} + Q \ov{Q})}, \quad \frac{Q_3}{D} = \frac{P \ov{P} - Q \ov{Q}}{P \ov{P}+ Q \ov{Q}} \, .
$$
This implies that $u_1, u_2, u_3$ have a common denominator of degree equal to $2 \deg P$. Hence
$$
2(d-r) \geq 2N \, ,
$$
which implies $r \leq d-N$, leading to the desired equality $r=d-N$. By defining the rational function
$$
R(x) = \frac{P(x)}{Q(x)}  \in \C(X),
$$
we conclude that (iii) holds. This completes the proof of the implication $(ii) \Rightarrow (iii)$.

\medskip
\textbf{$(iii) \Rightarrow (ii)$.} Suppose we are given a rational function
$$
R(x) = \frac{P(x)}{Q(x)} \, ,
$$
where $P$ is a polynomial of degree $N$, $Q$ is a polynomial of degree at most $N-1$, and $P, Q$ have no common factor. The formulae
$$
u_1 = \frac{R + \ov{R}}{R \ov{R}+1}, \quad u_2 = \frac{R - \ov{R}}{\ii (R \ov{R}+1)}, \quad u_3 = \frac{R \ov{R}-1}{R \ov{R}+1} 
$$
clearly define a rational map $\uu=(u_1,u_2, u_3)$ from $\R$ with values in $\Ss^2$. Furthermore, we see that $|P|^2 + |Q|^2$ is a common denominator of $u_1,u_2,u_3$ and is degree is $2N$. Let us prove that $|P|^2 + |Q|^2$ is the least common denominator of $u_1, u_2, u_3$. We argue by contradiction. Suppose there is a common factor of the polynomials
$$
P \ov{Q} + \ov{P} Q, P \ov{Q}- \ov{P} Q, P \ov{P} - Q \ov{Q}, P \ov{P}  + Q \ov{Q}
$$
or equivalently of the polynomials
$$
P \ov{Q}, \ov{P} Q, P \ov{P}, Q \ov {Q} \, .
$$
Since $P, Q$ have no common factor, there exist polynomials $U, V$ such that
$$
UP + VQ = 1 \, .
$$
Therefore the polynomials
$$
\ov{U} (P \ov{P} )+ \ov{V}(P \ov{Q}) = P, \quad \ov{U} (\ov{P} Q) + \ov{V} ( Q \ov{Q} ) = Q
$$
would have a common factor, which yields a contradiction. This proves that (iii) implies (ii).

\medskip
\textbf{$(ii) \Rightarrow (i).$} Let us assume that $\uu = (u_1, u_2, u_3) : \R \to \Ss^2$ has a least common denominator of degree $2N$. We claim that
\be \label{eq:rank_K}
\mathrm{Rank} (K_\Ub) = N \, .
\ee
Indeed, if $\uu : \R \to \Ss^2$ has only simple poles (i.\,e., the assumptions of Lemma \ref{lem:simple_poles} are satisfied), we can directly apply Lemma \ref{lem:simple_poles} to conclude that \eqref{eq:rank_K} holds.

To deal with the case multiple poles occurring in $\uu : \R \to \Ss^2$, we need the following approximation result.

\begin{lem} \label{lem:rat_approx}
Let $P \in \C[X]$ be a polynomial of degree $N \geq 1$, $Q \in  \C[X]$ be a non-zero polynomial of degree at most $N-1$, and assume that $P,Q$ have no common factor. Then there exist sequence $P_n, Q_n \in \C[X]$ such that $P_n, Q_n$ have no common factor and
$$
\deg P_n = N, \quad \deg Q_n \leq N-1, \quad P_n \to P, \quad Q_n \to Q \quad \mbox{in $\C[X]$} \, .
$$
Furthermore, the zeros of $|P_n|^2 + |Q_n|^2$ are simple for every $n \in \N$.
\end{lem}

\begin{proof}[Proof of Lemma \ref{lem:rat_approx}]
Consider the set $\mathcal{A}$ of pairs of polynomials $(P,Q) \in \C[X] \times \C[X]$ such that $\deg = N$, $\deg Q \leq N-1$ and $P,Q$ have no common factor and $P$ is monic. By Lemma \ref{lem:R_n}, we can identify $\mathcal{A}$ with a connected open subset in $\C^{2N}$. On the set $\mathcal{A}$, the condition that the discriminant of $|P|^2 + |Q|^2$ is different from 0 is an open dense subset. This completes the proof. 
\end{proof}

Suppose now that $\uu : \R \to \Ss^2$ has multiple poles and the least common denominator of $u_1, u_2, u_3$ has degree $2N$. By Lemma \ref{lem:mult_poles}, we must have
$$
\mathrm{Rank} \, (K_{\Ub}) \geq N \, .
$$

On the other hand, by the proven implication $(ii) \Rightarrow (iii)$, there exists a rational function 
$$
R(x) = \frac{P(x)}{Q(x)} \in \C(X)
$$
with $\deg P = N$, $\deg Q \leq N-1$ with $Q \not \equiv 0$ and $Q,P$ have no common factor, such that (up to rotation on $\Ss^2$) the rational function $\uu=(u_1, u_2, u_3)$ is given by the inverse stereographic projection applied to $R(x)$. Now, let us take sequences $P_n, Q_n \in \C[X]$ as provided in Lemma \ref{lem:rat_approx}. Define $R_n(x) = P_n(x)/Q_n(x) \in \C(X)$ and consider the sequence of rational maps $\uu^{(n)} : \R \to \Ss^2$ with components
$$
u_1^{(n)}(x) + \ii u_2^{(n)}(x)= \frac{2 R_n(x)}{|R_n(x)|^2 +1} , \quad u_3^{(n)}(x) = \frac{|R_n(x)|^2-1}{|R_n(x)|^2 +1} \, .
$$
Since $|P_n|^2 + |Q_n|^2$ has only simple zeros, we see that each rational map $\uu^{(n)}$ has only simple poles. Applying the known implication $(iii) \Rightarrow (ii)$, we conclude that $u_1^{(n)}, u_2^{(n)}, u_3^{(n)}$ for all $n$ have a least common denominator of degree $2N$. Thus for every rational map $\uu^{(n)} : \R \to \Ss^2$ we can apply Lemma \ref{lem:simple_poles} to conclude that
$$
\mathrm{Rank} \, (K_{\Ub_n}) = N \quad \mbox{for all $n \in \N$} \, .
$$
On the other hand, since $\uu^{(n)}(x) \to \uu(x)$ pointwise and $|\uu^{(n)}(x)| = 1$, we see that $K_{\Ub_n} f \to K_{\Ub} f$ in $L^2(\R, \C^2)$ for every $f \in L^2_+(\R, \C^2)$ by dominated convergence. From this we easily deduce that
$$
N=\liminf_{n \to \infty} \mathrm{Rank} (K_{\Ub_n}) \geq \mathrm{Rank}(K_{\Ub}) \, .
$$ 
This completes the proof that \eqref{eq:rank_K} holds whenever $u_1, u_2, u_3$ have a least common denominator of degree $2N$.

\medskip
\textbf{$(i) \Rightarrow (ii)$}.  Suppose now that $\mathrm{Rank}(K_{\Ub}) = N$ holds for some integer $N \geq 1$. Let $D \in \R[X]$ denote the least common denominator of $u_1, u_2, u_3$. Since $D$ has no zeros in $\R$, we must have that $\deg D = 2 m$ for some integer $m \geq 1$. We claim that 
$$
m=N \, .
$$ 
Indeed, if $\uu : \R \to \Ss^2$ has simple poles (in the sense of Lemma \ref{lem:simple_poles}), we can use Lemma \ref{lem:simple_poles} directly to deduce that $m=N$ must hold. 

If $\uu : \R \to \Ss^2$ has multiple poles, then $\deg D = 2m$ where $m \geq 1$ is the number poles $\uu$ counted with multiplicity. By the same argument using approximation with simple pole rational functions $\uu^{(n)} : \R \to \Ss^2$ as in the previous step, we conclude that $\mathrm{Rank}(K_{\Ub}) = m$. Hence $m=N$ is also true in this case.

The proof of Theorem \ref{thm:toeplitz_stereo_app} is now complete.
\end{proof}

%%%%%%%%%%%%%%%%%% Section Constructing T_U with Simple Discrete Spectrum

\section{Construction of $T_\Ub$ with Simple Discrete Spectrum}

\label{app:perturbation}

The aim of this section is to construct, for given $N \geq 1$, rational maps $\uu \in \mathcal{R}at(\R; \Ss^2)$ such that the corresponding Toeplitz operator $T_{\Ub} : L^2_+(\R; \C^2) \to L^2_+(\R; \C^2)$ has simple discrete spectrum 
$$
\sigma_{\mathrm{d}}(T_\Ub) = \{ v_1, \ldots, v_N \} \, ,
$$
where $v_j \in (-1, 1)$ for $j=1, \ldots, N$ are arbitrarily given simple eigenvalues. To achieve this, we will use a perturbative construction by using $N$ simple traveling solitary waves for (HWM) with different velocities $v_j \in (-1,1)$ that are sufficiently far separated from each other. 

For a rational map $\uu \in \mathcal{R}at(\R; \Ss^2)$,  we henceforth assume without loss of generality that
$$
\uu_\infty := \lim_{|x| \to \infty} \uu(x) =  \mathbf{e}_3=(0,0,1) \in \Ss^2
$$ 
by rotational symmetry on the sphere $\Ss^2$. For a given velocity $v \in (-1,1)$, we define the unit vector $\nv_v \in \Ss^2$ by setting
\be \label{eq:nv}
\nv_v := (0,\sqrt{1-v^2}, v )  \quad \mbox{so that $v = \nv \cdot \uu_\infty$} \, .
\ee
For later use, we also define the unit vectors $\nv_{v,1}, \nv_{v,2} \in \Ss^2$ with
\be \label{eq:nv2}
\nv_{v,1} := \mathbf{e}_1 = (1,0,0) \quad \mbox{and} \quad \nv_{v,2} := \nv_v \times \nv_{v,1} = (0,v, -\sqrt{1-v^2}) \, .
\ee
Thus $(\nv_v, \nv_{v,1}, \nv_{v,2})$ forms a (positively oriented) orthonormal basis of unit vectors in $\R^3$ whose use will become clear below.

Furthermore, it will be convenient to consider poles $z \in \C_-$ of the form
\be \label{eq:pole}
z = y - \ii \in \C_- \quad \mbox{with $y \in \R$} \, .
\ee
Next, we construct a rational function $\qv_{v,z} : \R \to \Ss^2$ of the form
$$
\qv_{v,z}(x) := \mathbf{e}_3 +  \frac{\sv_{v}}{x-z} + \frac{\ov{\sv}_{v}}{x-\ov{z}} 
$$
with some complex vector $\sv_v \in \C^3 \setminus \{ 0 \}$. By plugging this ansatz into the pointwise constraint $|\qv_{v,z}(x)|^2 = 1$ for $x \in \R$ and equating all terms proportional to $(x-z)^{-1}$ and $(x-z)^{-2}$ to zero, we easily find following the constraints equivalent to the condition $|\qv_{v,z}(x)|^2=1$:
\be \label{eq:cons_v}
\sv_v \cdot \sv_v = 0 \quad \mbox{and} \quad \sv_v \cdot \left ( \mathbf{e}_3 + \frac{\ov{\sv}_v}{z - \ov{z}} \right ) = 0 \, ,
\ee
where $\mathbf{a} \cdot \mathbf{b} = a_1 b_1 + a_2 b_2 + a_3 b_3$ for $\mathbf{a}, \mathbf{b} \in \C^3$.  In view of \cite{BeKlLa-20}[Lemma B.1], we make the ansatz 
$$
\sv_v = s_{v} ( \nv_{v,1} + \ii \nv_{v,2} )
$$
with some complex number $s_{v} \in \C^*$ and with the real unit vectors $\nv_{v,1}$ and $\nv_{v,2}$ from \eqref{eq:nv2} above. This automatically ensures that the first constraint in \eqref{eq:cons_v} holds. Next, by recalling that $z-\ov{z} = -2 \ii$ for the pole $z \in \C_-$, the second equation in \eqref{eq:cons_v} becomes
$$
s_v ( \nv_{v,1} + \ii \nv_{v,2}) \cdot \left ( \mathbf{e}_3 - \frac{\ov{s}_v(\nv_{v,1} -\ii \nv_{v,2})}{2 \ii} \right ) = 0 \, .
$$
Since $(\nv_{v,1}+\ii \nv_{v,2})\cdot(\nv_{v,1}-\ii \nv_{v,2}) = 2$, we readily find that $s_v \in \C^*$ is given by
\be \label{eq:sv_nice}
s_v = -\ii ( \nv_{v,1} -\ii \nv_{v,2}) \cdot \mathbf{e}_3 =  \ii \sqrt{1-v^2} \, .
\ee

In summary, we find that 
$$
\qv_{v,z}(x) = \mathbf{e}_3 + \frac{\sv_v}{x-z} + \frac{\ov{\sv}_v}{x-\ov{z}} 
$$
with
\be  \label{eq:sv}
\sv_v =  \ii \sqrt{1-v^2} ( \nv_{v,1} + \ii \nv_{v,2} ) =  \sqrt{1-v^2}  \left ( \begin{array}{c} \ii  \\ -v  \\  \sqrt{1-v^2} \end{array} \right )  \, .
\ee 
We remark that the simple pole rational function $\qv_{v,z} : \R \to \Ss^2$ yields a traveling solitary wave solution
$$
\uu(t,x) = \qv_{v,z}(x-vt)
$$
of (HWM) with velocity $v$ and $\lim_{|x| \to \infty} \uu(t,x) = \mathbf{e}_3$, which follows by a direct calculation which we omit here.

We have the following main result.

\begin{lem} \label{lem:flea_circus}
Let $N \geq 1$ be an integer and let $v_1, \ldots, v_N \in (-1,1)$ be given. Then there is a sufficiently small constant $\eps_0=\eps_0(N) > 0$ such that the following holds.

Let $z_1, \ldots, z_N \in \C_-$ be pairwise distinct poles of the form \eqref{eq:pole} and define
$$
\eps := \frac{1}{\min_{j \neq k} |z_k -z_j|} > 0 \quad \mbox{and} \quad \vec{z} = (z_1, \ldots, z_N).
$$ 
Then if $\eps < \eps_0$, there exists a rational map $\uu_{\vec{z}} : \R \to \Ss^2$ of the form
$$
\uu_{\vec{z}}(x) = \mathbf{e}_3 + \sum_{j=1}^N \frac{\sv_{j,\vec{z}}}{x-z_j} + \sum_{j=1}^N \frac{\ov{\sv}_{j,\vec{z}}}{x- \ov{z}_j} \, 
$$ 
with some $\sv_{j,\vec{z}} \in \C^3 \setminus \{ 0 \}$. Moreover, we have that
$$
 \sv_{j,\vec{z}} = \sv_{v_j} + O(\eps) \quad \mbox{for $j=1, \ldots, N$} \, ,
$$
where $\sv_{v_j} \in \C^3 \setminus \{0\}$ is given by \eqref{eq:sv} with $v=v_j$. 
\end{lem}

\begin{proof}
We arrange the proof into the following steps.

\medskip
\textbf{Step 1.} Let $z_1, \ldots, z_N \in \C_-$ be pairwise distinct with $\Im \, z_j = -1$ for $j=1, \ldots, N$ and set $\eps = 1/{\min_{j \neq k}|z_k-z_j|} > 0$. We denote $\vec{z} = (z_1, \ldots, z_N)$- For $\eps \in (0, \eps_0)$ with $\eps_0=\eps_0(N) > 0$ chosen below, we need to find $\sv_{1, \vec{z}}, \ldots, \sv_{N, \vec{z}} \in \C^3 \setminus \{ 0 \}$ such that the following nonlinear constraints are satisfied:
\be \label{eq:cons1}
\sv_{j,\vec{z}} \cdot \sv_{j,\vec{z}} = 0 \quad \mbox{for $j=1, \ldots, N$} \, ,
\ee
\be \label{eq:cons2}
\sv_{j,\vec{z}} \cdot \left (  \mathbf{e}_3 + \sum_{k \neq j}^N \frac{\sv_{k,\vec{z}}}{z_j - z_k} + \sum_{k=1}^N \frac{\ov{\sv}_{k,\vec{z}}}{z_j-\ov{z}_k} \right ) = 0 \quad \mbox{for $j=1, \ldots, N$} \, .
\ee
In fact, these conditions follow simply by a partial fraction expansion for the constraint $\uu_{\vec{z}}(x) \cdot \uu_{\vec{z}}(x) = 1$ with our ansatz for $\uu_{\vec{z}}(x)$ stated above.  As for \eqref{eq:cons1}, we recall from \cite{BeKlLa-20}[Lemma B.1] the algebraic fact that any $\sv \in \C^3 \setminus \{ 0 \}$ with $\sv \cdot \sv = 0$ can be written as
$$
\sv = s ( \nv_1 + \ii \nv_2)
$$
with a complex number $s \in \C^*$ and real unit vectors $\nv_1, \nv_2 \in \Ss^2$ such that $\nv_1 \cdot \nv_2 = 0$. In fact, this representation is unique modulo $U(1)$-rotations in the plane spanned by $\nv_1$ and $\nv_2$ with a corresponding phase rotation of $s_j$. 

Next, we define the vectors 
$$\sv_j:= \mbox{$\sv_{v_j} \in \C^3 \setminus \{0 \}$ given by \eqref{eq:sv} with $v=v_j$}
$$
and we fix corresponding real unit vectors $\nv_{j,1}, \nv_{j,2} \in \Ss^2$ as defined in \eqref{eq:nv2} with $v=v_j \in (-1,1)$. Thus we have
$$
\sv_{j} = s_j ( \nv_{j,1} + \ii \nv_{j,2})
$$ 
with some complex numbers $s_j \in \C^*$ to be determined for $j=1, \ldots, N$. 

For the vectors $\sv_{j,\vec{z}}$ to be found, we make the ansatz
$$
\sv_{j,\vec{v}} = s_{j,\vec{z}} ( \nv_{j,1} + \ii \nv_{j,2}) \quad \mbox{with $s_{j,\vec{z}} \in \C^*$} \, .
$$
Note that the vectors $\nv_{j,1}$ and $\nv_{j,2}$ are fixed and only depend on $v_j$ but not on the poles $(z_1, \ldots, z_N)$. Clearly, the first set of constraints \eqref{eq:cons1} is automatically satisfied by our ansatz for $\sv_{j,\vec{z}}$. Thus we only need to show how to solve \eqref{eq:cons2} in the rest of the proof, provided that the constant $\eps_0 =\eps_0(N) \ll 1$ is sufficiently small.

\medskip
\textbf{Step 2.} In order to solve \eqref{eq:cons2}, we devise an iteration scheme as follows inspired by the discussion in \cite{BeKlLa-20}\footnote{In \cite{BeKlLa-20}, a different sign convention for the poles $z_j$ and spin vectors $\sv_j$ are used. The reader should be aware of this when comparing with our formulae here.}. First, let us write \eqref{eq:cons2} as
$$
\sv_{j,\vec{z}} \cdot \left ( \mv_{j,\eps} + \frac{\ov{\sv}_{j,\vec{z}}}{z_j-\ov{z}_j} \right ) = 0 \quad \mbox{with} \quad \mv_{j,\vec{z}}:= \mathbf{e}_3 + \sum_{k \neq j}^N \left ( \frac{\ov{\sv}_{k,\vec{z}}}{z_j - \ov{z}_k} + \frac{\sv_{k,\vec{z}}}{z_j - z_k} \right ) \, .
$$
If we recall that $\sv_{j,\vec{z}} = s_{j,\vec{z}}(\nv_{j,1} + \ii \nv_{j,2})$ and $z_j-\ov{z}_j=-2 \ii$, we find  the equation
$$
s_{j,\vec{z}} (\nv_{j,1} + \ii \nv_{j,2}) \cdot \left (  \mv_{j,\vec{z}} - \frac{\ov{s}_{j,\vec{z}}(\nv_{j,1} - \ii \nv_{j,2})}{2 \ii} \right ) = s_{j,\vec{z}} \left ( (\nv_{j,1} + \ii \nv_{j,2}) \cdot \mv_{j,\vec{z}} + \ii \ov{s}_{j,\vec{z}} \right ) = 0,
$$
which has the unique non-trivial solution
$$
s_{j,\vec{z}} = -\ii (\nv_{j,1} - \ii \nv_{j,2}) \cdot \ov{\mv}_{j,\vec{z}}  \, .
$$
Since $\mv_{j,\vec{z}}$ does not depend on $\sv_{j,\vec{z}}$, this suggest the following iteration scheme: If $s_{j,\vec{z}}^{(n)}$ is given, we define the next iterate $s_{j,\vec{z}}^{(n+1)}$ by 
$$
s_{j,\vec{z}}^{(n+1)} := -\ii (\nv_{j,1}- \ii \nv_{j,2}) \cdot \left ( \mathbf{e}_3 + \sum_{k \neq  j}^N \left ( \frac{s_{k,\vec{z}}^{(n)}(\nv_{k,1} + \ii \nv_{k,2})}{\ov{z}_j-z_k} + \frac{\ov{s}_{k,\vec{z}}^{(n)}(\nv_{k,1} - \ii \nv_{k,2})}{\ov{z}_j - \ov{z}_k} \right ) \right ) \, .
$$
Thus we need to solve the fixed point equation
$$
\vec{s}_{\vec{z}} = F_{\vec{z}}(\vec{s}_{\vec{z}})
$$
with the variable $\vec{s}_{\vec{z}} = (s_{1,\vec{z}}, \ldots, s_{N,\vec{z}})$ and the given parameters $\vec{z} = (z_1, \ldots, z_N)$, where the map $F_{\vec{z}} : \C^{N} \to \C^N$ is defined by the right-hand side of the iteration scheme above. Recalling that $s_j = -\ii (\nv_{j,1} - \ii \nv_{j,2}) \cdot \mathbf{e}_3$ from \eqref{eq:sv_nice}, we find
$$
F_{\vec{z}}(\vec{s}_{\vec{z}}) = \vec{s} + A_{\vec{z}} (\vec{s}_{\vec{z}}) + B_{\vec{z}} (\ov{\vec{s}}_{\vec{z}}) \, ,
$$
where $\vec{s}=(s_1, \ldots, s_N) \in \C^N$ and $A_{\vec{z}}, B_{\vec{z}} : \C^N \to \C^N$ are linear maps with operator norms 
\be \label{ineq:AB}
\| A_{\vec{z}} \|_{\C^N \to \C^N} + \| B_{\vec{z}} \|_{\C^N \to \C^N} \leq C \eps \leq C \eps_0
\ee
with some constant $C> 0$ depending only on $N$. Hence by taking $\eps_0:=1/(2C)$, we see that, for any $\eps \in (0,\eps_0)$, the map $G:=\id - A_{\vec{z}} - B_{\vec{z}}(\ov{\cdot}) : \C^N \to \C^N$ is invertible by using the Neumann series. Hence $\vec{s}_\eps = G^{-1} (\vec{s})$ is the unique solution of the fixed point equation $\vec{s}_{\vec{z}} = F_{\vec{z}}(\vec{s}_{\vec{z}})$ provided that $\eps \in (0, \eps_0)$ holds.

\medskip
\textbf{Step 3.} It remains to show  that
$$
 \sv_{j,\vec{z}} = \sv_{j} + O(\eps) \quad \mbox{for $j=1, \ldots, N$}.
$$
Since $\sv_{j,\vec{z}} = s_{j,\vec{z}} ( \nv_{j,1}+\ii \nv_{j,2})$ with vectors $\nv_{j,1}, \nv_{j,2}$ independent of $\eps$, this claim is equivalent to proving that 
$$
 \vec{s}_{\vec{z}} = \vec{s} + O(\eps)
$$ 
 with the notation from \textbf{Step 2}. But from the fixed point equation and estimate \eqref{ineq:AB} we readily find
$$
\| \vec{s}_{\vec{z}} - \vec{s} \|_{\C^N} = \| F_{\vec{z}}(\vec{s}_{\vec{z}}) - \vec{s} \|_{\C^N} \leq C \eps \, .
$$
with some constant $C=C(N) > 0$. Furthermore, since $s_j \neq 0$ for all $j =1, \ldots N$, we conclude that $\sv_{j,\vec{z}} \neq 0$ for all $j=1, \ldots, N$, provided that $\eps \in (0, \eps_0)$ with $\eps_0=\eps_0(N) > 0$  sufficiently small. 

The proof of Lemma \ref{lem:flea_circus} is now complete.
\end{proof}

With the help of Lemma \ref{lem:flea_circus} we are now able to prove the following result. Recall that $\Ub = \uu \cdot \bm{\sigma}$ for a map $\uu : \R \to \Ss^2$.

\begin{lem} \label{lem:T_simple_exist}
For any integer $N \geq 0$,  there exists a rational map $\uu : \R \to \Ss^2$ with exactly $N$ simple poles such that the discrete spectrum $\sigma_{\mathrm{d}}(T_\Ub^2)$ is simple. 
\end{lem}

\begin{remark*}
Recall that, by self-adjointness of $T_\Ub$, the simplicity of $\sigma_{\mathrm{d}}(T_\Ub^2)$ implies that $\sigma_{\mathrm{d}}(T_\Ub)$ is simple.
\end{remark*}

\begin{proof}
For $N=0$, this is trivially true by taking the constant function $\uu(x) \equiv \mathbf{e}_3$ and noticing that $\mathrm{Rank}(K_\Ub)=0$ and hence $\sigma_{\mathrm{d}}(T_\Ub) = \emptyset$. If $N=1$, we take the stationary solution (i.\,e.~a half-harmonic map)
$$
\uu(x) = \left (0, \frac{2x}{x^2+1}, \frac{x^2-1}{x^2+1} \right ) \in \mathcal{R}at(\R; \Ss^2) \, ,
$$
 which has $\mathrm{Rank}(K_{\Ub}) = 1$ with simple discrete spectrum $\sigma_{\mathrm{d}}(T_\Ub) = \{0 \}$. Hence it remains to discuss the case $N \geq 2$, which will be proved in the following steps.
 
 \medskip

\textbf{Step 1.} Assume $N \geq 2$ in what follows. Let $z_1, \ldots, z_N \in \C_-$ and $v_1, \ldots, v_N \in (-1,1)$ be as in Lemma \ref{lem:flea_circus} with the additional assumption that
$$
v_j \neq v_k \quad \mbox{for $j \neq k$} \, .
$$
Consider the rational map $\uu_{\vec{z}} : \R \to \Ss^2$ given by Lemma \ref{lem:flea_circus} with $\eps = 1/\min_{j \neq k} |z_j-z_k| \in (0, \eps_0)$, where $\eps_0=\eps_0(N) > 0$ denotes the small constant from Lemma \ref{lem:flea_circus}. In particular, the rational map $\uu_{\vec{z}} : \R \to \Ss^2$ has exactly $N$ simple poles. 

Note that the rational matrix-valued function $\Ub_{\vec{z}} = \uu_{\vec{z}} \cdot \bm{\sigma}$ is given by
$$
\Ub_{\vec{z}}(x) = \sigma_3 + \sum_{j=1}^N \frac{A_{j,\vec{z}}}{x-z_j} + \sum_{j=1}^N \frac{A_{j,\vec{z}}^*}{x-\ov{z}_j} \, ,
$$ 
with the non-zero matrices $A_{j,\vec{z}} \in \C^{2 \times 2}$ given by $A_{j,\vec{z}} := \sv_{j,\vec{z}} \cdot \bm{\sigma}$. Note that $A_{j,\vec{z}}^2 = 0$ which follows from $\sv_{j,\vec{z}} \cdot \sv_{j,\vec{z}}=0$. Thus the nilpotent matrices $A_{j,\vec{z}} \in M_2(\C)$ have rank one and we can write
$$
A_{j, \vec{z}} = e_{j,\vec{z}} \langle \cdot, \xi_{j,\vec{z}} \rangle_{\C^2}
$$
with some non-zero vectors $e_{j,\vec{z}}, \xi_{j,\vec{z}} \in \C^2 \setminus \{ 0 \}$ such that 
$$
\|e_{j,\vec{z}} \|_{\C^2} = 1 \quad \mbox{and}  \quad \langle e_{j,\vec{z}}, \xi_{j, \vec{z}} \rangle_{\C^2} = 0 \, .
$$ 
Note that $\mathrm{span} \{ e_{j,\vec{z}} \} = \ran(A_{j, \vec{z}})$ and we readily check that
$$
\Hfr_1 = \ran(K_{\Ub_{\vec{z}}}) = \ran(H^*_{\Ub_{\vec{z}}}) = \mathrm{span} \left \{ \frac{e_{j,\vec{z}}}{x-z_j}  \mid j=1, \ldots, N \right \} \, 
$$
with the operator $K_{\Ub_{\vec{z}}} = H^*_{\Ub_{\vec{z}}} H_{\Ub_{\vec{z}}} : L^2_+(\R; \C^2) \to L^2_+(\R; \C^2)$.

\medskip
\textbf{Step 2.}
For later use, we recall that the constraint equations \eqref{eq:cons1} and \eqref{eq:cons2} can be rephrased in terms of matrix-valued functions as follows:
\be \label{eq:AB}
A_{j,{\vec{z}}}^2 = 0, \quad B_{j,\vec{z}} A_{j,\vec{z}} + A_{j,\vec{z}} B_{j,\vec{z}}= 0
\ee
for all $j=1, \ldots, N$, where we define the complex $2 \times 2$-matrices
\be \label{eq:AB2}
B_{j,\vec{z}} := \sigma_3  + \sum_{k\neq j}^N \frac{A_{k,\vec{z}}}{z_j - z_k} + \sum_{k =1}^N \frac{A_{k,\vec{z}}^*}{z_j - \ov{z}_k} \, .
\ee
Because of $A_{j,\vec{z}} e_{j,\vec{z}}=0$ and by \eqref{eq:AB}, we see that $A_{j,\vec{z}} B_{j,\vec{z}} e_{j,\vec{z}}=0$. Since $\ker(A_{j,\vec{z}}) = \mathrm{span} \{ e_{j,\vec{z}} \}$, we deduce  
$$
B_{j,\vec{z}} e_{j,\vec{z}} = b_{j,\vec{z}} e_{j,\vec{z}} \, 
$$
with some eigenvalue $b_{j,\vec{z}} \in \C$. Since $\| e_{j,\vec{z}} \|_{\C^2}=1$, the eigenvalue $b_{j,\vec{z}}$ is evidently given by
\begin{align}
b_{j,\vec{z}} & = \langle B_{j,\vec{z}} e_{j,\vec{z}}, e_{j,\vec{z}} \rangle_{\C^2}  \nonumber \\ \label{eq:bj} 
& = \langle \sigma_3 e_{j,\vec{z}}, e_{j,\vec{z}} \rangle_{\C^2} + \sum_{k \neq j}^N \left ( \frac{\langle A_{k,\vec{z}} e_{j,\vec{z}}, e_{j,\vec{z}} \rangle_{\C^2}}{z_j - z_k} + \frac{\langle A_{k,\vec{z}}^* e_{j,\vec{z}}, e_{j,\vec{z}} \rangle_{\C^2}}{z_j - \ov{z}_k} \right ) , 
\end{align}
where we also used the simple fact that $\langle A_{j,\vec{z}}^* e_{j,\vec{z}}, e_{j,\vec{z}}  \rangle_{\C^2} =0$ because of $A_{j,\vec{z}} e_{j,\vec{z}} =0$. Next, by partial fraction decomposition and $A_{j,\vec{z}} e_{j,\vec{z}} = 0$, we obtain
\begin{align*}
T_{\Ub_{\vec{z}}} \left ( \frac{e_{j, \vec{z}}}{x-z_j} \right ) & = \Pi_+ \left [ \left ( \sigma_3 + \sum_{k=1}^N \frac{A_{k, \vec{z}}}{x-z_k} + \sum_{k=1}^N \frac{A_{k,\vec{z}}^*}{x-\ov{z}_k} \right ) \frac{e_{j,\vec{z}}}{x-z_j}  \right ]\\
&= \frac{\sigma_3 e_{j,\vec{z}}}{x-z_j} + \sum_{k \neq j}^N \frac{A_{k,\vec{z}} e_{j,\vec{z}}}{(x-z_k)(x-z_j)} + \sum_{k=1}^N \frac{A_{k,\vec{z}}^* e_{j,\vec{z}}}{(z_j-\ov{z}_k) (x-z_j)}  \\
& = \left ( \sigma_3 + \sum_{k \neq j}^N \frac{A_{k,\vec{z}}}{z_j - z_k} + \sum_{k=1}^N \frac{A_{k,\vec{z}}^*}{z_j - \ov{z}_k} \right ) \frac{e_{j,\vec{z}}}{x-z_j} +\sum_{k \neq j}^N \frac{A_{k,\vec{z}} e_{j,\vec{z}}}{(z_k -z_j)(x-z_k)} \\
& = \frac{B_{j,\vec{z}} e_{j,\vec{z}}}{x-z_j} + \sum_{k \neq j}^N \frac{A_{k,\vec{z}} e_{j,\vec{z}}}{(z_k-z_j)(x-z_k)}  = \frac{b_{j,\vec{z}} e_{j,\vec{z}}}{x-z_j} + \sum_{k \neq j}^N \frac{A_{k,\vec{z}} e_{j,\vec{z}}}{(z_k-z_j)(x-z_k)} \, .
\end{align*}
for any $j=1, \ldots, N$ and with the eigenvalues $b_{j,\vec{z}}$ from above. Let $\mathsf{T} \in \C^{N \times N}$ denote the matrix of $T_{\Ub_{\vec{z}}} : \Hfr_1 \to \Hfr_1$ with respect to the basis $\mathcal{B} = \left ( \frac{e_{1,\vec{z}}}{x-z_1}, \ldots, \frac{e_{N,\vec{z}}}{x-z_N} \right )$. Since $\| A_{j,\vec{z}} \|_{\C^2 \to \C^2} \lesssim \|\sv_{j,\vec{z}}\|_{\C^3} \lesssim 1$, we see that the matrix $\mathsf{T} \in \C^{N \times N}$ is of the form
$$
\mathsf{T} = \mathrm{diag}(b_{1,\vec{z}}, \ldots, b_{N,\vec{z}}) + \mathsf{B}
$$
with some matrix  $\mathsf{B}=\mathsf{B}(z_1, \ldots, z_N, v_1, \ldots, v_n)$ such that
$$
\| \mathsf{B} \|_{\C^N \to \C^N} = O(\eps) \, ,
$$
where we recall that $\eps = 1/\min_{j \neq k} |z_j - z_k|$. Furthermore, from \eqref{eq:bj} we deduce that
$$
b_{j,\vec{z}} = \langle \sigma_3 e_{j,\vec{z}}, e_{j,\vec{z}} \rangle_{\C^2} + O(\eps) \, .
$$
Next, we recall that $A_{j,\vec{z}} \to A_j = \sv_j \cdot \bm{\sigma}$ as $\eps \to 0$ by Lemma \ref{lem:flea_circus}. Notice that $\sv_j$ is given by \eqref{eq:sv} with $v=v_j$ and an elementary calculation shows that $\ran(A_j) = \mathrm{span} \{ e_j \}$ with the unit vector
$$
e_j=\frac{1}{\sqrt{2}} \left ( \begin{array}{c} \sqrt{1+v_j} \\ \ii \sqrt{1-v_j} \end{array} \right ) \in \C^2 \, .
$$
 Thus we conclude that
$$
\langle \sigma_3 e_{j,\vec{z}}, e_{j, \vec{z}} \rangle_{\C^2} \to  \langle \sigma_3 e_j, e_j \rangle_{\C^2} = v_j \quad \mbox{as $\eps \to 0$} \, ,
$$
whence it follows that $b_{j,\vec{z}} \to v_j$ as $\eps \to 0$. 

In summary, we have shown that the matrix $\mathsf{T} \in \C^{N \times N}$ for $T_{\Ub_{\vec{z}}} : \Hfr_1 \to \Hfr_1$ with respect to the basis $\mathcal{B} = \left ( \frac{e_{1,\vec{z}}}{x-z_1}, \ldots, \frac{e_{N,\vec{z}}}{x-z_N} \right )$ is of the form
$$
\mathsf{T} = \mathrm{diag}(v_1, \ldots, v_N) + \mathsf{M}
$$
with some matrix $\mathsf{M}=\mathsf{M}(z_1, \ldots, z_N, v_1, \ldots, v_N) \in \C^{N \times N}$ such that 
$$
\| \mathsf{M} \|_{\C^N \to \C^N} = O(\eps) \to 0 \quad \mbox{as} \quad \eps \to 0.
$$ 

\medskip
\textbf{Step 3.} Let 
$$
p_{\mathsf{T}}(z)= \det(\mathsf{T}-z \mathds{1}_N) = z^N + a_{N-1} z^{N-1} + \ldots + a_0 
$$ 
denote the characteristic polynomial of $\mathsf{T} \in \C^{N \times N}$. Since 
$$
\lim_{\eps \to 0} \| \mathsf{T} - \mathrm{diag}(v_1, \ldots, v_n) \|_{\C^N \to \C^N} = 0 \, ,
$$
we deduce that $a_{k} \to c_k$ as $\eps \to 0$ for all $k=0, \ldots, N-1$, where
$$
p(z) = z^N + c_{N-1} z^{N-1} + \ldots + c_0 = \prod_{j=1}^N (z-v_j)
$$
is the characteristic polynomial of $\mathrm{diag}(v_1, \ldots, v_N)$. Note that $p(z)$  has simple zeros due to $v_j \neq v_k$ for $j \neq k$ by assumption. Hence the roots $\{ \lambda_j \}_{j=1}^N$ of $\mathsf{T}$ are also simple, provided that $\eps > 0$ is sufficiently small, and we have $\lambda_j \to v_j$ as $\eps \to 0$. Since $v_j^2 \neq v_k^2$ for $j \neq k$ by our assumption above, we also find that $\lambda_j^2 \neq \lambda_k^2$ for $j \neq k$ provided that $\eps > 0$ is sufficiently small. This shows that $T_{\Ub_{\vec{z}}}^2 |_{\Hfr_1}$ has simple spectrum if $\eps > 0$ is sufficiently small and, by self-adjointness of $T_{\Ub_{\vec{z}}}$, this implies simple spectrum of $T_{\Ub_{\vec{z}}} |_{\Hfr_1}$ if $\eps=1/\min_{j \neq k}|z_j-z_k| >0$ is sufficiently small. Since $\sigma_{\mathrm{d}}(T_{\Ub_{\vec{z}}}) = \sigma (T_{\Ub_{\vec{z}}}|_{\Hfr_1})$, this completes the proof of Lemma \ref{lem:T_simple_exist}.
\end{proof}

\begin{remark*}
To conclude our discussion, let us remark that there exist rational data $\uu : \R \to \Ss^2$ with non-simple discrete spectrum $\sigma_{\mathrm{d}}(T_{\Ub})$. For instance, take a solitary wave profile $\qv : \R \to \Ss^2$ given by a Blasche product of degree $m \geq 2$ and set $\Qv_v = \qv_v \sigma \bm{\sigma}$. Then it is easy to see that the Toeplitz operator $T_{\Qv_v} : L^2_+(\R; \C^2) \to L^2_+(\R; \C^2)$ has discrete spectrum $\sigma_{\mathrm{d}}(T_{\Qv_v}) = \{ v \}$ where the eigenvalue $v$ is $m$-fold degenerate. 
\end{remark*}
\section{Local Well-Posedness} 
\label{app:LWP}

In this section, we prove local well-posedness for \eqref{eq:HWMd} for sufficiently regular initial data as stated in Lemma \ref{lem:lwp}. Also, we will show well-posedness for the initial-value problem formulated in \eqref{eq:ode_B} above.

\subsection*{Proof of Lemma \ref{lem:lwp}} 
Let $s  > \frac{3}{2}, d \geq 2$ and assume that $\Ub_0 : \R \to M_d(\C)$ is of the form
$$
\Ub_0(x) = \Ub_\infty + \Vb_0(x) \in M_d(\C) \oplus H^s(\R; M_d(\C)) \equiv H^s_\bullet(\R; M_d(\C))  \, ,
$$
satisfying the pointwise constraints
$$
\Ub_0(x) = \Ub_0(x)^*, \quad \Ub_0(x)^2 = \mathds{1}_d \quad \mbox{for $x \in \R$} \, .
$$
Note that $\Ub_\infty \in M_d(\C)$ is a constant matrix with $\Ub_\infty=\Ub_\infty^*$ and $\Ub_\infty^2 = \mathds{1}_d$. 

Now, for $R> 0$ given and assuming that $\| \Vb_0 \|_{H^s} < R$, we wish to prove existence and uniqueness of the solution
$$
\Ub(t) = \Ub_\infty + \Vb(t) \in  M_d(\C) \oplus C([0,T]; H^s(\R; M_d(\C)) \, ,
$$
of \eqref{eq:HWMd} with initial datum $\Ub(0) = \Ub_0$, where $T=T(R) > 0$ is chosen sufficiently small. Once this solution is constructed, it is elementary to check that $\Ub(t,x)$ satisfies the pointwise constraints above for all $x \in \R$ and times $t \in [0,T]$. Furthermore, as explained before Lemma \ref{lem:lwp} above, we deduce that $\Ub(t,x) \in \Gr_k(\C^d)$ for $(t,x) \in [0,T] \times \R$ with some integer $0 \leq k \leq d$.

\medskip
\textbf{Step 1 (Setup).} To deal with the quasilinear equation \eqref{eq:HWMd}, we use the following iteration scheme. Suppose we are given an initial datum
\be \label{eq:lwp1}
\Ub_0 = \Ub_\infty + \Vb_0 \in M_d(\C) \oplus H^s(\R; H^s(\R; M_d(\C))
\ee
with values in the Hermitian $d \times d$-matrices, i.\,e., we assume
\be \label{eq:lwp2}
\Ub_0(x) = \Ub_0(x)^* \quad \mbox{for $x \in \R$} \, ,
\ee
and with some constant Hermitian matrix $\Ub_\infty = \Ub_\infty^* \in M_d(\C)$. Note that $\Vb_0(x)=\Vb_0(x)^*$ must be Hermitian valued, too.

Now, let $R> 0$ be arbitrary and let $T=T(R) > 0$ to be chosen later. We construct the sequence
$$
\Ub^{(n)}= \Ub_\infty + \Vb^{(n)} \in  M_d(\C) \oplus C([0,T]; H^s(\R; M_d(\C))) \quad \mbox{with $n \in \N$}
$$
by means of the iteration scheme
\be \label{eq:Un_iter}
\pt_t \Ub^{(n+1)} = -\frac{\ii}{2} [\Ub^{(n)}, |D| \Ub^{(n+1)}] \quad \mbox{for $t \in [0,T]$}, \quad \Ub^{(n+1)}(0) = \Ub_0 
\ee
and we take $\Ub^{(0)}(t) \equiv \Ub_0$. It is straightforward to show that, given $\Ub^{(n)} \in M_d(\C) \oplus C([0,T]; H^s(\R; M_d(\C))$, there exists indeed a unique (Hermitian-valued) solution 
$$
\Ub^{(n+1)} = \Ub_\infty + \Vb^{(n+1)} \in C([0,T]; M_d(\C) \oplus H^s(\R; M_d(\C)) 
$$
of \eqref{eq:Un_iter} with initial datum $\Ub^{(n+1)}(0) = \Ub_0$; see Lemma \ref{lem:lwp2} below and its proof for details. Also, since $\Ub = \Ub_\infty + \Vb$ with the constant matrix $\Ub_\infty \in M_d(\C)$, we have
$$
\pt_t \Vb^{(n+1)} = -\frac{\ii}{2} [\Ub^{(n)}, |D| \Vb^{(n+1)}] \quad \mbox{for $t \in [0,T]$}, \quad \Vb^{(n+1)}(0) = \Vb_0 \, .
$$

\medskip
\textbf{Step 2 (Bounds).} We assume that $\| \Vb_0 \|_{H^s} < R$ holds. We claim that the following a-priori bound holds
\be \label{ineq:Vn_apriori}
\sup_{t \in [0,T]} \| \Vb^{(n)}(t) \|_{H^s} \leq 2 \| \Vb_0 \|_{H^s} \quad \mbox{for all $n \in \N$}\, ,
\ee
 provided that $T=T(R) > 0$ is chosen sufficiently small. 

We prove the bound \eqref{ineq:Vn_apriori} as follows. We use $\Dsr^s$ to denote the regularized fractional derivative of order $s$ given by $\widehat{(\Dsr^s f)}(\xi) = (1+|\xi|^2)^{s/2} \widehat{f}(\xi)$. Omitting the dependence on $t$ for notational convenience, we find (where the assumed regularity suffices to justify the following manipulations):
\begin{align*}
\frac{d}{dt} \big \| \Dsr^s \Vb^{(n+1)} \big  \|_{L^2}^2 & = 2 \mathrm{Re} \left \langle \Dsr^s \pt_t \Vb^{(n+1)}, \Dsr^s \Vb^{(n+1)} \right \rangle \\
& = \mathrm{Im} \left \langle \Dsr^s [\Ub^{(n)}, |D| \Vb^{(n+1)}], \Dsr^s \Vb^{(n+1)} \right \rangle \\
& = \mathrm{Im} \left \langle [\Ub^{(n)},  |D| \Dsr^s \Vb^{(n+1)}], \Dsr^s \Vb^{(n+1)} \right \rangle \\
& \quad + \mathrm{Im} \left \langle [\Dsr^s, \Ub^{(n)}] |D| \Vb^{(n+1)}, \Dsr^s \Vb^{(n+1)} \right \rangle  =: I + II \, .
\end{align*}
Here we also used the trivial fact that $|D|$ and $\Dsr^s$ commute. Next, we assert that the term $I$ can be written in exact commutator form with
\be \label{eq:iter1}
I = \frac{1}{2} \, \mathrm{Im} \left \langle \big [ [\Ub^{(n)}, \cdot], |D|] \Dsr^s \Vb^{(n+1)}, \Dsr^s \Vb^{(n+1)} \right \rangle \, .
\ee
Here $[\Ub, \cdot] \mathbf{F} \equiv \Ub \mathbf{F} - \mathbf{F} \Ub$ denotes the pointwise matrix-commutator for matrix-valued functions $\Ub, \mathbf{F} : \R \to M_d(\C)$.  To see that \eqref{eq:iter1} holds true, let us write $\Ub = \Ub^{(n}$ and $\Wb = \Dsr^s \Vb^{(n+1)}$ for the moment. Then
\begin{align*}
I & = \mathrm{Im} \left \langle [\Ub, |D| \Wb], \Wb \right \rangle = \mathrm{Im} \left \langle \big [ [\Ub, \cdot], |D|] \Wb, \Wb \right \rangle + \mathrm{Im} \left \langle |D| [\Ub, \Wb], \Wb \right \rangle \\
& = \mathrm{Im} \left \langle \big [ [\Ub, \cdot], |D|] \Wb, \Wb \right \rangle + \mathrm{Im} \left \langle  \Wb, [\Ub, |D| \Wb] \right \rangle = \mathrm{Im} \left \langle \big [ [\Ub, \cdot], |D|] \Wb, \Wb \right \rangle - I \, ,
\end{align*}
where the second last step we used that $|D| = |D|^*$ is symmetric together with the fact $\Tr([\Ub, \mathbf{A}] \mathbf{B}^*) = \Tr(\mathbf{A} [\Ub, \mathbf{B}]^*)$ for matrix-valued functions $\Ub, \mathbf{A}, \mathbf{B} : \R \to M_d(\C)$ provided that $\Ub= \Ub^*$ is Hermitian. This proves \eqref{eq:iter1}. 

Next, by a classical commutator estimate due to Calder\'on applied to \eqref{eq:iter1} and recalling that $\pt_x \Ub^{(n)} = \pt_x \Vb^{(n)}$, we deduce  
$$
|I|  \leq C \| \pt_x \Vb^{(n)} \|_{L^\infty} \| \Dsr^s \Vb^{(n+1)} \|_{L^2}^2 \leq C \| \Dsr^s  \Vb^{(n)} \|_{L^2}  \| \Dsr^s \Vb^{(n+1)} \|_{L^2}^2 \, ,
$$
where in the last step we used the Sobolev inequality $\| \pt_x f \|_{L^\infty}  \leq C \| \Dsr^{s-1} \pt_x f \|_{L^2} \leq C \| \Dsr^{s} f \|_{L^2}$, since $H^{s-1}(\R) \subset L^\infty(\R)$ thanks to $s > \frac{3}{2}$.

To estimate the second term $II$ above, we use Cauchy--Schwarz and apply the classical Kato--Ponce commutator to $[\Dsr^s, \Ub^{(n)}] = [\Dsr^s, \Vb^{(n)}]$. This yields
\begin{align*}
|II| & \leq \| [\Dsr^s, \Vb^{(n)}] |D| \Vb^{(n+1)} \|_{L^2} \| \Dsr^s \Vb^{(n+1)} \|_{L^2} \\
& \leq C ( \| \Dsr^s \Vb^{(n)} \|_{L^2} \| |D| \Vb^{(n+1)} \|_{L^\infty} + \| \pt_x \Vb^{(n)} \|_{L^\infty} \|  \Dsr^{s-1} |D| \Vb^{(n+1)} \|_{L^2}  ) \| \Dsr^s \Vb^{(n+1)} \|_{L^2} \\
& \leq C \| \Dsr^s \Vb^{(n)} \|_{L^2} \| \Dsr^s \Vb^{(n+1)} \|_{L^2}^2 \, ,
\end{align*}
where in the last step we used again the Sobolev inequalities $\| \pt_x \Vb^{(n)} \|_{L^\infty} \leq C \| \Dsr^s \Vb^{(n)} \|_{L^2}$ and $\| |D| \Vb^{(n+1)} \|_{L^\infty} \leq C \| \Dsr^s \Vb^{(n+1)} \|_{L^2}$ in view of $s > \frac{3}{2}$.

Combing the estimates for $I$ and $II$, we obtain the differential inequality
\be \label{ineq:iter_ode}
 \frac{d}{dt} \big \| \Dsr^s \Vb^{(n+1)}(t) \big \|_{L^2}^2  \leq C \| \Dsr^s \Vb^{(n)}(t) \|_{L^2} \| \Dsr^s \Vb^{(n+1)}(t) \|_{L^2}^2  \, .
\ee
Next we define the quantities
$$
M_n(T) = \sup_{t \in [0,T]} \big \| \Dsr^s \Vb^{(n}(t) \big \|_{L^2}^2 \quad \mbox{with $n \in \N$} \, .
$$
From \eqref{ineq:iter_ode} and Gr\"onwall's inequality we obtain
\be \label{ineq:iter5}
M_{n+1}(T) \leq M_0 \cdot \eu^{C T \sqrt{M_n(T)} }
\ee
since $M_0:=M_{k}(0) = \| \Dsr^s \Vb_0\|_{L^2}^2$ for all $k \in \N$. Clearly, we have the bound
\be \label{ineq:iter6}
M_0 \cdot \eu^{2CT R} \leq 4 M_0
\ee
for some sufficiently small time $T=T(R)> 0$. From $M_0(T) = M_0 < R^2$ and \eqref{ineq:iter5}--\eqref{ineq:iter6}, it follows by induction that 
$$
M_n(T) \leq 4M_0  \quad \mbox{for all $n \in \N$} \, .
$$
Since $M_0 = \| \Dsr^s \Vb_0 \|_{L^2}^2 = \| \Vb_0 \|_{H^s}^2$, we obtain the claimed a-priori bound \eqref{ineq:Vn_apriori}.

\medskip
\textbf{Step 3 (Cauchy Property in $L^2$).} We demonstrate that the sequence $(\Vb^{(n)})_{n \in \N}$ is Cauchy in $C([0,T]; L^2(\R; M_d(\C))$, provided that $T=T(R) > 0$ is small enough. Indeed, let be $n \geq 1$ given. We find
\begin{align*}
\pt_t \left ( \Vb^{(n+1)} - \Vb^{(n)} \right ) & = \frac{1}{2 \ii} \left ( [\Ub^{(n)}, |D| \Vb^{(n+1)}] - [\Ub^{(n-1)}, |D| \Vb^{(n)}] \right ) \\
& = \frac{1}{2 \ii} \left ( [\Ub^{(n)}, |D| (\Vb^{(n+1)}- \Vb^{(n)})] + [ \Vb^{(n)}-\Vb^{(n-1)}, |D| \Vb^{(n)}] \right ) \, ,
\end{align*}
where used the simple fact that $\Ub^{(n)}-\Ub^{(n-1)} = \Vb^{(n)}-\Vb^{(n-1)}$. Hence we get
\begin{align*}
& \frac{d}{dt} \left \| \Vb^{(n+1)}- \Vb^{(n)} \right \|_{L^2}^2  = 2 \mathrm{Re} \left \langle \pt_t(\Vb^{(n+1)}- \Vb^{(n)}), \Vb^{(n+1)}- \Vb^{(n)} \right  \rangle \\
&= \mathrm{Im} \left \langle [\Ub^{(n)}, |D| (\Vb^{(n+1)}-\Vb^{(n)})], \Vb^{(n+1)}-\Vb^{(n)} \right  \rangle \\
& \quad + \mathrm{Im} \left \langle [\Vb^{(n)}-\Vb^{(n-1)}, |D| \Vb^{(n)}], \Vb^{(n+1)}-\Vb^{(n)} \right \rangle \\
& \leq C ( \| \pt_x \Vb^{(n)} \|_{L^\infty} \| \Vb^{(n+1)}- \Vb^{(n)} \|_{L^2}^2 ) \\ 
& \quad + C (\| \Vb^{(n)}-\Vb^{(n-1)} \|_{L^2} \| |D| \Vb^{(n)} \|_{L^\infty} \| \| \Vb^{(n+1)}-\Vb^{(n)} \|_{L^2}) \\
& \leq C( \sqrt{K} ( \| \Vb^{(n+1} - \Vb^{(n)} \|_{L^2} + K \| \Vb^{(n)}-\Vb^{(n-1)}  \|_{L^2} )
\end{align*}
with the constant $K>0$ from the a-priori bound \eqref{ineq:Vn_apriori} above. Since $\Vb^{(n+1)}(0)-\Vb^{(n)}(0) =0$, we learn from Gr\"onwall's inequality that
$$ 
\sup_{t \in [0,T]} \| \Vb^{(n+1)}(t)- \Vb^{(n)}(t) \|_{L^2} \leq C T \sqrt{K}  \sup_{t \in [0,T]} \| \Vb^{(n)}(t) - \Vb^{(n-1)}(t) \|_{L^2} \, .
$$
By choosing $T = T(R) > 0$ even smaller to ensure that $CT \sqrt{K} \leq \frac{1}{2}$, we deduce that the series 
$$
\sum_{n=0}^\infty \sup_{t \in [0,T]} \| \Vb^{(n+1)}(t) - \Vb^{(n)}(t) \|_{L^2} < +\infty
$$ 
is geometrically convergent. In particular, the implies that the sequence $(\Vb^{(n)})_{n \in \N}$ is Cauchy in $C([0,T]; L^2(\R; M_d(\C))$.

Thanks to the a-priori bound \eqref{ineq:Vn_apriori}, this yields that $(\Vb^{(n)})_{n \in \N}$ forms a Cauchy sequence in $C([0,T]; H^{\tilde{s}}(\R; M_d(\C))$ for $0 \leq \tilde{s} <s$. Moreover, we readily check that its limit
$$
\Ub := \Ub_\infty + \lim_{n \to \infty} \Vb^{(n)} \in M_d(\C) \oplus C([0,T]; H^{\tilde{s}}(\R; M_d(\C))
$$
solves \eqref{eq:HWMd} with initial datum $\Ub(0)=\Ub_0$. 

\medskip
{\bf Step 4 (Continuity of Flow in $H^s$).} It remains to show that 
$$
\Vb \in C([0,T]; H^s(\R; M_d(\C)) \, .
$$
Note that, by previous discussion, we can only deduce that $\Vb \in C_{w}([0,T];  H^s(\R; M_d(\C))$ holds, i.\,e., for $t_n \to t$ we only have that $\Vb(t_n) \weakto \Vb(t)$ in $H^s$. To extend this to strong continuity, we can make use the idea of frequency envelopes, which was recently generalized aa an abstract interpolation result in \cite{AlBuIfTaZu-24}. 

Indeed, for real $t \geq 0$, we introduce the Sobolev spaces $H^s_{\mathrm{H}}$ of matrix-valued maps with Hermitian values by setting
$$
H^t_{\mathrm{H}} := \{ \Fb \in H^t(\R; M_d(\C)) \mid \mbox{$\Fb(x) = \Fb(x)^*$ for a.\,e.~$x \in \R$} \} \, ,
$$
equipped with the norm $\| \cdot \|_{H^t}$. Let $B_R = \{ \Fb \in H^s_{\mathrm{H}} \mid \| \Fb \|_{H^s} < R \}$. From \textbf{Step 2} and \textbf{Step 3}, we obtain the map
$$
\Phi : B_R \to C([0,T]; H^{0}_{\mathrm{H}}), \quad \Wb_0 \mapsto \Wb := \lim_{n \to \infty} \Wb^{(n)} 
$$
using the iteration scheme with initial data $\Ub_0 = \Ub_\infty + \Wb_0$. Moreover, from the previous discussion, we deduce the following bounds 
\begin{enumerate}
\item[$(B_1)$]  $\| \Phi(\Wb_0) - \Phi(\widetilde{\Wb}_0) \|{C_T H^0} \leq C_0 \| \Wb_0 - \widetilde{\Wb}_0 \|_{H^0}$ for all $\Wb_0, \widetilde{\Wb}_0 \in B_R$,
\item[$(B_2)$] $\| \Phi(\Wb_0) \|_{C_T H^{s+1}} \leq 2 \| \Wb_0 \|_{H^{s+1}}$ for all $\Wb_0 \in B_R \cap H^{s+1}_{\mathrm{H}}$,
\end{enumerate}
with some constant $C_0 > 0$. Indeed, the weak Lipschitz estimate $(B_1)$ follows from the arguments in \textbf{Step 3}, where as the bound $(B_2)$  simply follows from repeating \textbf{Step 2} with $s>\frac{3}{2}$ replaced by $s+1$ and by choosing $T=T(R) > 0$ possibly even smaller. From \cite{AlBuIfTaZu-24} we now conclude that 
$$
\Phi(\Vb_0) \in C([0,T]; H^s_{\mathrm{H}})
$$
and that we have continuous dependence of the map $\Vb_0 \mapsto \Phi(\Vb_0)$ on the initial data in $B_R$.

\medskip
\textbf{Step 5 (Conclusion).} Thus far we have proved local-in-time existence of solutions for \eqref{eq:HWMd} for initial data in $H^s$ with $s > \frac{3}{2}$ and satisfying the Hermitian condition \eqref{eq:lwp2}. Moreover, by a direct calculation and using the regularity of the solutions, we readily check by a Gr\"onwall-type argument that uniqueness holds for $C([0,T]; H^s)$ for a given initial datum $\Ub(0) = \Ub_0$. 

Also, a direct calculation (which we omit) shows that the pointwise constraint $\Ub_0(x)^2 = \mathds{1}_d$ is also preserved by the flow.

Finally, the claimed propagation of higher Sobolev regularity also follows from the previous estimates. Indeed, let $\sigma > s > \frac{3}{2}$ and suppose that $\Vb_0 \in H^\sigma_{\mathrm{H}}$. Inspecting the arguments in \textbf{Step 2}, we deduce that
\begin{align*}
\| \Dsr^\sigma \Vb(t) \|_{L^2}^2 & \leq C \left ( \| \pt_x \Vb(t)\|_{L^\infty} + \| |D| \Vb(t) \|_{L^\infty} \right ) \| \Dsr^\sigma \Vb(t) \|_{L^2}^2 \\
& \leq C \| \Vb(t) \|_{H^{s}} \|\Dsr^\sigma \Vb(t) \|_{L^2}^2 \, ,
\end{align*}
where we used the Sobolev embedding $H^{s}(\R) \subset L^\infty(\R)$ for $s > \frac{3}{2}$. By Gr\"onwall's inequality, we readily deduce that the maximal times of existence of $H^\sigma$ and $H^s$-solutions with $\sigma > s > \frac{3}{2}$ coincide. 

This completes the proof of Lemma \ref{lem:lwp}. \hfill $\qed$

\medskip
In the proof above, we need the following auxiliary result.

\begin{lem} \label{lem:lwp2}
Let $s > \frac{3}{2}, d \geq 2$, and $\Ub = \Ub_\infty + \Vb \in C([0,T]; M_d(\C) \oplus H^s(\R; M_d(\C)))$. Then, for every $\widetilde{\Vb}_0 \in H^s(\R; M_d(\C))$, there exists a unique solution $\widetilde{\Ub} = \Ub_\infty + \widetilde{\Vb} \in M_d(\C) \oplus C([0,T]; H^s(\R; M_d(\C)))$ of
$$
\pt_t \widetilde{\Ub} = -\frac{\ii}{2} [ \Ub, |D| \widetilde{\Ub}] \quad \mbox{on $[0,T]$ and $\widetilde{\Ub}(0) = \Ub_\infty + \widetilde{\Vb}_0$} \, .
$$
Moreover, if $\widetilde{\Ub}(0,x)=\widetilde{\Ub}(0,x)^*$ and $\widetilde{\Ub}(0,x)^2=\mathds{1}_d$ for all $x \in \R$, then $\widetilde{\Ub}(t,x)= \widetilde{\Ub}(t,x)$ and $\widetilde{\Ub}(t,x) = \widetilde{\Ub}(t,x)^*$ for all $(t,x) \in [0,T] \times \R$. 
\end{lem}

\begin{proof}
Since $\Ub_\infty \in M_d(\C)$ is constant matrix, it suffices to show existence and uniqueness of $\widetilde{\Vb} \in C([0,T]; H^s(\R; M_d(\C)))$ solving
\be
\pt_t \widetilde{\Vb} = -\frac{\ii}{2} [\Ub, |D| \widetilde{\Vb}] \quad \mbox{with $\widetilde{\Vb} = \widetilde{\Vb}_0$} \, .
\ee
We construct approximate solutions of this linear equation by the following scheme. For $\eps > 0$, we introduce the smoothing operator
$$
J_\eps :=  (1+\eps |D|)^{-1} \quad \mbox{with $\|J_\eps \|_{L^2 \to L^2} \leq 1$ and $\| J_\eps \|_{H^s \to H^{s+1}} \leq \eps^{-1}$} \, .
$$
By standard arguments, we obtain a unique solution $\widetilde{\Vb}_\eps \in C([0,T]; H^s(\R; M_d(\C)))$ of the initial-value problem
$$
\pt_t \widetilde{\Vb}_\eps = -\frac{\ii}{2} [\Ub, |D| J_\eps \widetilde{\Vb}_\eps] \quad \mbox{with $\widetilde{\Vb}_\eps(0) = \widetilde{\Vb}_0$} \, ,
$$
using that $|D| J_\eps : H^s \to H^s$ is a bounded map together with the fact that $H^s(\R)$ is an algebra for $s > \frac{3}{2}$. Now, by adapting the discussion from the previous discussion, we derive the estimate
$$
\frac{d}{dt} \| \Dsr^s \widetilde{\Vb}(t) \|_{L^2}^2 \leq  C \left ( \left \| [\Ub, \cdot], |D|J_\eps ] \right \|_{L^2 \to L^2} + \| \Dsr^s \Vb\|_{L^2} \right ) \| \Dsr^s \widetilde{\Vb}_\eps \|_{L^2}^2
$$
using also again the Kato--Ponce estimate together with the fact that $\|\Dsr^{s-1} |D| J_\eps \widetilde{\Vb} \|_{L^2} \leq \| \Dsr^s \widetilde{\Vb} \|_{L^2}$. To bound the commutator term, we note that if $a=a(x)$ denotes multiplication by a Lipschitz function then
\begin{align*}
[a, |D| J_\eps] & = |D| [a, (1+\eps |D|)^{-1}] + [a,|D|] (1+\eps |D|)^{-1} \\
& = -\eps |D| (1+\eps |D|)^{-1} [a,|D|] (1+\eps |D|)^{-1} + [a,|D|] (1+\eps |D|)^{-1} \, .
\end{align*}
Thus by Calder\'on's commutator estimate and the facts $\| \eps |D|(1+\eps |D|)^{-1} \|_{L^2 \to L^2} \leq 1$ and $\| (1+\eps |D|)^{-1} \|_{L^2 \to L^2} \leq 1$, we deduce
$$
\left \| [\Ub, \cdot], |D|J_\eps ] \right \|_{L^2 \to L^2} \leq C \| \pt_x \Vb \|_{L^\infty} \leq C \| \Dsr^s \Vb \|_{L^2} 
$$
since $s > \frac 3/2$. Because of $\sup_{t \in [0,T]} \| \Dsr^s \Vb(t) \|_{L^2} < +\infty$, integrating the previous differential inequality yields the bound
\be \label{ineq:lwp2_1}
\sup_{t \in [0,T]} \| \Dsr^s \widetilde{\Vb}_\eps(t) \|_{L^2} \leq  \eu^{CT}  \| \Dsr^s \widetilde{\Vb}_0 \|_{L^2}
\ee
which is independent of $\eps >0$. Moreover, this bound and the equation for $\widetilde{\Vb}_\eps$ imply that
$$
\| \pt_t \widetilde{\Vb}_\eps(t) \|_{L^2} \leq C \| \Ub(t) \|_{L^\infty} \| |D| J_\eps \widetilde{\Vb}_\eps \|_{L^2} \leq C \left ( \| \Ub_\infty \|_{L^\infty} + \| \Dsr^s \Vb(t) \|_{L^2} \right ) \| \Dsr^s \widetilde{\Vb}_\eps(t) \|_{L^2} \, .
$$
Hence it follows that
$$
\sup_{t \in [0,T]} \| \pt_t \widetilde{\Vb}_\eps(t) \|_{L^2} \leq C \| \Dsr^s \widetilde{\Vb}_0 \|_{L^2}
$$
independent of $\eps > 0$. Thus, by standard compactness arguments (see, e.\,g.,~\cite{Ca-03}[Proposition 1.1.2]), we deduce that $(\widetilde{\Vb}_{\eps_n})$ converges for some sequence $\eps_n \to 0$ to some limit $\widetilde{\Vb} \in C_w([0,T]; H^s(\R; M_d(\C))$ solving 
$$
\pt_t \widetilde{\Vb} = -\frac{\ii}{2} [ \Ub, |D| \widetilde{\Vb} ] \quad \mbox{with $\widetilde{\Vb}(0) = \widetilde{\Vb}_0$} \, .
$$
By mimicking the arguments in the previous proof using the abstract interpolation result, we actually deduce strong continuity, i.\,e., we have $\widetilde{\Vb} \in C([0,T]; H^s(\R; M_d(\C))$. Uniqueness of  the solution follows from a simple Gr\"onwall argument in the same fashion when deriving \eqref{ineq:lwp2_1}.

Finally, we remark that the conversation of the pointwise constraints follows by a direct calculation, which we omit. This completes the proof of Lemma \ref{lem:lwp2}.
\end{proof}

\medskip
We conclude this section by showing existence and uniqueness for the operator-valued initial-value problem \eqref{eq:ode_B} that appears in the discussion of the Lax structure which reads
\be \label{eq:ode_lax_app}
\pt_t \UU(t) = B_{\Ub(t)}^+ \UU(t) \quad \mbox{for $t \in [0,T]$}, \quad \UU(0) = \id \, .
\ee 
for the operator-valued map $\UU : [0,T] \to \mathcal{B}(L^2_+(\R; \VV))$. As usual, we use $\mathcal{B}(H)$ to denote the Banach space of bounded linear maps $H \to H$ with a given Hilbert space $H$. Recall that 
\be
B_{\Ub}^+ = \frac{\ii}{2} \left ( T_\Ub \circ D + D \circ T_\Ub \right )  - \frac{\ii}{2} T_{|D| \Ub}
\ee
with $D = -\ii \pt_x$ denotes the compression of $B_\Ub$ on the Hardy space $L^2_+(\R; \VV)$. Recall that for solutions $\Ub \in M_d(\C) \oplus H^s(\R; M_d(\C))$ with $s > \frac{3}{2}$ as given by Lemma \ref{lem:lwp}, the operators $\{ B^+_{\Ub(t)} \}_{t \in [0,T]}$ are a family of (essentially) skew-adjoint operators on $L^2_+(\R; \VV)$ with operator domain $H^1_+(\R;\VV) = L^2_+(\R; \VV) \cap H^1(\R; \VV)$; see also the remark below. Recall that we either take $\VV= \C^d$ or $\VV = M_d(\C)$ equipped with their natural scalar products.

\begin{lem}  \label{lem:lax_flow}
Let $s > \frac{3}{2}$ and $d \geq 2$. Assume $\Ub = \Ub_\infty + \Vb \in M_d(\C) \oplus C([0,T]; H^s(\R; M_d(\C))$ is a solution given by Lemma \ref{lem:lwp}. Then there exists a unique solution $\UU : [0,T] \to \mathcal{B}(L^2_+(\R; \VV))$ of \eqref{eq:ode_lax_app} with the following properties.
\begin{itemize}
 \item[(i)] The map $[0,T] \to L^2_+(\R; \VV)$ with $t \mapsto \UU(t) \phi$ is continuous for every $\phi \in L^2(\R; \VV)$. 
 \item[(ii)] The equation $\pt_t \UU(t) = B_{\Ub(t)}^+ \UU(t)$ holds in $H^{-1}_+(\R;\VV)$ for any $t \in [0,T]$.
 \item[(iii)] $\UU(t) : L^2_+(\R; \VV) \to L^2_+(\R; \VV)$ is unitary for all $t \in [0,T]$.
 \item[(iv)] For $\phi \in H^1_+(\R; \VV) \cap \dom(X^*)$, we have $\UU(t) \phi \in H^1_+(\R;\VV) \cap \dom(X^*)$ for $t \in [0,T]$.
 \end{itemize}
\end{lem}

\begin{remark*}
In particular, the proof below shows that, given a time-dependent $\Ub = \Ub_\infty + \Vb \in M_d(\C) \oplus H^s(\R; M_d(\C))$ with some $s > 3/2$ and satisfying $\Ub(x) = \Ub(x)^*$ for all $x \in \R$, the operator $B_\Ub^+ : H^1_+(\R; \VV) \subset L^2_+(\R; \VV) \to L^2_+(\R; \VV)$ is essentially skew-adjoint, i.\,e., there exists a unique extension with $(B_{\Ub}^+)^* = -B_\Ub^*$, since it is found to be the generator of a strongly continuous one-parameter unitary group on $L^2_+(\R; \VV)$. 
\end{remark*}

\begin{proof}
For notational convenience, we shall write $L^2_+, H^1_+$ and $H^{-1}_+$ for $L^2_+(\R; \VV)$, $H^1_+(\R; \VV)$ and $H^{-1}_+(\R; \VV)$, respectively.

\medskip
\textbf{Step 1.} We first show that, for every $F_0 \in L^2_+$, the initial-value problem
\be \label{eq:B_eq}
\pt_t F = B_{\Ub}^+ F, \quad F(0) = F_0
\ee
has a unique solution $F \in C([0,T]; L^2_+(\R; \VV))$ and we have $\| F(t) \|_{L^2} = \| F_0 \|_{L^2}$ for $t \in [0,T]$.

For $\eps > 0$, we introduce the smoothing operators
$$
J_\eps := (1+\eps D)^{-1} : L^2_+ \to H^1_+ \quad \mbox{with $\| J_\eps \|_{L^2 \to L^2} \leq 1$ and $\| J_\eps \|_{L^2 \to H^1} \leq \eps^{-1}$} \, .
$$
Consider the approximate initial-value problem
\be
\pt_t F_\eps = J_\eps B_{\Ub}^+ J_\eps F_\eps, \quad F_\eps(0) = F_0 \, ,
\ee
which has a unique solution $F_\eps \in C^1([0,T]; L^2_+)$ by standard arguments. Since $J_\eps B_{\Ub(t)} J_\eps$ is a bounded skew-adjoint operator for every $t \in [0,T]$, we readily find
$$
\| F_\eps(t) \|_{L^2} = \| F_0 \|_{L^2} \, .
$$
By the equation, this implies that $\pt_t F_\eps \in C([0,T]; H^{-1}_+)$ uniformly in $\eps>0$. Hence the family $\{ F_\eps \}_{\eps > 0}$ is uniformly equicontinuous in $C([0,T]; H^{-1}_+)$ and uniformly bounded in $C([0,T]; L^2_+)$. By a standard compactness argument (see, e.\,g.~\cite{Ca-03}[Proposition 1.1.2]), we can find a suitable sequence $\eps_n \to 0$ with the limit $F := \lim_{n \to \infty} F_{\eps_n} \in C([0,T]; H^{-1}_+) \cap C_w([0,T]; L^2_+)$ which satisfies
\be
\pt_t F = B_{\Ub} F, \quad F(0) = F_0 \, .
\ee
We now claim that 
\be \label{eq:L2_cons}
\| F(t) \|_{L^2} = \| F_0 \|_{L^2} \quad \mbox{for $t \in [0,T]$} \, .
\ee
Indeed, we calculate
\be \label{ineq:L2_cons}
\frac{d}{dt} \langle J_\eps F(t), F(t) \rangle = 2 \mathrm{Re} \left \langle [J_\eps, B_{\Ub(t)}^+] F(t), F(t) \right \rangle \, .
\ee
Using that $[J_\eps, D] = 0$, $[J_\eps, AB] = A[J_\eps, B] + [J_\eps, A]B$ and $[J_\eps,A] = -J_\eps[\eps D, A]J_\eps$, we find
$$
[J_\eps, B_\Ub] = -\frac{\ii}{2} \left ( J_\eps [\eps D, T_{\Ub}] D J_\eps + D J_\eps [\eps D, T_\Ub] J_\eps \right ) - \frac{\ii}{2} [ J_\eps, T_{|D| \Ub}]  =: I_\eps + II_\eps \, .
$$
Next, we claim that
\be \label{eq:friedrichs}
I_\eps \phi \to 0 \quad \mbox{for every $\phi \in L^2_+$ as $\eps \to 0$} \, .
\ee
By Leibniz' formula, we find
$$
\| I_\eps \|_{L^2 \to L^2} \leq C \eps \| \pt_x \Ub \|_{L^2} \| J_\eps \|_{L^2 \to L^2} \| D J_\eps \|_{L^2 \to L^2} \leq C \eps \eps^{-1} = C
$$
independent of $\eps > 0$. Furthermore, it is easy to dominated convergence (and taking adjoints) that $I_\eps \phi \to 0$ in $L^2_+$ as $\eps \to 0$ for every $\phi \in H^{1}_+$. By density of $H^1_+ \subset L^2_+$ and the uniform bound $\| I_\eps \|_{L^2 \to L^2} \leq C$, we readily deduce that \eqref{eq:friedrichs} holds. Next, we observe that
$$
\|II_\eps \|_{L^2 \to L^2} \leq C \| J_\eps \|_{L^2} \| |D| \Ub \|_{L^\infty} \leq C
$$
independent of $\eps > 0$. Also, by dominated convergence (and taking adjoints) we see that $II_\eps \phi \to 0$ in $L^2_+$ as $\eps \to 0$ for every $\phi \in H^1_+$. Again, we conclude
\be \label{eq:friedrichs2}
II_\eps \phi \to 0 \quad \mbox{for every $\phi \in L^2_+$ as $\eps \to 0$} \, .
\ee
Going back to \eqref{ineq:L2_cons} and using \eqref{eq:friedrichs} and \eqref{eq:friedrichs2}, we find by integration
$$
\langle J_\eps F(t), F(t) \rangle = \langle J_\eps F_0, F_0 \rangle + \int_0^t g_\eps(\tau) \, d\tau
$$
with $g_\eps(t) \to 0$ in $L^2_+$ as $\eps \to 0$ for every $t \in [0,T]$. Since also $\| g_\eps(t) \|_{L^2} \leq C$, we can use dominated convergence when passing to the limit $\eps \to 0$ to find
$$
\langle F(t), F(t) \rangle = \langle F_0, F_0 \rangle \quad \mbox{for all $t \in [0,T]$} \, ,
$$
which is the desired identity \eqref{eq:L2_cons}. Finally, we also remark that conservation of the $L^2$-norm implies that the strong continuity $F \in C([0,T]; L^2_+)$. Uniqueness of solutions in this class for the linear equation $\pt_t F = B_{\Ub} F$ directly follows from $L^2$-conservation as well.

\medskip
\textbf{Step 2.} We define the map $\UU : [0,T] \to \mathcal{B}(L^2_+)$ by setting $\UU(t) F_0 := F(t)$ for $F_0 \in L^2_+$, where $F \in C([0,T]; L^2_+)$ is the unique solution of $\pt_t F = B_{\Ub}^+ F$ with $F(0)=F_0$. By $L^2$-conservation, we see that $\| \UU(t) F_0 \|_{L^2} = \| F_0 \|_{L^2}$ and hence $\UU(t)$ is an isometry on $L^2_+$ for any $t \in [0,T]$. Furthermore, by a time reversal argument for the Schr\"odinger-type equation \eqref{eq:B_eq}, we see that $\UU(t)$ is also surjective on $L^2_+$. Thus $\UU(t)$ is a unitary map on $L^2_+$ for any $t \in [0,T]$. This proves (iii), whereas the items (i) and (ii) are directly verified.

\medskip
\textbf{Step 3.} It remains to show property (iv). For $\phi \in H^1_+(\R;\VV) \cap \dom(X^*)$, we can show, by using an approximation argument (whose details we omit) with the family of operators $J_\eps=(1+\eps D)^{-1}$  and $R_\eps=(\eps X^* - \ii)^{-1}$ with $\eps > 0$, that the solution $F \in C([0,T]; L^2_+)$ of $\pt_t F = B_\Ub F$ with $F(0)=\phi$ satisfies $F(t) \in H^1_+(\R; \VV) \cap \dom(X^*)$ for $t \in [0,T]$. Since $F(t)=\UU(t) \phi$ this shows that (iv) holds true.

This completes the proof of Lemma \ref{lem:lax_flow}.
\end{proof}

\end{appendix}

\bibliographystyle{siam}
\bibliography{HWM_Bibliography}

\end{document}